\documentclass[12pt,reqno]{
amsart}

\usepackage[margin=1in]{geometry}
\usepackage[tt=false]{libertine}

\usepackage{amssymb}

\usepackage[varbb]{newpxmath}
\let\savedbigtimes\bigtimes
\let\bigtimes\relax
\usepackage{mathabx} 
\let\bigtimes\savedbigtimes

\usepackage{graphicx}
\usepackage{enumerate}
\usepackage{bm}

\usepackage{bbm}

\usepackage{verbatim}
\usepackage{hyperref,color}
\usepackage[capitalize,nameinlink]{cleveref}
\usepackage[dvipsnames]{xcolor}
\hypersetup{
	colorlinks=true,
	pdfpagemode=UseNone,
    citecolor=OliveGreen,
    linkcolor=NavyBlue,
    urlcolor=black,
	pdfstartview=FitW
}
\usepackage{appendix}
\crefname{appsec}{Appendix}{Appendices}
\usepackage{tikz}

\theoremstyle{plain}
\newtheorem{theorem}{Theorem}[section]
\newtheorem{proposition}[theorem]{Proposition}
\newtheorem{lemma}[theorem]{Lemma}
\newtheorem{corollary}[theorem]{Corollary}

\newtheorem{fact}[theorem]{Fact}

\theoremstyle{definition}
\newtheorem{definition}[theorem]{Definition}

\newtheorem*{assumption*}{Assumption}

\theoremstyle{remark}
\newtheorem{remark}[theorem]{Remark}

\crefname{lemma}{Lemma}{Lemmas}
\crefname{theorem}{Theorem}{Theorems}
\crefname{definition}{Definition}{Definitions}
\crefname{fact}{Fact}{Facts}
\crefname{claim}{Claim}{Claims}
\crefname{proposition}{Proposition}{Propositions}

\newcommand{\E}{\mathbb{E}}
\newcommand{\Var}{\mathrm{Var}}

\newcommand{\one}{\mathbb{1}}
\newcommand{\allone}{\mathbf{1}}

\makeatletter
\newcommand{\vast}{\bBigg@{4}}
\newcommand{\Vast}{\bBigg@{5}}
\makeatother

\newcommand{\eps}{\varepsilon}
\renewcommand{\epsilon}{\varepsilon}

\newcommand{\N}{\mathbb{N}}
\newcommand{\Q}{\mathbb{Q}}
\newcommand{\R}{\mathbb{R}}

\newcommand{\GG}{\mathcal{G}}
\newcommand{\LL}{\mathcal{L}}
\newcommand{\MM}{\mathcal{M}}

\newcommand{\QQ}{\mathbb{Q}}
\newcommand{\PP}{\mathbb{P}}
\newcommand{\EE}{\mathbb{E}}
\newcommand{\NN}{\mathbb{N}}
\newcommand{\ZZ}{\mathbb{Z}}

\newcommand{\beq}{\begin{equation}}
\newcommand{\eeq}{\end{equation}}

\newcommand{\bH}{\mathbf{H}}

\renewcommand{\emptyset}{\varnothing}

\begin{document}

\title[Stable algorithms Lower Bounds for Estimation]{Stable algorithms Lower Bounds for Estimation\\ from MMSE Discontinuities}

\author
[X. Yu, I. Zadik]{Xifan Yu$^\dagger$ and Ilias Zadik$^{\circ}$}

\thanks{\raggedright$^\circ$Department of Statistics and Data Science, Yale University;
$^\dagger$Department of Computer Science, Yale University.\\
Email: \texttt{\{xifan.yu, ilias.zadik\}@yale.edu}}

\date{\today}

\maketitle

\begin{abstract}

Recent works in average-case complexity have identified stable (noise-stable) algorithms as a central class. Specifically, in average-case optimization, the class is conjectured to capture the power of polynomial-time computation for many problems. This perspective has been supported by establishing variants of the Overlap Gap Property (OGP) phase transitions just below conjectured polynomial-time thresholds. Yet, it was recently challenged by Schramm and Li (2025), who showed that Shortest Path in random graphs exhibits the OGP—and hence all stable algorithms fail—despite being solvable in polynomial time. This counterexample has also been particularly curious as it appeared rather distinct from other classical ``noiseless" counterexamples, such as solving random linear systems.

By contrast, the power of stable methods in statistical estimation has remained unclear. A central difficulty is the absence of a definition of an OGP-type phenomenon that can uniformly exclude all stable methods. Instead, existing lower bounds largely focus on the related class of low-degree polynomials and are confined to restricted models, such as Gaussian additive models, reflecting the high technical difficulty of controlling the minimum mean-squared error (MMSE) of low-degree estimators.

In this work, we show that for all statistical estimation problems, a natural MMSE instability (discontinuity) condition implies the failure of stable algorithms, serving as a version of OGP for estimation tasks. Using this criterion, we establish separations between stable and polynomial-time algorithms for the following MMSE-unstable tasks (i) Planted Shortest Path, where Dijkstra’s algorithm succeeds, (ii) random Parity Codes, where Gaussian elimination succeeds, and (iii) Gaussian Subset Sum, where lattice-based methods succeed. For all three, we further show that all low-degree polynomials are stable, yielding separations against low-degree methods and a new method to bound the low-degree MMSE. In particular, our technique highlights that MMSE instability is a common feature for Shortest Path and the noiseless Parity Codes and Gaussian subset sum.

Last, we highlight that our work places rigorous algorithmic footing on the long-standing physics belief that first-order phase transitions—which in this setting translates to MMSE instability—impose fundamental limits on classes of efficient algorithms.
    
\end{abstract}

\newpage
\tableofcontents

\newpage
\section{Introduction}

\allowdisplaybreaks{

Understanding the computational complexity of average-case problems has been a central task of interest over the last few decades, departing from the classical worst-case complexity theory. A dominant approach in this area is to analyze restricted but powerful families of algorithms for a given average-case problem, and to interpret the point at which these algorithms fail as a proxy for the onset of computational hardness. This methodology has been most influential in two broad settings: average-case optimization problems under a “null” distribution, and statistical inference problems under a “planted” distribution.

For average-case optimization, the most extensively studied algorithmic class is that of (noise) stable algorithms. This class often encompasses low-degree polynomials, low-depth circuits, and local-search procedures run for bounded time (see e.g., \cite{gamarnik2024hardness}). Despite its apparent restrictiveness in the vast class of $\mathcal{P}$, a striking pattern has emerged in recent years: for many canonical random problems, stable algorithms match the performance of the best known polynomial-time methods. Prominent examples include optimization in $p$-spin glass models \cite{gamarnik2024hardness,huang2025tight}, maximum independent set in random graphs \cite{wein2022optimal}, and random $k$-SAT \cite{bresler2022algorithmic}. It is worth pointing out that from a technical standpoint, the key technical tools underpinning these tight stable algorithm lower bounds are variants of a disconnectivity landscape property called Overlap Gap Property (OGP), first suggested in \cite{gamarnik2014limits} (see also \cite{gamarnik2021overlap} for a survey). Now, these successes have motivated the conjecture that the failure of stable algorithms delineates the true computational hardness threshold of a problem.

However, it has long been understood—at least at a folklore level—that such a conjecture cannot hold in full generality. The canonical family of such counterexamples are variants of XORSAT, where the task is to find a solution to system of linear random equations. While Gaussian elimination succeeds whenever a solution exists, it is widely believed that stable algorithms fail in certain satisfiable regimes.\footnote{While this belief is widespread, we are not aware of a formal proof in the literature. Related results establish shattering or clustering phenomena in $k$-XORSAT, which often signal—but do not by themselves imply— failure of stable methods \cite{ayre2020satisfiability, ibrahimi2012set, achlioptas2015solution}.} Yet, this family of counterexamples is often dismissed as an artifact of the problem’s noiseless structure. Remarkably, the recent work \cite{li2024some} showed that even a seemingly noisy problem can violate the stable-algorithm optimality paradigm. They proved that the Shortest Path problem on Erdős–Rényi random graphs, while solvable in polynomial time by Dijkstra’s algorithm, exhibits the OGP, implying provable failure of stable algorithms. This result significantly blurs the boundary of problems for which stable algorithms should be conjectured to be optimal for random tasks, and understanding this question is one of the main motivations of this work.

By contrast, in statistical inference—specifically detection and estimation tasks—the study of stable algorithms as a unified class remains largely undeveloped. A primary reason is the absence of an OGP-type phenomenon capable of yielding general lower bounds against all stable methods\footnote{There is a successful OGP notion for planted problems, but it is only able to produce tight MCMC/local-search lower bounds (see e.g., \cite{gamarnik2022sparse, gamarnik2024landscape}).}. Instead, research has largely focused on alternative restricted classes, most notably low-degree polynomials (LDPs)\footnote{It should be noted that planted models, the question of whether LDPs are in fact a subset of stable methods is rather non-trivial.}, see e.g., \cite{wein2025computational} for a recent survey.

For detection problems, the power of LDPs has been extensively investigated, leading to the well-known low-degree conjecture \cite{hopkins2018statistical}, which posits that low-degree polynomials match the power of all polynomial-time algorithms for suitably “nice” detection tasks. The exact notion of “niceness” is an ongoing debate in the community and typically encodes symmetry assumptions (see \cite{hsieh2026rigorous} for a proof under full symmetry, and \cite{buhai2025quasi} for counterexamples of a stronger variant under weaker symmetry) as well as the presence of sufficient noise. It is perhaps beneficial to point out that two of the main classes of counterexamples for low-degree optimality in the absence of noise are (a) planted random linear equations where Gaussian elimination succeeds, similar to the optimization literature above, but also (b) versions of the Gaussian subset sum problem, where lattice-based methods succeed in polynomial time \cite{frieze1986lagarias,zadik2018high}. Although challenged in recent works, the low-degree conjecture remains widely believed for many canonical detection problems.

In contrast, the power of LDPs in estimation, while expected to be similar to detection, is far less mathematically understood. From a technical standpoint, controlling the minimum mean-squared error (MMSE) of low-degree estimators—the low-degree MMSE—is substantially more difficult than analyzing the low-degree likelihood ratio in detection. While a recent work \cite{schramm2022computational} has developed important techniques for bounding the low-degree MMSE in Gaussian additive models and certain Bernoulli noise settings, applying these methods rely on delicate arguments and their successful applications remain restricted to a few model classes. As a consequence, no formal low-degree conjecture for estimation has been formally posed, despite the fact that researchers in the literature do predict hardness based on low-degree MMSE lower bounds (see e.g., \cite{even2025computational}). For this reason, we believe that clarifying the power of low-degree polynomials compared to other polynomial-time algorithms for estimation tasks is an important direction and one of the main motivations of this work as well.

\subsection{Contributions}\label{sec:contr}

In this work, we focus on \emph{statistical estimation} problems, where an instance $(x,y)\sim\PP_{XY}$ is drawn and the goal is to design an algorithm $\mathcal{A}$ that, given the observation $y$, approximates the signal $x$ in mean-squared error (MSE). Specifically, we aim to minimize
\[
\E_{(x,y)\sim\PP_{XY}}\big[\|\mathcal{A}(y)-x\|_2^2\big],
\]
ideally matching the \emph{minimum mean-squared error} (MMSE),
\[
\mathrm{MMSE}
\;=\;
\inf_{\mathcal{A}:\R^N\to\R^n}
\E_{(x,y)\sim\PP_{XY}}\big[\|\mathcal{A}(y)-x\|_2^2\big].
\]

Let $T_\rho:\R^N\to\R^N$, $\rho\in[0,1]$, denote a noise operator acting on the observation $y$. Of central importance to this work is the following class of algorithms.

\begin{definition}[Stable algorithm]
An algorithm $\mathcal{A}:\R^N\to\R^n$ is $(\rho,\eta)$-stable if
\[
\E_{(x,y)\sim\PP_{XY}}\big[\|\mathcal{A}(y)-\mathcal{A}(T_\rho(y))\|_2^2\big]
\;\le\;
\eta\cdot
\E_{(x,y)\sim\PP_{XY}}\big[\|\mathcal{A}(y)\|_2^2\big].
\]
\end{definition}

\textbf{MMSE instability implies stable-algorithm lower bounds.}
Our first contribution is to uncover a general  principle: \begin{quote}
\centering
\emph{Noise-instability of the MMSE yields lower bounds against all stable algorithms}. 
\end{quote}To describe this further, define the \emph{noisy MMSE},
\begin{align}\label{eq:MMSE_noisy}
\mathrm{MMSE}_\rho
\;:=\;
\inf_{\mathcal{A}:\R^N\to\R^n}
\E_{(x,y)\sim\PP_{XY}}
\big[\|\mathcal{A}(T_\rho(y))-x\|_2^2\big],
\qquad \rho\in[0,1].
\end{align}

\begin{theorem}[Informal: MMSE jumps imply stable algorithm failure]\label{thm:informal_main}
Suppose that for some $\rho\in[0,1]$ and $\alpha>0$, the MMSE is $(\rho,\alpha)$-unstable in the sense,
\[
\mathrm{MMSE}_\rho-\mathrm{MMSE}
\;\ge\;
\alpha\cdot
\E_{(x,y)\sim\PP_{XY}}\big[\|x\|_2^2\big].
\]
Then any $(\rho,o(\alpha^2))$-stable algorithm $\mathcal{A}$ is $\Omega(\alpha)$-suboptimal:
\[
\E_{(x,y)\sim\PP_{XY}}\big[\|\mathcal{A}(y)-x\|_2^2\big]
\;\ge\;
\mathrm{MMSE}
+\Omega(\alpha)\cdot
\E_{(x,y)\sim\PP_{XY}}\big[\|x\|_2^2\big].
\]
\end{theorem}

\begin{figure}
    \centering

\tikzset{every picture/.style={line width=0.75pt}} 

\begin{tikzpicture}[x=0.75pt,y=0.75pt,yscale=-1,xscale=1]

\draw    (53.2,217.5) -- (330.2,217.5) ;
\draw [shift={(332.2,217.5)}, rotate = 180] [color={rgb, 255:red, 0; green, 0; blue, 0 }  ][line width=0.75]    (10.93,-3.29) .. controls (6.95,-1.4) and (3.31,-0.3) .. (0,0) .. controls (3.31,0.3) and (6.95,1.4) .. (10.93,3.29)   ;
\draw    (111.2,271.5) -- (111.2,39.5) ;
\draw [shift={(111.2,37.5)}, rotate = 90] [color={rgb, 255:red, 0; green, 0; blue, 0 }  ][line width=0.75]    (10.93,-3.29) .. controls (6.95,-1.4) and (3.31,-0.3) .. (0,0) .. controls (3.31,0.3) and (6.95,1.4) .. (10.93,3.29)   ;
\draw  [dash pattern={on 4.5pt off 4.5pt}]  (112,70) -- (305.2,70) ;
\draw    (111.33,216.67) .. controls (113.6,198.53) and (114.57,185.49) .. (115.6,153.87) .. controls (116.63,122.24) and (119.54,85.81) .. (121.84,80.56) .. controls (124.14,75.31) and (127.37,74.38) .. (132.64,72.96) .. controls (137.91,71.54) and (146.64,71.36) .. (150.64,71.76) .. controls (154.64,72.16) and (227.6,70.53) .. (298.27,71.2) ;
\draw    (112.67,222.33) -- (122.24,222.33) ;
\draw [shift={(122.24,222.33)}, rotate = 180] [color={rgb, 255:red, 0; green, 0; blue, 0 }  ][line width=0.75]    (0,5.59) -- (0,-5.59)   ;
\draw [shift={(112.67,222.33)}, rotate = 180] [color={rgb, 255:red, 0; green, 0; blue, 0 }  ][line width=0.75]    (0,5.59) -- (0,-5.59)   ;
\draw  [dash pattern={on 0.84pt off 2.51pt}]  (122.24,80.56) -- (122.24,222.33) ;
\draw  [dash pattern={on 0.84pt off 2.51pt}]  (111.84,80.56) -- (122.24,80.56) ;
\draw    (103.04,80.56) -- (103.04,215.53) ;
\draw [shift={(103.04,215.53)}, rotate = 270] [color={rgb, 255:red, 0; green, 0; blue, 0 }  ][line width=0.75]    (0,5.59) -- (0,-5.59)   ;
\draw [shift={(103.04,80.56)}, rotate = 270] [color={rgb, 255:red, 0; green, 0; blue, 0 }  ][line width=0.75]    (0,5.59) -- (0,-5.59)   ;

\draw (91,227.4) node [anchor=north west][inner sep=0.75pt]    {$0$};
\draw (302,48.6) node [anchor=north west][inner sep=0.75pt]   [align=left] {trivial};
\draw (38,45) node [anchor=north west][inner sep=0.75pt]  [font=\small]  {$MMSE_{\rho }$};
\draw (112.4,229) node [anchor=north west][inner sep=0.75pt]  [font=\footnotesize]  {$\rho $};
\draw (83.6,135) node [anchor=north west][inner sep=0.75pt]  [font=\footnotesize]  {$\alpha $};
\draw (310.8,229.4) node [anchor=north west][inner sep=0.75pt]  [font=\small] [align=left] {Noise};

\end{tikzpicture}
    \caption{A pictorial representation of a sharp MMSE jump under small noise, which leads to stable algorithm failure per Theorem \ref{thm:informal_main}.}
    \label{fig:MMSE-jump}
\end{figure}
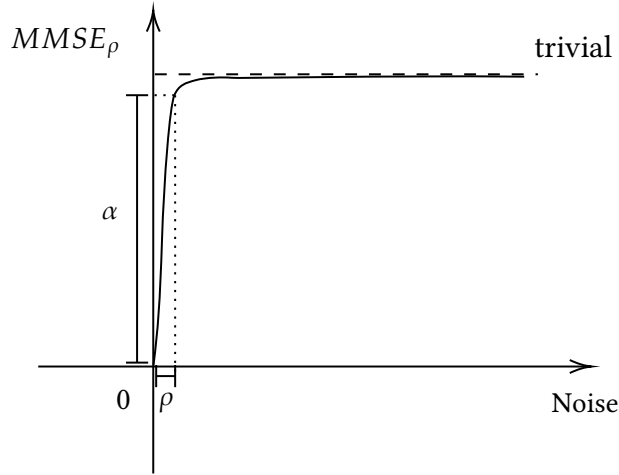

This reduces lower bounds for stable algorithms in estimation to understanding the threshold behavior of the MMSE under small noise $\rho$: the sharper the transition of the MMSE (i.e., the larger the $\alpha$), the stronger the resulting failure of stable methods (see Figure \ref{fig:MMSE-jump}). This naturally raises two questions. First, can the MMSE exhibit a genuine ``jump'' under small noise? Second, when such a jump occurs, can it be sufficiently sharp to exclude algorithmic classes of real interest?

To formalize this discussion, consider the normalized MMSE (NMMSE),
\[
\mathrm{NMMSE}_\rho
\;=\;
\frac{\mathrm{MMSE}_\rho}{\E_{(x,y)\sim\PP_{XY}}[\|x\|_2^2]},
\qquad \rho\in[0,1].
\]
While many canonical models—such as spiked matrix models with i.i.d.\ Gaussian priors—exhibit a continuous limiting NMMSE curve (see e.g., \cite{lelarge2017fundamental}), it is also well documented that the asymptotic NMMSE can be discontinuous, a phenomenon known in statistical physics as a \emph{first-order phase transition}, see e.g., the discussions in \cite{dia2016mutual}. An extreme but widely observed instance is the \emph{all-or-nothing} (AoN) phenomenon, first observed in \cite{reeves2019all}, where the NMMSE jumps abruptly from nearly zero at $\rho=0$ to nearly one upon the injection of an arbitrarily small constant noise $\rho>0$. Many inference models are now known to exhibit such behavior, including Gaussian additive models \cite{niles2020all}, generalized linear models \cite{luneau2022information}, planted subgraph models \cite{wu2022settling, mossel2023sharp}, and a large family of discrete channels, such as Bernoulli group testing \cite{niles2023all}.

At first glance, one might therefore conclude that any model exhibiting AoN must necessarily imply the failure of all stable algorithms. While this is true to some extent in principle, the implication must be interpreted carefully. Theorem~\ref{thm:informal_main} implies that AoN yields $\Omega(1)$-suboptimality only for algorithms that are, say, $(\rho,0.9)$-stable at a constant noise level $\rho$. However, many powerful stable methods—most notably low-degree polynomials—do not often satisfy such strong stability. In fact, under canonical product measures (e.g.\ Gaussian or Bernoulli), a degree-$D$ polynomial is typically $(\rho,O(D\rho))$-stable, a property that often follows from standard noise-operator arguments. While extending such stability guarantees beyond null models is highly nontrivial, we show in this work that analogous bounds hold for several planted distributions of interest. Consequently, AoN alone typically rules out only constant-degree polynomials, and excluding higher-degree methods requires a finer analysis of the width of the MMSE transition window.

In the application sections, we develop tools to tightly control these transition windows, allowing us to prove stability results for low-degree polynomials at the relevant noise scales and to derive the desired separations.

Finally, we note that the perspective that MMSE discontinuities have algorithmic implications has deep roots in the literature. In statistical physics, first-order phase transitions are routinely associated with computational hardness \cite{dia2016mutual,zdeborova2016statistical}. In mathematical works, we only know one relevant result, \cite{gamarnik2023sharp} showing that sharp probabilistic thresholds in random Boolean problems imply the failure of constant-depth circuits; however, this connection relies on classical threshold combinatorial notions and techniques such as H{\aa}stad's switching lemma, and is not expected to extend to MMSE-based phenomena or beyond Bernoulli measures. We do highlight that our approach for proving stable algorithm failure holds for all parametric estimation tasks.

\textbf{Applications and separations.}
As discussed in the introduction, the main known separations between stable or low-degree algorithms and polynomial-time methods arise from a small number of canonical settings\footnote{We exclude from this list the very interesting recent counterexample \cite{buhai2025quasi} as it seperate quasipolynomial-time methods from low-degree polynomials}: random linear systems (in both optimization and detection), the shortest path problem (in optimization), and the Gaussian subset sum problem (in detection).

We show that our MMSE-instability framework applies uniformly to the \emph{estimation} variants of all three settings, yielding separations between stable algorithms—and, as a corollary we prove, also low-degree polynomials—from polynomial-time methods (see Figure \ref{fig:venn}). Notably, our results reveal a common structural feature underlying these previously disparate examples: all three problems exhibit MMSE instability, which in turn forces the suboptimality of stable methods. This perspective clarifies what the shortest path problem shares with the “noiseless” linear system and subset-sum settings, and hopefully will help the community finalize the right class of tasks an appropriate low-degree/stable algorithm optimality conjecture should apply. We now provide more details on the specific separations.

\medskip
\noindent\textbf{Planted Shortest Path.}
Consider the planted shortest path problem on an Erd\H{o}s--R\'enyi graph $G(n,q=\Theta(\log n/n))$, where a path of length $L=C\log n/\log\log n$ is planted uniformly at random. For $C<1$, the planted path is the unique shortest path with high probability and can be recovered by Dijkstra's algorithm.
We show that the MMSE is $(O(1/L),\Omega(1))$-unstable, implying that all $(O(1/L),o(1))$-stable algorithms fail. Moreover, we also prove all degree-$D=o(L)$ polynomials are stable, yielding a low-degree MMSE lower bound and a separation from polynomial-time methods.

\medskip
\noindent\textbf{Random Linear Code.}
Finally, consider the random linear code problem where one observes $y=Ax$ with $A\in\{0,1\}^{m\times n}$ having i.i.d.\ uniform entries and $x\in\{0,1\}^n$ uniform. For $m\ge n+\omega(1)$, Gaussian elimination recovers $x$ in polynomial time with high probability.
We prove that the MMSE is $(O((m-n)/n),\Omega(1))$-unstable. Consequently, all $(O((m-n)/n),o(1))$-stable algorithms fail. We further show that all degree-$D=O(n/(m-n))$ polynomials are stable and hence suboptimal, yielding another separation to the polynomial-time class.

\medskip
\noindent\textbf{Gaussian Subset Sum.}
Let $X_i\sim\mathcal{N}(0,1)$ i.i.d.\ for $i=1,\ldots,N$. The goal is to recover a hidden subset $S\subset[N]$ of size $k=N^{\alpha+o(1)}$, $\alpha\in(0,1)$, from the observation $\sum_{i\in S}X_i$. A lattice-based method using the Lenstra--Lenstra--Lov\'asz algorithm succeeds with high probability.
Using AoN results, we show that the MMSE is $(\exp(-\Theta(k\log(N/k))),\Omega(1))$-unstable. By Theorem~\ref{thm:informal_main}, all $(\exp(-\Theta(k\log(N/k))),o(1))$-stable algorithms fail. We further prove that all degree-$D=o(\min\{k^{1/4},(n/k)^{1/5})$ polynomials are stable at this noise level, yielding a low-degree MMSE lower bound and a separation between low-degree polynomials and polynomial-time algorithms.

\section{Main Result: MMSE discontinuities imply Stable Algorithm Failure}

In this section, we formally present our main connection. We work under the assumption of an arbitrary parametric estimation setting as described in Section \ref{sec:contr} and an arbitrary noise operator $T_{\rho}: \R^N \rightarrow \R^n, \rho \in [0,1]$.

Our main result in the following.
\begin{theorem} \label{thm:noisy-MMSE-stable-barrier}
    Let $\rho, \eta \in [0,1].$ Recall the definition of the noisy MMSE, $\mathrm{MMSE}_{\rho},$ defined in \eqref{eq:MMSE_noisy}. 

   Then any $\mathcal{A}$ which is $(\rho,\eta)$-stable satisfies
        \begin{align*}
            \E_{(x,y) \sim \PP_{XY}}[\|\mathcal{A}(y) - x\|_2^2] \ge \text{MMSE}_{\rho} - 2\sqrt{2(7+4\eta)\eta} \cdot \E_{(x,y) \sim \PP_{XY}}[\|x\|_2^2]. 
        \end{align*}

In particular, if for some $\epsilon>0,$ $ \E_{(x,y) \sim \PP_{XY}}[\|\mathcal{A}(y) - x\|_2^2] \leq \text{MMSE}+\epsilon \EE_{XY}[\|x\|_2^2],$ then 
    \begin{align*}
         \text{MMSE}_{\rho} - \text{MMSE} \leq (2\sqrt{2(7+4\eta)\eta}+\epsilon) \cdot \E_{(x,y) \sim \EE_{XY}}[\|x\|_2^2]. 
    \end{align*}

\end{theorem}

Theorem \ref{thm:noisy-MMSE-stable-barrier} has the following immediate corollary.

\begin{corollary}\label{cor:main}
    Suppose for some $\rho \in [0,1],\alpha>0$ that \[\text{MMSE}_{\rho} - \text{MMSE}  \geq \alpha \E_{(x,y) \sim \PP_{XY}}[\|x\|_2^2].\]Then, for any $\eta \leq \min \{\alpha^2/400,1\}$, any $\mathcal{A}$ algorithm which is $(\rho, \eta)$-stable must be $\alpha/2$-suboptimal, in the sense 
    \begin{align}\label{eq:MMSE_sub}
    \E_{(x,y) \sim \PP_{XY}}[\|\mathcal{A}(y) - x\|_2^2] \geq \text{MMSE}+\alpha \EE_{XY}[\|x\|_2^2]/2.
    \end{align}
\end{corollary}

\begin{proof}
    The proof follows from the last displayed inequality of Theorem \ref{thm:noisy-MMSE-stable-barrier} using that for the range of $\eta$'s of interest it holds $2\sqrt{2(7+4\eta)\eta} \leq \alpha/2.$
\end{proof}
In words, the Corollary suggests that to prove the failure of stable algorithms for an estimation task, we may study how quickly the MMSE jumps if noise is injected to the input. This is exactly our approach in our applications in the next section

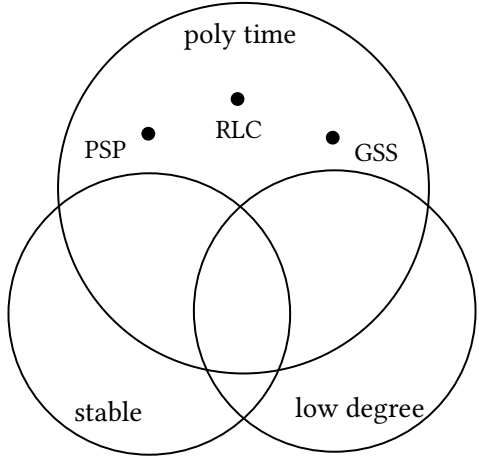
\begin{figure}
    \centering

\tikzset{every picture/.style={line width=0.75pt}} 

\begin{tikzpicture}[x=0.75pt,y=0.75pt,yscale=-1,xscale=1]

\draw  [color={rgb, 255:red, 0; green, 0; blue, 0 }  ,draw opacity=1 ] (132.45,123.52) .. controls (132.45,71.87) and (174.32,30) .. (225.98,30) .. controls (277.63,30) and (319.5,71.87) .. (319.5,123.52) .. controls (319.5,175.18) and (277.63,217.05) .. (225.98,217.05) .. controls (174.32,217.05) and (132.45,175.18) .. (132.45,123.52) -- cycle ;
\draw   (200.95,185.52) .. controls (200.95,146.3) and (232.75,114.5) .. (271.98,114.5) .. controls (311.2,114.5) and (343,146.3) .. (343,185.52) .. controls (343,224.75) and (311.2,256.55) .. (271.98,256.55) .. controls (232.75,256.55) and (200.95,224.75) .. (200.95,185.52) -- cycle ;
\draw   (107.45,187.02) .. controls (107.45,147.8) and (139.25,116) .. (178.48,116) .. controls (217.7,116) and (249.5,147.8) .. (249.5,187.02) .. controls (249.5,226.25) and (217.7,258.05) .. (178.48,258.05) .. controls (139.25,258.05) and (107.45,226.25) .. (107.45,187.02) -- cycle ;
\draw  [fill={rgb, 255:red, 0; green, 0; blue, 0 }  ,fill opacity=1 ] (175,96.02) .. controls (175,94.38) and (176.33,93.05) .. (177.98,93.05) .. controls (179.62,93.05) and (180.95,94.38) .. (180.95,96.02) .. controls (180.95,97.67) and (179.62,99) .. (177.98,99) .. controls (176.33,99) and (175,97.67) .. (175,96.02) -- cycle ;
\draw  [fill={rgb, 255:red, 0; green, 0; blue, 0 }  ,fill opacity=1 ] (268,98.02) .. controls (268,96.38) and (269.33,95.05) .. (270.98,95.05) .. controls (272.62,95.05) and (273.95,96.38) .. (273.95,98.02) .. controls (273.95,99.67) and (272.62,101) .. (270.98,101) .. controls (269.33,101) and (268,99.67) .. (268,98.02) -- cycle ;
\draw  [fill={rgb, 255:red, 0; green, 0; blue, 0 }  ,fill opacity=1 ] (220,78.52) .. controls (220,76.88) and (221.33,75.55) .. (222.98,75.55) .. controls (224.62,75.55) and (225.95,76.88) .. (225.95,78.52) .. controls (225.95,80.17) and (224.62,81.5) .. (222.98,81.5) .. controls (221.33,81.5) and (220,80.17) .. (220,78.52) -- cycle ;

\draw (194.5,38.5) node [anchor=north west][inner sep=0.75pt]  [font=\small] [align=left] {poly time};
\draw (250.5,227.5) node [anchor=north west][inner sep=0.75pt]  [font=\small] [align=left] {low degree};
\draw (139,229.5) node [anchor=north west][inner sep=0.75pt]  [font=\small] [align=left] {stable};
\draw (280.5,99.5) node [anchor=north west][inner sep=0.75pt]  [font=\footnotesize] [align=left] {GSS};
\draw (144.5,98) node [anchor=north west][inner sep=0.75pt]  [font=\footnotesize] [align=left] {PSP};
\draw (210.5,88) node [anchor=north west][inner sep=0.75pt]  [font=\footnotesize] [align=left] {RLC};

\end{tikzpicture}
    \caption{Pictorial representation of our separation results between the  class polynomial-time algorithms can solve, and the classes stable algorithms/low-degree polynomials can solve for parametric estimation.}
    \label{fig:venn}
\end{figure}

\section{Applications}
\subsection{The Planted Shortest Path Model}
We start with defining the first parametric estimation problem of interest.
\begin{definition}
    For $c,C>0$ and $n \in \N$. In the planted shortest path (PSP) problem defined on $n$-vertex undirected graphs, a path $H$ of length $L := (C+o(1))\frac{\log n}{\log\log n}$ is sampled uniformly at random between vertex $1$ and vertex $2$. The observed graph $G$ is the union of the sampled path $H$ and an independent instance of the Erd\H{o}s-R\'{e}nyi random graph $G(n, q)$ where $q := c\frac{\log n}{n}$. 
    
    The estimation goal of the statistician is to estimate the planted path $H$ from the observed graph $G$. 
    
\end{definition}
A few basic remarks are in order.
\begin{remark}
   We will often, equivalently, parametrize the planted shortest path problem by $L$ and $q$, which are implicitly parametrized by the constants $C$ and $c$.  
\end{remark}
\begin{remark}
    We will often use the adjacency matrix to represent a graph $G$. The indicator of an edge between vertices $i$ and $j$ is given by the matrix entry $G_{ij} \in \{0,1\}$.
\end{remark}

Now, we make the following important remark on the ``easiness" of the task.
\begin{remark}[``Polynomial-time solvability of PSP"]
    The reason that the above model is called the planted shortest path problem is that, for any constant $c > 0$, once $C < 1$, the planted path $H$ will with high probability (as $n$ grows) be the shortest path between vertex $1$ and $2$ in the observed graph $G$. In particular, in that regime Dijsktra's algorithm can exactly output $H$ from $G$ in polynomial-time, with high probability.
\end{remark}

We are interested to understand the power of stable algorithms for this task, and for this reason we define the following natural noise operator.
\begin{definition}[Noise operator for PSP]
    We consider the following natural noise operator $T_{\rho}$ for the planted shortest path problem. For $\rho \in [0,1]$, the noise operator $T_{\rho}$ maps the observed graph $G$ to a noisy version $\hat{G}$ of it, where for every pair of vertices $i,j$, independently with probability $\rho$, the edge connection $\hat{G}_{ij}$ is resampled from $\text{Bern}(q)$, and with the remaining probability, the edge connection is unchanged and $\hat{G}_{ij} = G_{ij}$.
\end{definition}

\begin{remark}
    Note that if $\rho = 0$, $T_{\rho}(G) = G$, and if $\rho = 1$, $T_{\rho}(G)$ is distributed as a fresh instance of an Erd\H{o}s-R\'{e}nyi random graph $G(n,q)$.
\end{remark}

Our first result is on the MMSE instability of the PSP task. As a corollary, using Theorem \ref{thm:noisy-MMSE-stable-barrier}, we conclude the failure of stable algorithms for this polynomial-time solvable task.
\begin{theorem}\label{thm:noisy-MMSE-PSP}
    For any constants $c > 0$ and $C \in (0,1)$, if $\rho \in (0, 1/2]$ satisfies that $\rho = \omega\left(1/L\right)$, then the noisy MMSE of the PSP problem is at least
    \begin{align*}
        \text{MMSE}_{\rho} \ge (1 - o(1))L.
    \end{align*}
    In particular, for large enough $n$, all $(\rho,0.01)$-stable algorithms are $0.2$-suboptimal, in the sense of \eqref{eq:MMSE_sub}.
\end{theorem}

The MMSE instability result for PSP says that even if we disconnect on average any growing number of edges from the planted path, the recovery problem becomes suddenly impossible. The proof of this sharp MMSE jump is deferred to Section \ref{sec:PSP}. We highlight that it is an interesting application of the so-called planting trick from the literature of random constraint-satisfaction problems \cite{achlioptas2008algorithmic}, and especially how it has been recently used in the AoN/threshold literature \cite{coja2022statistical,mossel2025bayesian}, alongside a careful second moment method.

Our next result establishes the failure of low-degree polynomials for the PSP task, via Theorem \ref{thm:noisy-MMSE-PSP}. We do this in two steps. First, we prove all symmetric low-degree polynomials are stable for PSP and conclude their failure from Theorem \ref{thm:noisy-MMSE-PSP}. Then we prove that there exists an optimal polynomial for PSP task among the degree-$D$ polynomials which is symmetric to conclude the lower bound.

\begin{definition}[Symmetric Polynomials]\label{def:sym-poly}
    A polynomial $f: \{0,1\}^{\binom{[n]}{2}} \to \R$ is said to be symmetric for PSP if for any permutation $\pi: [n] \to [n]$ that fixes vertices $1$ and $2$, the polynomial is invariant under permuting the indices of the variables $G_{i,j}$ according to the permutation $\pi$. We will use $G_{\pi}$ to denote the graph obtained from $G$ by permuting vertices according to $\pi$.
\end{definition}

The stability of symmetric low-degree polynomials is then as follows.
\begin{theorem}[Stability of Symmetric Low-Degree Polynomials]\label{PSP:stab}
    Let $f: \{0,1\}^{\binom{n}{2}} \to \R$ be a symmetric polynomial of degree at most $D$. Suppose $\frac{D(1-q)}{nq} = o(1)$, $2D < L$, $\frac{L^2D}{n} = o(1)$, and $q \le 1/2$. Then, $f$ is $\left(\rho, 2\left(1  - (1 - \rho)^D + O\left(\sqrt{\frac{D(1-q)}{nq}}\right)\right)\right)$-stable for the PSP task.
\end{theorem}The proof of this theorem is also deferred to Section \ref{sec:PSP}. We highlight that while in nature direct, it follows from careful counting arguments and is significantly more involved than any noise stability calculation under the null product measure.

Now, while our stability result is stated for symmetric polynomials, our stable algorithm lower bound actually applies more generally to arbitrary low-degree polynomials. This is due to the following convexity argument.

\begin{proposition}\label{PSP:sym}
    Let $g: \{0,1\}^{\binom{[n]}{2}} \to \R$ be a polynomial of degree at most $D$. Then, there exists a symmetric polynomial $f$ of degree at most $D$, defined as \begin{align*}f(G) = \frac{1}{(n-2)!}\sum_{\pi \text{ that fixes } 1,2} g(G_{\pi})\end{align*}
    that satisfies
    \begin{align*}
        \E_{\PP}\left[\left(f(G) - \allone\{\{i,j\} \in \bH\}\right)^2\right] \le \E_{\PP}\left[\left(g(G) - \allone\{\{i,j\} \in \bH\}\right)^2\right].
    \end{align*}
\end{proposition}

\begin{proof}
This fact is a simple corollary of Jensen's inequality:
\begin{align*}
    \E_{\PP}\left[\left(f(G) - \allone\{\{i,j\} \in \bH\}\right)^2\right] &= \E_{\PP}\left[\left(\underset{\pi \text{ that fixes } 1,2}{\E} g(G_{\pi}) - \allone\{\{i,j\} \in \bH\}\right)^2\right]\\
    &\le \E_{\PP}\left[\underset{\pi \text{ that fixes } 1,2}{\E}\left(g(G_{\pi}) - \allone\{\{i,j\} \in \bH\}\right)^2\right]\\
    &= \E_{\PP}\left[\left(g(G) - \allone\{\{i,j\} \in \bH\}\right)^2\right],
\end{align*}
where the last equality holds because for any $\pi$, $G_{\pi}$ and $G$ are equidistributed since $\PP$ is invariant under such permutation.
\end{proof}

Combining Theorem~\ref{thm:noisy-MMSE-stable-barrier}, Theorem~\ref{PSP:stab}, and Proposition~\ref{PSP:sym}, we obtain the following low-degree MMSE lower bound for PSP. In particular, there exists a separation between $D = O(\log n/\log\log n)$-degree polynomials and polynomial-time algorithms for PSP.
\begin{theorem}[Hardness for Polynomials]\label{PSP:LD}
    Let $f: \{0,1\}^{\binom{[n]}{2}} \to \R^{\binom{n}{2}}$ be a polynomial of degree at most $D$. Suppose $D < \frac{L}{2} \le \frac{C}{2} \frac{\log n}{\log\log n}$ and $\rho = \omega\left(1/L\right) = \omega\left(\frac{\log \log n}{\log n}\right)$, then the mean squared error of $f$ for PSP task is at least
    \begin{align*}
        \E\left[\|f(G) - H\|_2^2\right] &\ge \left(1 - O\left(\sqrt{1 - (1-\rho)^D + \sqrt{D/\log n}}\right) - o(1)\right)L\\
        &\ge \left(1 - O\left(\sqrt{\rho D + \sqrt{D/\log n}}\right) - o(1)\right)L.
    \end{align*}
\end{theorem}We last highlight that our hardness of low-degree polynomial for the PSP task is almost tight, in the sense that there exists a degree-$L$ polynomial estimator that achieves exact recovery with high probability, whereas the theorem above shows that any degree-$o(L)$ polynomial estimator has trivial mean squared error. 

\subsection{Random Linear Code}

Our second separation result is for the Random Linear Code (RLC) setting.
\begin{definition}
    In the RLC problem, parametrized by $n,m \in \NN$, a Boolean matrix $A \in \{0,1\}^{m \times n}$ is sampled uniformly at random, and a message $x \in \{0,1\}^n$ is sampled independently uniformly at random. Then the codeword $y=Ax \in \{0,1\}^m$ is observed where the matrix-vector multiplication is computed in the field $\mathbb{F}_2$. The goal is to estimate the message $x$ from $(y,A)$.
\end{definition}

\begin{remark}[``Polynomial-time solvability of RLC"]
    We note that in the noiseless case where $y = Ax$ for RLC, once $A$ has full column rank, one may recover $x$ exactly from $A$ and $y$ and in polynomial time by performing Gaussian elimination. It is also easy to check that the random Boolean matrix $A \in \{0,1\}^{m \times n}$ has full column rank with probability at least $1 - 2^{n-m}$. Thus, as long as $m - n = \omega(1)$, there exists a polynomial-time algorithm that can achieve exact recovery of $x$ with high probability.
\end{remark}

Turning to stable algorithms, we define the noise operator.
\begin{definition}[Noise operator for RLC]
     For $\rho \in [0,1]$, the noise operator $T_{\rho}$ maps the encoded message $y = Ax$ to a noisy version $\hat{y}$ of it, where every coordinate of $y$ is resampled from $\text{Bern}(1/2)$ independently with probability $\rho$.
\end{definition}

\begin{remark}
    Note that the noise operator $T_{\rho}$ acts only on the encoded message $y$ and leaves $A$ unchanged. If $\rho = 0$, $T_{\rho}$ is the identity map, and if $\rho = 1$, $T_{\rho}$ maps any $y$ to a uniform random vector distributed as $\text{Bern}(1/2)^{\otimes m}$.
\end{remark}

\begin{theorem} \label{thm:linear-code-noisy-MMSE}
    For any function $f(n) = \omega(1)$ and $m = n + f(n)$, if $\rho \in (0,1)$ satisfies $\frac{1}{\left(1 - \frac{3}{2}\rho + \frac{3}{4}\rho^2\right)^m} = \omega \left(2^{2(m-n)}\right)$, then the noisy MMSE of the RLC problem is at least
    \begin{align*}
        \text{MMSE}_{\rho} \ge \left(\frac{1}{4} - o(1)\right)n.
    \end{align*}
    In particular, for large enough $n$, all $(\rho,0.01)$-stable algorithms are $0.2$-suboptimal, in the sense of \eqref{eq:MMSE_sub}.
\end{theorem}We prove this MMSE instability theorem now by a direct analysis of the posterior mean and the proof can be found in Section \ref{sec:RLC}. 

Some remarks are in order.
\begin{remark}
    We note that $\frac{1}{4}n$ is the mean squared error achieved by trivially outputting the constant vector $(1/2, 1/2, \dots, 1/2)^\top \in \R^n$.
\end{remark}

\begin{remark} \label{rem:sufficent-condition-linear-code-noisy-MMSE}
    Let us elaborate further on the condition $\frac{1}{\left(1 - \frac{3}{2}\rho + \frac{3}{4}\rho^2\right)^m} \gg 2^{2(m-n)}$ in Theorem~\ref{thm:linear-code-noisy-MMSE} above. For $1 \ll f(n) \ll n$, it essentially says that $\rho$ needs to be at least on the order of $\frac{m-n}{n}$. More precisely, if there exists a constant $\eps > 0$ such that $\frac{1}{2}\rho - \frac{1}{4} \rho^2 \ge \left(\frac{2\log 2}{3} + \eps\right)\cdot \frac{m-n}{m}$, then the condition $\frac{1}{\left(1 - \frac{3}{2}\rho + \frac{3}{4}\rho^2\right)^m} \gg 2^{2(m-n)}$ is satisfied. We could verify the claim by checking
    \begin{align*}
        \frac{1}{\left(1 - \frac{3}{2}\rho + \frac{3}{4}\rho^2\right)^m} = e^{-m\log\left(1 - \frac{3}{2}\rho + \frac{3}{4}\rho^2\right)} \ge e^{3m\left(\frac{1}{2}\rho - \frac{1}{4}\rho^2\right)} \ge e^{3\left(\frac{2\log 2}{3} + \eps\right)(m-n)} \gg 2^{2(m-n)}.
    \end{align*}Hence our MMSE instability result, Theorem \ref{thm:linear-code-noisy-MMSE} proves that as long as we randomize roughly $m-n$ of the $m$ entries of $y$ the recover task becomes impossible. 
\end{remark}

Our next result is the $(\rho, O(\rho D))$-stability of any $D$-degree polynomial for the RLC task.
\begin{theorem}[Stability of Low-Degree Polynomials for RLC]\label{thm:low-deg-stability-rlc}
    Let $f: \{0,1\}^{m \times n} \times \{0,1\}^m \to \R$ be a polynomial of degree at most $D$. Suppose $2D \le n$. Then, $f$ is $(\rho, 2(1 - (1-\rho)^D))$-stable for RLC.
\end{theorem}The $\mathbb{F}_2$-structure of the RLC setting is, in fact, leading to ``almost" independence between $y$ and $A$, which allows for a much simpler proof of the low-degree stability as opposed to other models. The proof is deferred to Section \ref{sec:RLC}.

As a corollary, we get the following low-degree MMSE lower bound for RLC. In particular, we conclude a separation between $D=O(m/(m-n))$-degree polynomials and polynomial-time methods for RLC. 
\begin{corollary}
    For any function $f(n) = \omega(1)$ and $m = n + f(n)$, if for some constant $\eps > 0$, we have $\frac{1}{2}\rho - \frac{1}{4}\rho^2 \ge \left(\frac{2\log 2}{3} + \eps\right)\cdot \frac{m-n}{m}$, then any degree-$D$ polynomial $g: \{0,1\}^{m \times n} \times \{0,1\}^n \to \R^n$ has a mean squared error at least
    \begin{align*}
        \E\left[\|g(A,y) - x\|_2^2\right] \ge \left(\frac{1}{4} - O\left(\sqrt{1 - (1-\rho)^D}\right) - o(1)\right)n.
    \end{align*}
    In particular, if $D \le \frac{1}{\rho}$, we have
    \begin{align*}
        \E\left[\|g(A,y) - x\|_2^2\right] \ge \left(\frac{1}{4} - O\left(\sqrt{\rho D}\right) - o(1)\right)n.
    \end{align*}
\end{corollary}

\subsection{Results for Gaussian Subset Sum (GSS)}
The final model we establish our separation is the Gaussian Subset Sum (GSS) model.

\begin{definition}
    Let $k, N \in \NN$ with $k \leq N$. In the Gaussian subset sum (GSS) problem, for some unknown $k$-subset $S \subseteq [N]$ we observe $Y=\sum_{i \in S} X_i$ where $(X_i)_{i \in [N]}$ are i.i.d. $N(0,1).$ The goal of the statistician is to estimate $1_S \in \{0,1\}^N$, the indicator of the set $S,$ from $Y$ and $X_i, i=1,2,\ldots,N$. 
\end{definition}

\begin{remark}[``Polynomial-time solvability of GSS"]
It is an easy calculation that for any $k,N \in \NN$ as long as $\sigma=0$ one can recover $S$ with probability 1. Now, differently, from PSP or RLC where the corresponding worst-case task is solvable even in the worst-case (shortest path and linear systems are in $\mathcal{P}$), subset sum is not expected be in $\mathcal{P}$. Surprisingly, though, there is a polynomial-time algorithm that can recover exactly $1_S$ in GSS with high probability \cite{frieze1986lagarias,zadik2018high}. Notably, the algorithm is based on Lenstra-Lenstra-Lovasz lattice-basis reduction scheme.
\end{remark}

We define the noise operator for the task.
\begin{definition}[Ornstein–Uhlenbeck (OU) operator for GSS]
     For $\rho \in [0,1]$, the OU noise operator $T_{\rho}$ maps the subset sum $y$ to a noisy version $\hat{y}=\sqrt{1-\rho^2}y+\rho Z$ where $Z$ is an independent $\mathcal{N}(0,1).$
\end{definition}
To prove this, we first prove the following proposition, which follows as a direct corollary of the main result of AoN results in \cite{reeves2019all}. The proof is deferred to Section \ref{sec:GSS_proofs}.

\begin{proposition}\label{prop:aon_GSS}
    Suppose $k \leq N^{0.49}.$ Then for some $\rho=\exp(-\Theta(k \log (N/k)))$ it holds  \[\text{MMSE}_{\rho} \geq (1-o(1))k.\]In particular, for large enough $n$, all $(\rho,0.01)$-stable must be $0.4$-suboptimal, in the sense of \eqref{eq:MMSE_sub}.
\end{proposition}

In addition to showing the stable algorithm separation, we also obtain hardness for low-degree polynomials. As usual by now, we first prove the following stability theorem for low-degree polynomials.

\begin{theorem}\label{thm:low-deg-stability-gss}
    Let $f: \R^n \times \R \to \R$ be a polynomial of degree at most $D$. Suppose $k = o(n)$ and $D = o\left(\min\left\{k^{1/4},  \left(n/k\right)^{1/5}\right\}\right)$. Then, $f$ is $(\rho, 2(1 - (\sqrt{1-\rho^2})^D + o(1)))$-stable for the GSS.
\end{theorem}This result is now non-trivial to prove and follows by a series of careful combinatorial bounds, as well as a Wick's formula for the Hermite basis. The proof is deferred to Section \ref{sec:GSS_proofs}. 

As a corollary, we get the following low-degree MMSE lower bound for GSS. In particular,  we conclude a separation between $D = o\left(\min\{k^{1/4}, (n/k)^{1/5}\}\right)$-degree polynomials and polynomial-time methods for GSS.

\begin{corollary}
    Suppose $k = o(n)$, $D = o\left(\min\{k^{1/4}, (n/k)^{1/5}\}\right)$, and $\rho \ge \exp(-k\log(n/k))$. Then, any degree-$D$ polynomial $f: \R^n \times \R \to \R^n$ has a mean squared error at least
    \begin{align*}
        \E\left[\|f(X,Y) - \one_S\|_2^2\right] \ge \left(1 - O\left(\sqrt{1 - \left(\sqrt{1-\rho^2}\right)^D}\right) - o(1)\right)k.
    \end{align*}
    In particular, if $D \le \frac{1}{\rho^2}$, we have
    \begin{align*}
        \E\left[\|f(X,Y) - \one_S\|_2^2\right] \ge \left(1 - O\left(\rho^2 D\right) - o(1)\right)k.
    \end{align*}
\end{corollary}

\newpage


\section{Getting started: Preliminaries for the Proofs}
Here, we present some background and discussion for our results that can benefit the reader.

First, recall that the minimum mean squared error (MMSE) is achieved by the posterior mean.

\begin{fact}\label{fact:MMSE-posterior-mean}
    Posterior mean achieves the MMSE:
    \[\inf_{\mathcal{A}: \R^N \to \R^n} \E_{(x,y) \sim \PP_{XY}}[\|\mathcal{A}(y) - x\|_2^2] = \E_{(x,y) \sim \PP_{XY}}[\|\E[x \vert y] - x\|_2^2].\]
\end{fact}

Now, while we left the notion of noise operator abstractly defined above, it is perhaps instructive to consider some properties a natural noise operator satisfies.
\begin{enumerate}
    \item When $\rho = 0$, the noise operator should be the identity map.
    \item When $\rho = 1$, the noise operator introduces the maximal amount of noise, and observing $T_{\rho}(y)$ provides no extra information about $x$, i.e., for any $x_1 \ne x_2$ in the support of $X$, the law of $T_{\rho}(y) \vert x_1$ is the same as the law of $T_{\rho}(y) \vert x_2$.
    \item For $\rho_1 < \rho_2$, the noise operator $T_{\rho_1}$ injects less noise than $T_{\rho_2}$, and there exists $\rho_3 \in [0,1]$ such that $T_{\rho_2} = T_{\rho_3} \circ T_{\rho_1}$.
\end{enumerate}

Our noise operator roughly acts in the following way on a fixed subsets of coordinates of the observation $y$:
\begin{itemize}
    \item ($q$-Bernoulli Noise) $T_{\rho}$ acts on a fixed subset of $y$ by independently with probability $\rho$ resampling each coordinate within the subset from $\text{Bern}(q)$.
    \item (Gaussian Noise) $T_{\rho}$ acts on a fixed subset of $y$ by independently applying Ornstein–Uhlenbeck operator at each coordinate. In other words, within the subset, the noise operator averages each coordinate of $y$ with an independent standard Gaussian $N(0,1)$ and outputs $\sqrt{1-\rho^2} y_i + \rho z_i$ where $z_i \sim N(0,1)$ for coordinate $y_i$.
\end{itemize}

\section{Proof of the key result: MMSE instability implies stable algorithm failure}
In this section, we include the (easy) proof of Theorem \ref{thm:noisy-MMSE-stable-barrier}.

\begin{proof}
    Let $\mathcal{A}: \R^N \to \R^n$ be a $(\rho, \eta)$-stable algorithm.

    We first claim that either $\E_{(x,y) \sim \PP_{XY}}[\|\mathcal{A}(y) - x\|_2^2] \ge \text{MMSE}_{\rho}$ in which case the desired statement holds, or $\E_{(x,y) \sim \PP_{XY}}[\|\mathcal{A}(y) - x\|_2^2] \leq \text{MMSE}_{\rho} \leq \E_{(x,y) \sim \PP_{XY}}[\|x\|_2^2]$ which implies
    \begin{align}\label{ineq:sec-moment-assumption-1}\E_{(x,y) \sim \PP_{XY}}\left[\|\mathcal{A}(y)\|_2^2\right] \le \E_{(x,y) \sim \PP_{XY}}[\|\mathcal{A}(y) - x\|_2^2]+ 2\E_{(x,y) \sim \PP_{XY}}[\|x\|_2^2] \leq  4\E_{(x,y) \sim \PP_{XY}}[\|x\|_2^2].
    \end{align}

    Moreover, our assumptions and \eqref{ineq:sec-moment-assumption-1} give

    \begin{align}&\E_{(x,y) \sim \PP_{XY}}\left[\|\mathcal{A}(T_{\rho}(y))\|_2^2\right]  \\
    &\leq 2\E_{(x,y) \sim \PP_{XY}}\left[\|\mathcal{A}(y)\|_2^2\right] +2\E_{(x,y) \sim \PP_{XY}}\left[\|\mathcal{A}(T_{\rho}(y))-\mathcal{A}(y)\|_2^2\right] \\
    & \leq 8\E_{(x,y) \sim \PP_{XY}}\left[\|x\|_2^2\right]+2\eta\E_{(x,y) \sim \PP_{XY}}\left[\|\mathcal{A}(y)\|_2^2\right] \\
    & \leq 8(1+\eta)\E_{(x,y) \sim \PP_{XY}}\left[\|x\|_2^2\right]. \label{ineq:sec-moment-assumption-2}
    \end{align}

Now,
    \begin{align*}
        &\quad \E_{(x,y) \sim \PP_{XY}}\left[\|\mathcal{A}(T_{\rho}(y))-x\|_2^2 - \|\mathcal{A}(y)-x\|_2^2\right]\\
        &= \E_{(x,y) \sim \PP_{XY}}\left[\langle \mathcal{A}(T_{\rho}(y))-\mathcal{A}(y) , \mathcal{A}(T_{\rho}(y)) + \mathcal{A}(y) - 2x \rangle\right]\\
        &\le \sqrt{\E_{(x,y) \sim \PP_{XY}}\left[\|\mathcal{A}(T_{\rho}(y))-\mathcal{A}(y)\|_2^2\right]\left(\E_{(x,y) \sim \PP_{XY}}\left[\|\mathcal{A}(T_{\rho}(y))\|_2^2 + \|\mathcal{A}(y)\|_2^2 + 2 \|x\|_2^2\right]\right)}
    \end{align*}
    by Cauchy-Schwarz inequality. Using the $(\rho, \eta)$-stablility of $\mathcal{A}$, \eqref{ineq:sec-moment-assumption-1} and \eqref{ineq:sec-moment-assumption-2}, we have
    \begin{align*}
        &\quad \E_{(x,y) \sim \PP_{XY}}\left[\|\mathcal{A}(T_{\rho}(y))-x\|_2^2 - \|A(y)-x\|_2^2\right]\\
        &\le \sqrt{\E_{(x,y) \sim \PP_{XY}}\left[\|\mathcal{A}(T_{\rho}(y))-\mathcal{A}(y)\|_2^2\right]\left(\E_{(x,y) \sim \PP_{XY}}\left[\|\mathcal{A}(T_{\rho}(y))\|_2^2 + \|\mathcal{A}(y)\|_2^2 + 2 \|x\|_2^2\right]\right)}\\
        &\le \sqrt{\E_{(x,y) \sim \PP_{XY}}\left[\|\mathcal{A}(T_{\rho}(y))-\mathcal{A}(y)\|_2^2\right]\left(\E_{(x,y) \sim \PP_{XY}}\left[8(1 + \eta)\|x\|_2^2 +4 \|x\|_2^2 + 2 \|x\|_2^2\right]\right)}\\
        &\le \sqrt{2(7+4\eta)\eta \cdot \E_{(x,y) \sim \PP_{XY}}\left[ \|\mathcal{A}(y)\|_2^2\right]\E_{(x,y) \sim \PP_{XY}}\left[\|x\|_2^2\right]}\\
        &\le  2\sqrt{2(7+4\eta)\eta \cdot \E_{(x,y) \sim \PP_{XY}}\left[ \|x\|_2^2\right]\E_{(x,y) \sim \PP_{XY}}\left[\|x\|_2^2\right]}\\
        &= 2\sqrt{2(7+4\eta)\eta}\cdot \E_{(x,y) \sim \PP_{XY}}\left[\|x\|_2^2\right].
    \end{align*}
    Rearranging the inequality, we get
    \begin{align*}
        &\quad \E_{(x,y) \sim \PP_{XY}}[\|\mathcal{A}(y) - x\|_2^2]\\
        &\ge \E_{(x,y) \sim \PP_{XY}}[\|\mathcal{A}(T_{\rho}(y)) - x\|_2^2] - 2\sqrt{2(7+4\eta)\eta} \cdot \E_{(x,y) \sim \PP_{XY}}[\|x\|_2^2]\\
        &\ge \text{MMSE}_{\rho} - 2\sqrt{2(7+4\eta)\eta} \cdot \E_{(x,y) \sim \PP_{XY}}[\|x\|_2^2]
    \end{align*}
    as desired.
\end{proof}

\section{Proofs for Planted Shortest Path}\label{sec:PSP}

\subsection{Noisy MMSE lower bound: Proof of Theorem \ref{thm:noisy-MMSE-PSP}}

Our main idea for showing a lower bound for noisy MMSE of the planted shortest path problem is to study a natural object we call an ``approximate path".

\begin{definition}[Approximate Path]
    Let $G$ be a graph, $\ell \in \NN$ and $\epsilon \in (0,1)$ such that $\epsilon \ell \in \NN$. A pair $(\hat{P}, P)$ is a $(1-\eps)$-approximate path of length $\ell$ in $G$ between vertices $1$ and $2$, if $P$ is a path of length $\ell$ between vertices $1$ and $2$ in the complete graph $K_n$, $\hat{P} \subseteq P$, $\hat{P} \subseteq G$, and $|\hat{P}| = (1-\eps) \ell$.
    
    In other words, $(\hat{P}, P)$ is a $(1-\eps)$-approximate path of length $\ell$ in $G$ if $P$ is a path of length $\ell$ in the complete graph $K_n$, $\hat{P}$ contains exactly $(1-\eps)$ fraction of edges of $P$, and $\hat{P}$ is a subgraph of $G$.
\end{definition}

\begin{definition}[Overlap of Approximate Paths]
    Let $(\hat{P}_1, P_1)$ and $(\hat{P}_2, P_2)$ be two $(1-\eps)$-approximate paths of length $\ell$ in a graph $G$. We say $(\hat{P}_1, P_1)$ and $(\hat{P}_2, P_2)$ shares $k$ edge, or have an overlap of size $k$, if $|E(P_1 \cap P_2)| = k$.
\end{definition}

\begin{remark}[Approximate Paths and Noisy Planted Shortest Path] \label{rem:approximate-path-noisy-model}
    The notion of approximate path appears naturally in the \emph{noisy version} of the planted shortest path problem. Recall that in the planted shortest path problem, a graph $G$ is sampled to be the union of an Erd\H{o}s-R\'{e}nyi random graph $G(n,q)$ and a uniform random path $H$ between vertices $1$ and $2$ of length $L$. The noise operator $T_{\rho}$ acts on $G$ by independently resampling each $G_{ij}$ from $\text{Bern}(q)$ with probability $\rho$. Thus, the observed noisy graph $T_{\rho}(G)$ is the union of an Erd\H{o}s-R\'{e}nyi random graph $G(n,q)$ and $\hat{H}$, where $\hat{H}$ is obtained from $H$ by disconnecting each edge of $H$ independently with probability $\rho$. Consequently, $(\hat{H}, H)$ is a $(1 - \eps)$-approximate path of length $L$ between vertices $1$ and $2$ in $\hat{G}$, where $\eps L$ is distributed according to $\text{Bin}(L, \rho)$. Moreover, conditioned on the event that $\eps L$ edges are disconnected in $H$, the distribution of $(\hat{H},H)$ is the uniform distribution over all $(1-\eps)$-approximate path of length $L$ in $G$ between vertices $1$ and $2$.
\end{remark}

Our first step for the proof is to bound the second moment of $(1-\eps)$-approximate paths of a fixed length between vertices $1$ and $2$ in the null model $\Q$ where we observe an Erd\H{o}s-R\'{e}nyi random graph $G(n,q)$. Our proof is adapted from \cite[Lemma 2.1]{li2024some}.

\begin{theorem}\label{thm:2nd_MM_Shortest}
    Let $N_{m, \eps}$ denote the number of $(1-\eps)$-approximate paths of length $m$ between vertices $1$ and $2$ inside an Erd\H{o}s-R\'{e}nyi random graph $G(n,q)$ and $N^{(2)}_{m,\epsilon}$ denote the number of non-disjoint pairs of $(1-\eps)$-approximate paths of length $m$ between vertices $1$ and $2$ inside an Erd\H{o}s-R\'{e}nyi random graph $G(n,q)$, i.e., the number of pairs of $((\hat{P}_1, P_1), (\hat{P}_2, P_2))$ of $(1-\eps)$-approximate paths of length $m$ between vertices $1$ and $2$ in $G(n,q)$ such that $P_1$ and $P_2$ are not edge-disjoint in the complete graph $K_n$. If $m = \omega(1)$, $ 1 \le \eps m \le m-1$, $nq \ge \omega(\sqrt{m})$, $m \ll n^{1/15}$, and $mq \le 1/3$, then \[E_{\mathbb{Q}}[N^{(2)}_{m, \eps}]\le   O\left(\frac{1}{n^{\eps m - 1}\cdot (nq)}+\frac{\sqrt{m}}{nq}\right) \EE_{\mathbb{Q}}[N_{m, \eps}]^2 = o(\EE_{\mathbb{Q}}[N_{m, \eps}]^2),\]
    \[ \E_{\mathbb{Q}}[N_{m, \eps}^2] \le \left(1 + O\left(\frac{1}{n^{\eps m - 1}\cdot (nq)}+\frac{\sqrt{m}}{nq}\right)\right)\E_{\QQ}[N_{m,\eps}]^2 \le (1 + o(1)) \EE_{\mathbb{Q}}[N_{m, \eps}]^2,\]
    and
    \[\E_{\mathbb{Q}}[N_{m, \eps}] \ge \Omega\left(n^{\eps m - 1}\cdot (nq)\right) \ge \omega(1).\]
 
\end{theorem}

We now prove that Theorem \ref{thm:2nd_MM_Shortest} implies the following MMSE lower bound.

\begin{corollary}\label{cor:conditional-noisy-MMSE}
  Let $x \in \{0,1\}^{\binom{n}{2}},$ be a path from vertex 1 to vertex 2 of length $m$ chosen uniformly at random. Consider the planted model $\PP_{m,\epsilon,q}$ where one plants in an instance $G_0 \sim G(n,q)$ a uniformly at random chosen $(1-\epsilon)m$-subset of $x$, denoted by $x_{\epsilon} \in \{0,1\}^{\binom{n}{2}}.$ Then if $m = \omega(1)$, $1 \le \eps m \le m-1$, $nq \ge \omega(\sqrt{m})$, $m \ll n^{1/15}$, and $mq \le 1/3$, then   
  \[\E[\|\E[x \vert G=G_0 \cup x_{\epsilon}] - x\|_2^2] \geq \E[\|x\|^2_2] \left(1-O\left(\frac{1}{n^{\eps m - 1}\cdot (nq)}+\frac{\sqrt{m}}{nq}\right)\right).\]
  \end{corollary}

\begin{proof}This proof uses some careful change of measure identities, often called by the name ``planting trick" \cite{achlioptas2008algorithmic,mossel2025bayesian}.
Let $S_{m}$ denote the set of paths of length $m$ between vertices $1$ and $2$ in $K_n$ and $M_{m}=|S_{m}|.$ 
    Let also $S_{m,\epsilon}$ denote the set of $(1-\epsilon)$-approximate paths of length $m$ between vertices $1$ and $2$ in $K_n$ and $M_{m,\epsilon}=|S_{m,\epsilon}|.$ Now denote for any $n$-vertex $G$ and $x \in S_{m}$ \[Z_Y(G)=(\binom{m}{\epsilon m}q^{m(1-\epsilon)}M_{m})^{-1} \sum_{P \in S_{m,\epsilon}, H \in S_m: P \subseteq H} 1(P \subseteq G),  \]and
    \[Z_Y(x,G)=(\binom{m}{\epsilon m}q^{m(1-\epsilon)}M_{m})^{-1} \sum_{P \in S_{m,\epsilon}, H \in S_m: H \cap x =\emptyset, P \subseteq H} 1(P \subseteq G).  \]
    Notice that for any $G$, the following hold. For any value of $x_{\epsilon} \in S_{m,\epsilon},$ \[\frac{\PP_{m,\epsilon,q}(G|x_{\epsilon})}{\QQ_q(G)}=1(x_{\epsilon} \subseteq G)q^{-(1-\epsilon)m}\] and therefore for any value of $x \in S_m,$ \begin{align}\label{eq:change_cond}\frac{\PP_{m,\epsilon,q}(G|x)}{\QQ_q(G)}= \frac{1}{\binom{m}{\epsilon m}q^{(1-\epsilon)m}}\sum_{P \in S_{m,\epsilon}: P \subseteq x} 1(P \subseteq G).
    \end{align} Hence,
    \begin{align}\label{eq:change}\frac{\PP_{m,\epsilon,q}(G)}{\QQ_q(G)}=M_m^{-1}\sum_{x \in S_m}\frac{\PP_{m,\epsilon,q}(G|x)}{\QQ_q(G)}=Z_Y(G).
    \end{align}
    Clearly, for $\QQ_q$ the $G(n,q)$ measure, using the uniformity of the prior in the first step and the change of measure \eqref{eq:change_cond} in the second step we have
    \begin{align*}
        &\EE_{\PP_{m,\epsilon,q}}Z_Y(x,G)\\
         &=(\binom{m}{\epsilon m}q^{m(1-\epsilon)}M_{m})^{-1} M_m^{-1}\sum_{x,H \in S_m: H \cap x =\emptyset} \EE_{\PP_{m,\epsilon,q}|x} \sum_{P \in S_{m,\epsilon}: P \subseteq H}1(P \subseteq G)\\
        &=(\binom{m}{\epsilon m}q^{m(1-\epsilon)}M_{m})^{-2} \sum_{x,H \in S_m: H \cap x =\emptyset} \EE_{\QQ_q} \sum_{P,P' \in S_{m,\epsilon}: P \subseteq H, P' \subseteq x}1(P,P' \subseteq G)\\
         &=(\binom{m}{\epsilon m}q^{m(1-\epsilon)}M_{m})^{-2}\EE_{\QQ_q} \sum_{x,H \in S_m: H \cap x =\emptyset}  \sum_{P,P' \in S_{m,\epsilon}: P \subseteq H, P' \subseteq x}1(P,P' \subseteq G)\\
    \end{align*}Now let $N_{m,\epsilon}$ the number of $1-\epsilon$-approximate paths. Then clearly $\EE_{\QQ_q} N_{m,\epsilon}=\binom{m}{\epsilon m}q^{m(1-\epsilon)}$ and \[\EE_{\QQ_q} N^{(2)}_{m,\epsilon}=\EE_{\QQ_q} \sum_{x,H \in S_m: H \cap x =\emptyset}  \sum_{P,P' \in S_{m,\epsilon}: P \subseteq H, P' \subseteq x}1(P,P' \subseteq G).\] Combining the above with Theorem \ref{thm:2nd_MM_Shortest} we conclude
\[\EE_{\PP_{m,\epsilon,q}}Z_Y(x,G)=\frac{\E_{\mathbb{Q}_q}[N^{(2)}_{m, \eps}]}{\EE_{\mathbb{Q}_q}[N_{m, \eps}]^2}= O\left(\frac{1}{n^{\eps m - 1}\cdot (nq)}+\frac{\sqrt{m}}{nq}\right).\] and therefore by Theorem \ref{thm:2nd_MM_Shortest}, \[\EE_{\PP_{m,\epsilon,q}}Z_Y(x,G)=O\left(\frac{1}{n^{\eps m - 1}\cdot (nq)}+\frac{\sqrt{m}}{nq}\right).\]

   Moreover, for any $\delta>0,$ by \eqref{eq:change}, \[\PP_{m, \epsilon, q}(Z_Y(G) \leq \delta)=\EE_{\PP_{m, \epsilon, q}}[\frac{\PP_{m,\epsilon,q}(G)}{\QQ_q(G)}1(Z_Y(G) \leq \delta)] \leq \delta.\] Hence, for any $\delta>0,$ 
    \[\EE_{\PP_{m,\epsilon,q}}\frac{Z_Y(x,G)}{Z_Y(G)} \leq \delta+\EE_{\PP_{m,\epsilon,q}}Z_Y(x,G)/\delta\] which optimizing over $\delta$ allows us to conclude 
        \[\EE_{\PP_{m,\epsilon,q}}\frac{Z_Y(x,G)}{Z_Y(G)}=O\left(\sqrt{\EE_{\PP_{m,\epsilon,q}}Z_Y(x,G)} \right)=O\left(\sqrt{\frac{1}{n^{\eps m - 1}\cdot (nq)}+\frac{\sqrt{m}}{nq}}\right).\]

        But in $\PP_{m,\epsilon,q}$ the posterior distribution of $x$ given $G$ satisfies from \eqref{eq:change_cond} for all realizations $H$, \[\PP_{m,\epsilon,q}(H|G) \propto \sum_{P \in S_{m,\epsilon}: P \subseteq H} 1(P \subseteq G).\]Hence, \[\PP_{x, G \sim \PP_{m,\epsilon,q}|x, H \sim \PP(x|G)} [H \cap x \neq \emptyset]=\EE_{\PP_{m,\epsilon,q}}\frac{Z_Y(x,G)}{Z_Y(G)}=O\left(\sqrt{\frac{1}{n^{\eps m - 1}\cdot (nq)}+\frac{\sqrt{m}}{nq}}\right),\] which implies 
        \[\E_{x, G \sim \PP_{m,\epsilon,q}|x, H \sim \PP(x|G)} [|H \cap x|]=O\left(m\sqrt{\frac{1}{n^{\eps m - 1}\cdot (nq)}+\frac{\sqrt{m}}{nq}}\right),\] and therefore by the Nishimori identity (see e.g., \cite[Lemma 2]{niles2023all})
         \begin{align*}
    \E[\|\E[x \vert G=G_0 \cup x_{\epsilon}] - x\|_2^2] &=\E[\|x\|^2_2]-\E_{x, G \sim \PP_{m,\epsilon,q}|x, H \sim \PP(x|G)} [|H \cap x|]\\
    &=\E[\|x\|^2_2]-O\left(m\sqrt{\frac{1}{n^{\eps m - 1}\cdot (nq)}+\frac{\sqrt{m}}{nq}}\right)\\
    &=\E[\|x\|^2_2](1-O\left(\sqrt{\frac{1}{n^{\eps m - 1}\cdot (nq)}+\frac{\sqrt{m}}{nq}}\right)). 
      \end{align*}
    
\end{proof}

\begin{proof}[Theorem~\ref{thm:noisy-MMSE-PSP}]
    Now, given Corollary~\ref{cor:conditional-noisy-MMSE}, we are ready to prove the stated lower bound for the noisy MMSE stated in Theorem~\ref{thm:noisy-MMSE-PSP}. In the planted shortest path problem, for some constants $c, C > 0$, $L = (C + o(1))\frac{\log n}{\log log n}$ and $q = c\frac{\log n}{n}$. We verify that $L = \omega(1), nq = \omega(\sqrt{L}), L \ll n^{1/15}$ and $Lq \le 1/3$.

    Recall that in the planted shortest path problem, we observe a graph $G$ which is the union of $G_0 \sim G(n,q)$ and a uniformly random path $H$ between vertice $1$ and $2$ of length $L$. As explained in Remark~\ref{rem:approximate-path-noisy-model}, in the noisy version of the planted shortest path problem, one observes $T_{\rho}(G)$ which is the union of an Erd\H{o}s-R\'{e}nyi random graph $\hat{G}_0 \sim G(n,q)$ and $\hat{H}$, with $(\hat{H}, H)$ being a $(1-\eps)$-approximate path of length $L$ between vertices $1$ and $2$ and $\eps L \sim \text{Bin}(L, \rho)$. Since $\hat{H}$ is obtained from $H$ by disconnecting each edge of $H$ independently with probability $\rho$, we see that conditioned on $\eps L$, $\hat{H}$ is a uniformly random $(1-\eps)L$-subset of $H$. Then, if we use $\PP_{L, q}$ to denote the distribution of $(G_0, H)$ in the planted shortest path problem and $\PP_{L, \eps, q}$ to denote the distribution of $(\hat{G}_0, (\hat{H}, H))$ in which $\hat{G}_0 \sim G(n,q)$, $H$ is a uniformly random path of length $L$ between vertices $1$ and $2$, and $\hat{H}$ is a uniformly random $(1-\eps)L$-subset of $H$, we have
    \begin{align*}
        \text{MMSE}_{\rho} &= \E_{(G_0, H) \sim \PP_{L,q}}\left[\left\|\E\left[H \vert T_{\rho}(G_0 \cup H)\right] - H\right\|_2^2\right]\\
        &= \underset{\eps L \sim \text{Bin}(L, \rho)}{\E} \E_{(\hat{G}_0, (\hat{H}, H)) \sim \PP_{L, \eps, q}} \left[\left\|\E\left[H \vert \hat{G}_0 \cup \hat{H}\right] - H\right\|_2^2\right].
    \end{align*}
    Since $L = \omega(1)$, $nq = \omega(\sqrt{L})$, $L \ll n^{1/15}$ and $Lq \le 1/3$, if $\eps L$ satisfies $1 \le \eps L \le m - 1$, by Corollary~\ref{cor:conditional-noisy-MMSE}, we have
    \begin{align*}
        \E_{(\hat{G}_0, (\hat{H}, H)) \sim \PP_{L, \eps, q}} \left[\left\|\E\left[H \vert \hat{G}_0 \cup \hat{H}\right] - H\right\|_2^2\right] \ge \E\left[\|H\|_2^2\right] \cdot \left(1-O\left(\frac{1}{n^{\eps L - 1}\cdot (nq)}+\frac{\sqrt{L}}{nq}\right)\right).
    \end{align*}
    
    Now, since $\rho L = \omega(1)$, we know that $\eps L \sim \text{Bin}(L,\rho)$ satisfies $\eps L \ge 1$ with probability $1 - (1 - \rho)^L \ge 1 - \exp(- \rho L) \ge 1 - o(1)$. Similarly, if $\rho \le 1/2$, we have $\eps L \le L - 1$ with probability at least $1 - o(1)$. Therefore,
    \begin{align*}
        &\quad \text{MMSE}_{\rho}\\
        &= \underset{\eps L \sim \text{Bin}(L, \rho)}{\E} \E_{(\hat{G}_0, (\hat{H}, H)) \sim \PP_{L, \eps, q}} \left[\left\|\E\left[H \vert \hat{G}_0 \cup \hat{H}\right] - H\right\|_2^2\right]\\
        &\ge \Pr(1 \le \eps L \le L-1) \cdot \underset{\eps L \sim \text{Bin}(L, \rho)}{\E} \left[ \E_{(\hat{G}_0, (\hat{H}, H)) \sim \PP_{L, \eps, q}} \left[\left\|\E\left[H \vert \hat{G}_0 \cup \hat{H}\right] - H\right\|_2^2\right] \Bigg\vert 1 \le \eps L \le L - 1 \right]\\
        &\ge (1 - o(1)) \cdot \E\left[\|H\|_2^2\right] \cdot \left(1-O\left(\frac{1}{nq}+\frac{\sqrt{L}}{nq}\right)\right)\\
        &= (1 - o(1))\E\left[\|H\|_2^2\right],
    \end{align*}
    since $nq = \omega(\sqrt{L})$. This concludes the proof.
\end{proof}

\begin{proof}[Theorem~\ref{thm:2nd_MM_Shortest}]
    First, we lower bound the expectation of $N_{m,\eps}$. We have
\begin{align*}
    \E_{\QQ}[N_{m,\eps}] &= (n-2)_{(m-1)} \binom{m}{\eps m} q^{m-\eps m}\\
    &\ge (1 - O(1/n)) n^{m-1} \binom{m}{\eps m}  q^{m-\eps m}\\
    &= (1 - O(1/n))  \frac{1}{n}\binom{m}{\eps m}\left( n q^{1-\eps}\right)^m\\
    &\ge (1 - O(1/n))  \frac{1}{n} \left( (nq)^{1-\eps} n^\eps \right)^m\\
    &\ge (1 - O(1/n))  \frac{n^{\eps m}}{n} (nq)^{(1-\eps)m}.
\end{align*}
In particular, if $1 \le \eps m \le m-1$, $nq \ge \omega(\sqrt{m}) \ge \omega(1)$, and $mq \le 1/3$, we have
\begin{align*}
    \E_{\QQ}[N_{m,\eps}] &\ge (1 - O(1/n))  \frac{n^{\eps m}}{n} (nq)^{(1-\eps)m}\\
    &\ge \Omega\left(n^{\eps m - 1}\cdot (nq)\right)\\
    &\ge \omega(1).
\end{align*}

Next, we upper bound the expectation on $N_{m,\eps}^{(2,k)}$, the number of pairs of $(1-\eps)$-approximate paths of length $m$ between vertices $1$ and $2$ that share $k$ edges in $G(n,q)$. Recall that this is the number of pairs of $((\hat{P}_1, P_1), (\hat{P}_2, P_2))$ of $(1-\eps)$-approximate paths of length $m$ between vertices $1$ and $2$, such that $|E(P_1 \cap P_2)| = k$. We then have
\begin{align}
    \E_{\QQ}[N_{m,\eps}^{(2,k)}] &\le M_{k,m-k,m-k} \sum_{j = 0}^{\min\{2\eps m, m-k + \eps m\}} \binom{2m-k}{j}  q^{2m-k-j}, \label{ineq:k-overlap-expectation}
\end{align}
where $M_{k,m-k,m-k}$ denotes the number of pairs $(P_1, P_2)$ of paths of length $m$ between vertices $1$ and $2$ that share $k$ edges in the complete graph $K_n$. The inner summation enumerates over the missing edges of $P_1 \cup P_2$ from $\hat{P}_1 \cup \hat{P_2}$. Note that at most $\min\{2\eps m, m-k + \eps m\}$ edges of $P_1 \cup P_2$ are not present in $\hat{P}_1 \cup \hat{P}_2$. Here, $2\eps m$ comes from that $\eps m$ edges of $P_i$ are missing from $\hat{P}_i$ for each $i \in \{1,2\}$, and $m - k + \eps m$ comes from that there are exactly $\eps m$ edges of $P_1$ missing from $\hat{P}_1$, and there are at most $m-k$ edges of $P_2 \setminus P_1$  missing from $\hat{P}_2$.

Recall that $N_{m,\eps}^{(2)}$ denote the number of non-disjoint pairs of $(1-\eps)$-approximate paths of length $m$ between vertices $1$ and $2$, i.e., the number of pairs of $((\hat{P}_1, P_1), (\hat{P}_2, P_2))$ of $(1-\eps)$-approximate paths of length $m$ between vertices $1$ and $2$, such that $|E(P_1 \cap P_2)| > 0$. By \eqref{ineq:k-overlap-expectation}, we have
\begin{align*}
    \E_{\QQ}[N_{m,\eps}^{(2)}] &= \sum_{k = 1}^{m} \E_{\QQ}[N_{m,\eps}^{(2,k)}]\\
    &= \sum_{k = 1}^{m}\E_{\QQ}[N_{m,\eps}^{(2,k)}]\\
    &\le \sum_{k=1}^{m} M_{k,m-k,m-k} \sum_{j=0}^{\min\{2\eps m, m-k + \eps m\}}\binom{2m-k}{j}  q^{2m-k - j}. \stepcounter{equation}\tag{\theequation}\label{ineq:overlap-expectation}
\end{align*}

Finally, we upper bound the second moment of $N_{m,\eps}$. Note that $N_{m,\eps}^2 = \sum_{k=0}^m N_{m,\eps}^{(2,k)} = N_{m,\eps}^{(2,0)} + N_{m,\eps}^{(2)}$ by considering all the overlap patterns of a pair of $(1-\eps)$-approximate paths. Moreover, we have $\E_{\QQ}[N_{m,\eps}^{(2,0)}] \le \E_{\QQ}[N_{m,\eps}]^2$, since for a pair disjoint paths, the events that they are $(1-\eps)$-approximate paths in $\Q$ are independent. Therefore, by \eqref{ineq:overlap-expectation}, we have 
\begin{align*}
    &\quad \E_{\QQ}[N_{m,\eps}^2]\\
    &= \E_{\QQ}[N_{m,\eps}^{(2,0)}] + \E_{\QQ}[N_{m,\eps}^{(2)}]\\
    &\le \E_{\QQ}[N_{m,\eps}]^2 + \sum_{k=1}^{m} M_{k,m-k,m-k} \sum_{j=0}^{\min\{2\eps m, m-k + \eps m\}}\binom{2m-k}{j}  q^{2m-k - j}
    \intertext{Since $mq \le 1/3$ and $m = \omega(1)$, the ratio of two consecutive terms in the inner sum over $j$ is $\frac{\binom{2m-k}{j+1}  q^{2m-k-j-1}}{\binom{2m-k}{j}  q^{2m-k-j}} = \frac{2m-k-j}{j+1}\cdot\frac{1}{q} \ge \frac{1}{2m} \cdot \frac{1}{q} \ge \frac{3}{2} $, and thus the inner sum is dominated by the last term up to a constant factor of 3. We then have}
    &\le \E_{\QQ}[N_{m,\eps}]^2  + \sum_{k=1}^{m} M_{k,m-k,m-k} \cdot 3 \binom{2m-k}{t} q^{2m-k-t}
    \intertext{where $t = \min\{2\eps m, m-k + \eps m\}$. Now we invoke \cite[Claim 2.1]{li2024some}, which says that for some constant $C$, we have $M_{k,m-k,m-k} \le \left(\frac{k+1}{n^2} + \frac{Ck^9 m^6}{n^3}\right)n^{2m-k}$ for $m \ll n^{1/3}$ and $1 \le k \le m-1$. When $k = m$, we have $M_{m,0,0} = (n-2)_{(m-1)} \le n^{m-1}$. Thus,}
    &\le \E_{\QQ}[N_{m,\eps}]^2  + 3 n^{m-1}\binom{m}{\eps m} q^{m - \eps m} \\
    &\quad + 3\sum_{k=1}^{m-1} \left(\frac{k+1}{n^2} + \frac{C k^9m^6}{n^3}\right)n^{2m-k} \binom{2m-k}{t} q^{2m-k-t}\\
    &= \E_{\QQ}[N_{m,\eps}]^2 + 3 n^{m-1}\binom{m}{\eps m} q^{m - \eps m} + 3\sum_{k=1}^{m-\eps m} \left(\frac{k+1}{n^2} + \frac{C k^9m^6}{n^3}\right)n^{2m-k} \binom{2m-k}{2\eps m}q^{2(1-\eps)m-k}\\
    &\quad + 3\sum_{k=m - \eps m + 1}^{m-1} \left(\frac{k+1}{n^2} + \frac{C k^9m^6}{n^3}\right)n^{2m-k} \binom{2m-k}{m-k+\eps m} q^{m-\eps m}\\
    &= \E_{\QQ}[N_{m,\eps}]^2+ 3\left(\frac{1}{n} \binom{m}{\eps m}\left(n  q^{1-\eps}\right)^m\right) + 3 \left(\frac{1}{n} \binom{m}{\eps m}\left(n  q^{1-\eps}\right)^m\right)^2 \\
    &\quad \cdot  \Bigg( \sum_{k=1}^{m-\eps m} \frac{1}{\binom{m}{\eps m}^2} \left(k+1 + \frac{C k^9m^6}{n}\right)\cdot \frac{\binom{2m-k}{2\eps m}}{n^k q^{k}}  + \sum_{k=m-\eps m+1}^{m-1} \frac{1}{\binom{m}{\eps m}^2}\left(k+1 + \frac{C k^9m^6}{n}\right)\cdot \frac{\binom{2m-k}{m-k+\eps m}}{n^k q^{(1-\eps)m}  } \Bigg)\\
    &\le \E_{\QQ}[N_{m,\eps}]^2 + 3 \E_{\QQ}[N_{m,\eps}] + 3 \E_{\QQ}[N_{m,\eps}]^2 \cdot  \Bigg( \sum_{k=1}^{m-\eps m} \frac{\binom{2m-k}{2\eps m}}{\binom{m}{\eps m}^2} \left(k+1 + \frac{C k^9m^6}{n}\right)\cdot \left(\frac{1}{nq}\right)^k\\
    &\quad + \sum_{k=m-\eps m+1}^{m-1} \frac{\binom{2m-k}{m - \eps m}}{\binom{m}{\eps m}^2}\left(k+1 + \frac{C k^9m^6}{n}\right)\cdot \left(\frac{1}{nq}\right)^{(1-\eps)m} \left(\frac{1}{n}\right)^{k - (1 - \eps)m} \Bigg)
    \intertext{We may check that $\frac{\binom{2m-k}{2\eps m}}{\binom{m}{\eps m}^2} \le \frac{\binom{2m}{2\eps m}}{\binom{m}{\eps m}^2} \le C_2 \sqrt{m}$ for some absolute constant $C_2$ using Stirling's approximation, and similarly $\frac{\binom{2m-k}{m - \eps m}}{\binom{m}{\eps m}^2} \le \frac{\binom{m + \eps m}{m - \eps m}}{\binom{m}{\eps m}^2} = \frac{\binom{m + \eps m}{2\eps m}}{\binom{m}{\eps m}^2} \le \frac{\binom{2m}{2\eps m}}{\binom{m}{\eps m}^2} \le C_2 \sqrt{m}$ for $k \ge m - \eps m$. Moreover, since $k \le m \ll n^{1/15}$, we have $k+1 + \frac{C k^9m^6}{n} \le k + 2$. Therefore, we have}
    &\le \E_{\QQ}[N_{m,\eps}]^2+ 3 \E_{\QQ}[N_{m,\eps}] + 3 \E_{\QQ}[N_{m,\eps}]^2 \cdot  C_2 \cdot \Bigg( \sum_{k=1}^{m-\eps m} \sqrt{m} \left(k+2\right)\cdot \left(\frac{1}{nq}\right)^k\\
    &\quad + \sum_{k=m-\eps m+1}^{m-1} \sqrt{m}\left(k+2\right)\cdot \left(\frac{1}{nq}\right)^{(1-\eps)m} \left(\frac{1}{n}\right)^{k - (1 - \eps)m} \Bigg)\\
    &\le \E_{\QQ}[N_{m,\eps}]^2 + O\left(\frac{1}{\E_{\QQ}[N_{m,\eps}]} + \frac{\sqrt{m}}{nq}\right) \cdot \E_{\QQ}[N_{m,\eps}]^2\\
    &\le \left(1 + O\left(\frac{1}{n^{\eps m - 1}\cdot (nq)}+ \frac{\sqrt{m}}{nq}\right)\right)\E_{\QQ}[N_{m,\eps}]^2.
\end{align*}
where in the second to last inequality we used that $nq \gg \sqrt{m}$ and that $n^{\eps m - 1}\cdot (nq) = \omega(1)$.

In particular, under the assumptions $nq \gg \sqrt{m}$ and $\eps m \ge 1$, we have $\E_{\QQ}[N_{m,\eps}^2] \le (1 + o(1))\E_{\QQ}[N_{m,\eps}]^2$. 

\end{proof}

\subsection{Stability of Symmetric Low-Degree Polynomials: Proof of Theorem \ref{PSP:stab}}

In this part, we prove the stability parameters for symmetric low-degree polynomials for the planted shortest path problem. Recall that in the context of planted shortest path problem, a polynomial $f$ is symmetric if for any permutation $\pi: [n] \to [n]$ that fixes vertices $1$ and $2$, we have $f(G) = f(G_{\pi})$, where $G_{\pi}$ is the graph obtained from $G$ by permuting vertices according to $\pi$. In other words, if we treat $G$ as the adjacency matrix, then $G_{\pi} = \Pi^\top G \Pi$ where $\Pi$ is the permutation matrix of $\pi$.

Before we move to the proof, we define some notations that will be useful later in the proof.

\begin{definition}[Partially Labeled Graph]
    Let $K \subseteq \N$. We will use $\GG_{\le D}^K$ to denote the collection of partially labeled graphs with at most $D$ edges and without isolated vertices, where each vertex is either unlabeled or is labeled by an element of $K$, and no two different vertices are labeled by the same vertex in $K$.

    For a partially labeled graph $\alpha$, we will use $v(\alpha)$ to denote the number of its unlabeled vertices.
\end{definition}

In our proof, the set $K$ will correspond to the two special vertices $1$ and $2$ between which a random path is planted. We will also need the notion of injective maps from a partially labeled graph to the complete graph $K_n$.

\begin{definition}
    An injective map from a partially labeled graph to the complete graph $K_n$ is an injective map from the vertex set of the partially labeled graph to $[n]$ that maps the labeled vertices to the corresponding elements of $K$, and unlabeled vertices to $[n] - K$. Given a partially labeled graph $\alpha$, we will use $\LL_{\alpha}$ to denote the set of all injective maps from $\alpha$ to $K_n$.
\end{definition}

Let us also recall the definition of symmetric polynomials for the planted shortest path problem, defined in Definition~\ref{def:sym-poly}, restated here for readers' convenience.

\begin{definition}
    A polynomial $f: \{0,1\}^{\binom{[n]}{2}} \to \R$ is said to be symmetric for the planted shortest path problem if for any permutation $\pi: [n] \to [n]$ that fixes vertices $1$ and $2$, we have $f(G) = f(G_{\pi})$ as polynomials, where $G_{\pi}$ is the graph obtained from $G$ by permuting vertices according to $\pi$.
\end{definition}

\begin{proof}
    Let $K = \{1,2\}$ denote the set of special vertices. Let $f: \{0,1\}^{\binom{[n]}{2}} \to \R$ be a symmetric polynomial of degree at most $D$. Since $\sum_{l \in \LL_{\alpha}} \chi_{l(\alpha)}(G)$ for $\alpha \in \GG_{\le D}^{K}$ are linearly independent and span the space of symmetric polynomial, we may write $f$ uniquely as
    \begin{align}
        f(G) = \sum_{\alpha \in \GG_{\le D}^{K}} \hat{f}_{\alpha} \sum_{l \in \LL_{\alpha}} \chi_{l(\alpha)}(G), \label{eq:symmetric-polynomial-expansion}
    \end{align}
    where $\chi_{S}(G)$ denotes
    \[\chi_{S}(G) = \prod_{\{i,j\} \in S} \frac{G_{i,j}-q}{\sqrt{q(1-q)}}.\]
    We use the standard fact that $\chi_{S}(G)$ forms an orthonormal basis for the measure $\Q = G(n,q)$, i.e., $\E_{\QQ}[\chi_{S}(G) \chi_{S'}(G)] = \allone\{S = S'\}$. We will also use $\PP$ to denote the planted distribution as in the planted shortest path problem, where one observes the union of $\bH$ with an Erd\H{o}s-R\'{e}nyi graph drawn from $G(n,q)$ and $\bH$ is a uniformly random path of length $L$ between vertices $1$ and $2$. Under the planted measure $\PP$, we have
    \begin{align*}
        \E_{\PP}[\chi_{S}(G)\chi_{S'}(G)] &= \E_{\PP}\left[\prod_{\{i,j\} \in S \triangle S'} \frac{G_{i,j}-q}{\sqrt{q(1-q)}} \cdot \prod_{\{i,j\} \in S \cap S'} \frac{(G_{i,j}-q)^2}{q(1-q)} \right]\\
        &= \E_{\PP}\left[\allone\{S \triangle S' \subseteq \bH\} \left(\frac{1-q}{q}\right)^{\frac{|S \cap \bH|}{2} + \frac{|S' \cap \bH|}{2}}\right], \stepcounter{equation}\tag{\theequation}\label{eq:correlation-shortest}
    \end{align*}
    where we used that when conditioned on $\bH$, the conditional distribution becomes a product distribution that is a Dirac measure on $1$ for $\{i,j\} \in \bH$, and is $\text{Bern}(q)$ otherwise.
    In what follows, the expectations are taken with respect to $\PP$.

    Now, we have
    \begin{align*}
        \E[\left(f(G) - f(T_{\rho}(G))\right)^2] = \E[f(G)^2] + \E[f(T_{\rho}(G))^2] - 2\E[f(G)f(T_{\rho}(G))].
    \end{align*}
    We will analyze each of the three terms separately.

    First, given the expansion of the symmetric polynomial $f$ in \eqref{eq:symmetric-polynomial-expansion}, we have
    \begin{align*}
        &\quad \E[f(G)^2]\\
        &= \sum_{\alpha_1, \alpha_2 \in \GG_{\le D}^K} \hat{f}_{\alpha_1} \hat{f}_{\alpha_2} \sum_{l_1 \in \LL_{\alpha_1}, l_2 \in \LL_{\alpha_2}} \E[\chi_{l_1(\alpha_1)}(G)\chi_{l_2(\alpha_2)}(G)]\\
        &=\sum_{\alpha_1, \alpha_2 \in \GG_{\le D}^K} \hat{f}_{\alpha_1} \hat{f}_{\alpha_2} \sum_{l_1 \in \LL_{\alpha_1}, l_2 \in \LL_{\alpha_2}} \E\left[\allone\{l_1(\alpha_1) \triangle l_2(\alpha_2) \subseteq \bH\} \left(\frac{1-q}{q}\right)^{\frac{|l_1(\alpha) \cap \bH|}{2} + \frac{|l_2(\alpha) \cap \bH|}{2}}\right]\\
        &= \sum_{\alpha \in \GG_{\le D}^K} \hat{f}_{\alpha}^2 \sum_{l_1, l_2 \in \LL_{\alpha}} \EE\left[\allone\{l_1(\alpha) \triangle l_2(\alpha) \subseteq \bH\} \left(\frac{1-q}{q}\right)^{\frac{|l_1(\alpha) \cap \bH|}{2} + \frac{|l_2(\alpha) \cap \bH|}{2}}\right]\\
        &\quad + \sum_{\substack{\alpha_1, \alpha_2 \in \GG_{\le D}^K:\\ \alpha_1 \ne \alpha_2}} \hat{f}_{\alpha_1}\hat{f}_{\alpha_2}\sum_{l_1 \in \LL_{\alpha_1}, l_2 \in \LL_{\alpha_2}} \EE\left[\allone\{l_1(\alpha_1) \triangle l_2(\alpha_2) \subseteq \bH\} \left(\frac{1-q}{q}\right)^{\frac{|l_1(\alpha) \cap \bH|}{2} + \frac{|l_2(\alpha) \cap \bH|}{2}}\right],
    \end{align*}
    where we used \eqref{eq:correlation-shortest} in the second to last equality.

    Since $q \le 1/2$, we have $\frac{1-q}{q} \ge 1$ and the first sum over pairs of $\alpha_1 = \alpha_2$ is at least
    \begin{align*}
        &\quad \sum_{\alpha \in \GG_{\le D}^K} \hat{f}_{\alpha}^2 \sum_{l_1, l_2 \in \LL_{\alpha}} \EE\left[\allone\{l_1(\alpha) \triangle l_2(\alpha) \subseteq \bH\} \left(\frac{1-q}{q}\right)^{\frac{|l_1(\alpha) \cap \bH|}{2} + \frac{|l_2(\alpha) \cap \bH|}{2}}\right]\\
        &\ge \sum_{\alpha \in \GG_{\le D}^K} \hat{f}_{\alpha}^2 \sum_{\substack{l_1, l_2 \in \LL_{\alpha}\\ l_1(\alpha) = l_2(\alpha)}} \EE\left[\left(\frac{1-q}{q}\right)^{\frac{|l_1(\alpha) \cap \bH|}{2} + \frac{|l_2(\alpha) \cap \bH|}{2}}\right]\\
        &\ge (1 - O(1/n))\sum_{\alpha} \hat{f}_{\alpha}^2 n^{v(\alpha)} |\text{Aut}(\alpha)|,
    \end{align*}
    where $v(\alpha)$ denote the number of unlabelled vertices in $\alpha$.

    The second sum over pairs $\alpha_1, \alpha_2$ of non-isomorphic partially labeled graphs can be bounded by
    \begin{align*}
        &\quad \left|\sum_{\substack{\alpha_1, \alpha_2 \in \GG_{\le D}^K:\\ \alpha_1 \ne \alpha_2}} \hat{f}_{\alpha_1}\hat{f}_{\alpha_2} \sum_{l_1 \in \LL_{\alpha_1}, l_2 \in \LL_{\alpha_2}} \EE\left[\allone\{l_1(\alpha_1) \triangle l_2(\alpha_2) \subseteq \bH\} \left(\frac{1-q}{q}\right)^{\frac{|l_1(\alpha_1) \cap \bH|}{2} + \frac{|l_2(\alpha_2) \cap \bH|}{2}}\right]\right|\\
        &\le \sum_{\substack{\alpha_1, \alpha_2 \in \GG_{\le D}^K:\\ \alpha_1 \ne \alpha_2}} \left|\hat{f}_{\alpha_1}\hat{f}_{\alpha_2}\right| \sum_{l_1 \in \LL_{\alpha_1}, l_2 \in \LL_{\alpha_2}}  \E\left[\allone\{l_1(\alpha_1) \triangle l_2(\alpha_2) \subseteq \bH\} \left(\frac{1-q}{q}\right)^{\frac{|l_1(\alpha_1) \cap \bH|}{2} + \frac{|l_2(\alpha_2) \cap \bH|}{2}}\right]\\
        &= \sum_{\substack{\alpha_1, \alpha_2 \in \GG_{\le D}^K:\\ \alpha_1 \ne \alpha_2}} \left|\hat{f}_{\alpha_1}\hat{f}_{\alpha_2}\right| \sum_{l_1 \in \LL_{\alpha_1}, l_2 \in \LL_{\alpha_2}}  \allone\{l_1(\alpha_1) \triangle l_2(\alpha_2) \subseteq H\} \left(\frac{1-q}{q}\right)^{\frac{|l_1(\alpha_1) \cap H|}{2} + \frac{|l_2(\alpha_2) \cap H|}{2}}
    \end{align*}
    for a fixed image $H$ of an arbitrary injective map of $\bH$, due to the symmetry of $f$. We now proceed to bound the sum by considering how the images $l_1(\alpha_1)$ and $l_2(\alpha_2)$ intersect with the path $H$. In the rest of the proof, we let
    \begin{align*}
        U_H &:= (l_1(\alpha_1) \cup l_2(\alpha_2)) \cap H,
        I_H:= (l_1(\alpha_1) \cap l_2(\alpha_2)) \cap H,\\
        s &:= |U_H|, t := |I_H|, v_s := |V(U_H) - K|,\\
        v_t &:= |V(l_1(\alpha_1) \cap H) \cap V(l_2(\alpha_2) \cap H) - K|,\\
        v_{s,1} &:= |V(l_1(\alpha_1)) \cap V(U_H) - K|, v_{s,2} := |V(l_2(\alpha_2)) \cap V(U_H) - K|,\\
        c_1 &:= \allone\{1 \in V(U_H)\}, c_2 := \allone\{2 \in V(U_H)\}, c_s := cc(U_H) - c_1 - c_2,
    \end{align*}
    where $cc(U_H)$ denotes the number of connected components of $U_H$. Now, for some choice of parameters $W = (s,t,v_s,v_t,v_{s,1},v_{s,2},c_1,c_2,c_s)$, let us use $\LL_{\alpha_1, \alpha_2}(W)$ to denote all pairs of $(l_1, l_2)$ whose images $l_1(\alpha_1)$ and $l_2(\alpha_2)$ satisfy the parameters defined above, and that $l_1(\alpha_1) \triangle l_2(\alpha_2) \subseteq H$. Then, we have
    \begin{align*}
        &\quad \left|\sum_{\substack{\alpha_1, \alpha_2 \in \GG_{\le D}^K:\\ \alpha_1 \ne \alpha_2}} \hat{f}_{\alpha_1}\hat{f}_{\alpha_2} \sum_{l_1 \in \LL_{\alpha_1}, l_2 \in \LL_{\alpha_2}} \EE\left[\allone\{l_1(\alpha_1) \triangle l_2(\alpha_2) \subseteq \bH\} \left(\frac{1-q}{q}\right)^{\frac{|l_1(\alpha_1) \cap \bH|}{2} + \frac{|l_2(\alpha_2) \cap \bH|}{2}}\right]\right|\\
        &\le \sum_{\substack{\alpha_1, \alpha_2 \in \GG_{\le D}^K:\\ \alpha_1 \ne \alpha_2}} \left|\hat{f}_{\alpha_1}\hat{f}_{\alpha_2}\right| \sum_{l_1 \in \LL_{\alpha_1}, l_2 \in \LL_{\alpha_2}}  \allone\{l_1(\alpha_1) \triangle l_2(\alpha_2) \subseteq H\} \left(\frac{1-q}{q}\right)^{\frac{|l_1(\alpha_1) \cap H|}{2} + \frac{|l_2(\alpha_2) \cap H|}{2}}\\
        &= \sum_{\substack{\alpha_1, \alpha_2 \in \GG_{\le D}^K:\\ \alpha_1 \ne \alpha_2}} \left|\hat{f}_{\alpha_1}\hat{f}_{\alpha_2}\right| \sum_{\substack{W = (s,t,v_s,v_t,\\ v_{s,1}, v_{s,2},c_1,c_2,c_s) } } \sum_{(l_1, l_2) \in \LL_{\alpha_1, \alpha_2}(W)} \allone\{l_1(\alpha_1) \triangle l_2(\alpha_2) \subseteq H\} \left(\frac{1-q}{q}\right)^{\frac{|l_1(\alpha_1) \cap H|}{2} + \frac{|l_2(\alpha_2) \cap H|}{2}}\\
        &= \sum_{\substack{\alpha_1, \alpha_2 \in \GG_{\le D}^K:\\ \alpha_1 \ne \alpha_2}} \left|\hat{f}_{\alpha_1}\hat{f}_{\alpha_2}\right| \sum_{\substack{W = (s,t,v_s,v_t,\\ v_{s,1}, v_{s,2},c_1,c_2,c_s) } } \sum_{(l_1, l_2) \in \LL_{\alpha_1, \alpha_2}(W)}  \left(\frac{1-q}{q}\right)^{\frac{|(l_1(\alpha_1) \cup l_2(\alpha_2)) \cap H|}{2} + \frac{|(l_1(\alpha_1) \cap l_2(\alpha_2)) \cap H|}{2}}
        \intertext{Now recall that $U_H := (l_1(\alpha_1) \cup l_2(\alpha_2)) \cap H$, $I_H := (l_1(\alpha_1) \cap l_2(\alpha_2)) \cap H$, and $s := |U_H|$, $t := |I_H|$. Thus,}
        &= \sum_{\substack{\alpha_1, \alpha_2 \in \GG_{\le D}^K:\\ \alpha_1 \ne \alpha_2}} \left|\hat{f}_{\alpha_1}\hat{f}_{\alpha_2}\right| \sum_{\substack{W = (s,t,v_s,v_t,\\ v_{s,1}, v_{s,2},c_1,c_2,c_s) } } \sum_{(l_1, l_2) \in \LL_{\alpha_1, \alpha_2}(W)}  \left(\sqrt{\frac{1-q}{q}}\right)^{s + t}\\
        &\le 2\sum_{\alpha_1 \in \GG_{\le D}^K} \left|\hat{f}_{\alpha_1}\right| \sum_{\substack{W = (s,t,v_s,v_t,\\ v_{s,1}, v_{s,2},c_1,c_2,c_s) } } \sum_{\substack{\alpha_2 \in \GG_{\le D}^K:\\ \alpha_1 \ne \alpha_2,\\ |\hat{f}_{\alpha_2}|\left(\frac{n}{2D}\right)^{\frac{v(\alpha_2)}{2}}|\text{Aut}(\alpha_2)|\\
        \le |\hat{f}_{\alpha_1}|\left(\frac{n}{2D}\right)^{\frac{v(\alpha_1)}{2}}|\text{Aut}(\alpha_1)| } } \left|\hat{f}_{\alpha_2}\right|  \sum_{(l_1, l_2) \in \LL_{\alpha_1, \alpha_2}(W)} \left(\sqrt{\frac{1-q}{q}}\right)^{s + t},
    \end{align*}
    where the last line follows because for any pair of $\alpha_1, \alpha_2$, either $|\hat{f}_{\alpha_2}|\left(\frac{n}{2D}\right)^{\frac{v(\alpha_2)}{2}}|\text{Aut}(\alpha_2)| \le |\hat{f}_{\alpha_1}|\left(\frac{n}{2D}\right)^{\frac{v(\alpha_1)}{2}}|\text{Aut}(\alpha_1)|$ or $|\hat{f}_{\alpha_2}|\left(\frac{n}{2D}\right)^{\frac{v(\alpha_2)}{2}}|\text{Aut}(\alpha_2)| > |\hat{f}_{\alpha_1}|\left(\frac{n}{2D}\right)^{\frac{v(\alpha_1)}{2}}|\text{Aut}(\alpha_1)|$.

    For fixed $\alpha_1 \in \GG_{\le D}^K$ and a fixed choice of $W = (s, t, v_s, v_t, v_{s,1}, v_{s,2}, c_1, c_2, c_s)$, let us consider how many $(\alpha_2, (l_1,l_2) \in \LL_{\alpha_1,\alpha_2}(W))$ there are. We note some relations between these parameters in $W$. Under the assumption $2D < L$ We have
    \begin{align}
        v_s &= s + c_s\\
        v_s + v_t &\le v_{s,1} + v_{s,2}\\
        t &\le s\\
        t &\le v_t\\
        0 &\le c_s \le s\\
        1 &\le s.
    \end{align}
    We may verify the six relations above as follows
    \begin{itemize}
        \item \begin{align*}
            v_s &= |V(U_H) - K|\\
            &= \sum_{\delta \in CC(U_H)} |V(\delta)| - |V(U_H) \cap K|
            \intertext{Since $H$ is a path, each connected component $\delta$ of $U_H = (l_1(\alpha_1) \cup l_2(\alpha_2)) \cap H$ is also a path, and thus $|V(\delta)| = |\delta| + 1$. Consequently,}
            &= \sum_{\delta \in CC(U_H)} (|\delta| + 1) - c_1 - c_2\\
            &= |U_H| + |CC(U_H)| - c_1 - c_2\\
            &= s + c_s,
        \end{align*} where $CC(U_H)$ denotes the collection of connected components of $U_H$.
        \item 
        \begin{align*}
            v_s + v_t &= |V(U_H) - K| + |V(l_1(\alpha_1) \cap H) \cap V(l_2(\alpha_2) \cap H) - K|\\
            &\le |(V(l_1(\alpha_1)) \cup V(l_2(\alpha_2)))\cap V(U_H) - K|\\
            &\quad+ |(V(l_1(\alpha_1)) \cap V(l_2(\alpha_2))) \cap V(U_H) - K|
            \intertext{where the inequality holds because for every $v \in V(U_H) = V((l_1(\alpha_1)\cup l_2(\alpha_2))\ \cap H)$, we have $v \in V(l_1(\alpha_1)) \cup V(l_2(\alpha_2))$, and for $i \in \{1,2\}$ and for every $v \in V(l_i(\alpha_i) \cap H)$, we have $v \in V(l_i(\alpha_i)) \cap V((l_1(\alpha_1) \cup l_2(\alpha_2)) \cap H) = V(l_i(\alpha_i)) \cap V(U_H)$,}
            &= |(V(l_1(\alpha_1)) \cap V(U_H) - K| + |V(l_2(\alpha_2))) \cap V(U_H) - K|\\
            &= v_{s,1} + v_{s,2}.
        \end{align*}
        \item \begin{align*}
            t = |I_H| = |(l_1(\alpha_1) \cap l_2(\alpha_2)) \cap H| \le |(l_1(\alpha_1) \cup l_2(\alpha_2)) \cap H| = |U_H| = s.
        \end{align*}
        \item \begin{align*}
            t &= |I_H|\\
            &= \sum_{\delta \in CC(I_H)} |\delta|\\
            &\le \sum_{\delta \in CC(I_H)} |V(\delta) - K|
            \intertext{since when $2D < L$, we have $|l_1(\alpha_1) \cap l_2(\alpha_2)| < L = |H|$, and thus no components of $I_H = (l_1(\alpha_1) \cap l_2(\alpha_2)) \cap H$ contains both vertex $1$ and $2$, and every component $\delta \in C((l_1(\alpha_1) \cap l_2(\alpha_2)) \cap H)$ satisfies $|\delta| \le |V(\delta) - K|$.}
            &= |V(I_H) - K|\\
            &= |V((l_1(\alpha_1) \cap l_2(\alpha_2)) \cap H) - K|\\
            &\le |V(l_1(\alpha_1) \cap H) \cap V(l_2(\alpha_2) \cap H) - K|\\
            &= v_t
        \end{align*}
        \item $c_s \le s$ is clear since the number of connected components of $U_H$ is at most $|U_H| = s$. On the other hand, $c_s \ge 0$ is also easy to verify, since no components of $U_H = (l_1(\alpha_1) \cup l_2(\alpha_2)) \cap H$ contains both vertex $1$ and $2$ when $2D < L$, and $c_s = cc(U_H) - c_1 - c_2$ is the number of connected components of $U_H$ that does not contain vertex $1$ or $2$.
        \item Since $\alpha_1 \ne \alpha_2$, we know that $l_1(\alpha_1) \triangle l_2(\alpha_2) \ne \emptyset$. Since $l_1(\alpha_1) \triangle l_2(\alpha_2) \subseteq H$, we have $s = |U_H| = |(l_1(\alpha_1) \cup l_2(\alpha_2))\cap H| \ge 1$.
    \end{itemize}

    Since $H$ is a path of length $L$ from $1$ to $2$, we know that $U_H = (l_1(\alpha_1) \cup l_2(\alpha_2)) \cap H$ is a union of vertex-disjoint paths. To enumerate the number of $U_H \subseteq H$ with the parameters in $W$ when $c_s \ge 0$, we may enumerate the start vertex of each connected component of $U_H$ excluding the connected component containing $1$ or $2$, for a total of $c_s$ components, and the length of each of the $c_s$ components together with the length of $c_1 + c_2$ components containing vertex $1$ or $2$. Enumerating the start vertices of each component gives at most $\binom{L}{c_s}$ choices. Enumerating the lengths of the components gives $\binom{s - 1}{c_s + c_1 + c_2 - 1}$ choices by a standard stars-and-bars argument, since the set of lengths of the components of $U_H$ have to be a solution to the following system:
    \begin{align*}
        \sum_{i = 1}^{c_s} x_i + c_1 \cdot y + c_2 \cdot z &= |U_H| = s,\\
        x_i, y, z &\in \ZZ^+,
    \end{align*}
    where $x_i$ are the lengths of components not containing vertices $1$ or $2$, and $y$ is the length of the component containing vertex $1$ if $c_1 = 1$, and $z$ is the length of the component containing vertex $2$ if $c_2 = 1$. We note that $\binom{s-1}{c_s + c_1 + c_2 - 1}$ is well-defined since $s\ge 1$, and $c_s + c_1 + c_2 = cc(U_H) \ge 1$. Combining $\binom{L}{c_s}$ choices for the start vertices of the components and $\binom{s-1}{c_s + c_1 + c_2 - 1}$ choices for the lengths of the components, we conclude that $L^{c_s} 2^{s-1}$ is a valid upper bound on the number of choices of $U_H$ given the parameters in $W$.

    Now for a fixed $U_H$, we proceed to enumerate the number of $(l_1, l_2) \in \LL_{\alpha_1, \alpha_2}(W)$ that satisfies the parameters in $W$ and $(l_1(\alpha_1) \cup l_2(\alpha_2)) \cap H = U_H$:
    \begin{itemize}
        \item We first enumerate the pair of $S_1, S_2$ such that $S_1 = l_1(\alpha_1) \cap H$ and $S_2 = l_2(\alpha_2) \cap H$. Since $S_1 \cup S_2 = U_H$, we know that every edge of $U_H$ has to be part of $S_1$, $S_2$, or both. This gives at most $3^{|U_H|} = 3^s$ choices of $(S_1, S_2)$.
        \item We then enumerate the pair of $V_{S,1}, V_{S,2}$ such that $V_{S,1} = V(l_1(\alpha_1)) \cap V(U_H) - K$ and $V_{S,2} = V(l_2(\alpha_2)) \cap V(U_H) - K$. Since $V_{S,1} \cup V_{S,2} = V( (l_1(\alpha_1) \cup l_2(\alpha_2))\cap H ) - K = V(U_H) - K$, every vertex of $V(U_H) - K$ has to be part of $V_{S,1}$, $V_{S,2}$, or both. This gives at most $3^{|V(U_H) - K|} = 3^{v_s}$ choices of $(V_{S,1}, V_{S,2})$.
        \item Next for fixed $S_1$, $S_2$, $V_{S,1}$ and $V_{S,2}$, we proceed to enumerate the number of $(l_1, l_2) \in \LL_{\alpha_1, \alpha_2}(W)$ that satisfies the parameters in $W$ and $S_1 = l_1(\alpha_1) \cap H$ and $S_2 = l_2(\alpha_2) \cap H$, $V_{S,1} = V(l_1(\alpha_1)) \cap V(U_H) - K$ and $V_{S,2} = V(l_2(\alpha_2)) \cap V(U_H) - K$.
        
        Notice that since $l_1(\alpha_1) \triangle l_2(\alpha_2) \subseteq H$, we have $l_1(\alpha_1) \cup l_2(\alpha_2) - H \subseteq l_1(\alpha_1) \cap l_2(\alpha_2)$ and consequently it must be the case that every vertex in $V(l_1(\alpha_1) \cup l_2(\alpha_2)) - V(U_H) - K = V(l_1(\alpha_1) \cup l_2(\alpha_2)) - V((l_1(\alpha_1) \cup l_2(\alpha_2)) \cap H) - K$ is shared by both $V(l_1(\alpha_1))$ and $V(l_2(\alpha_2))$. Therefore, if we denote $v_o = |V(l_1(\alpha_1) \cup l_2(\alpha_2)) - V(U_H) - K|$, we have
    \begin{align*}
        v(\alpha_1) &= |V(l_1(\alpha_1))- K|\\
        &= |V(l_1(\alpha_1)) - V(U_H) - K|+ |V(l_1(\alpha_1)) \cap V(U_H) - K|\\
        &= |V(l_1(\alpha_1) \cup l_2(\alpha_2)) - V(U_H) - K|+ |V(l_1(\alpha_1)) \cap V(U_H) - K|\\
        &= v_o + v_{s,1}.
    \end{align*}
    Similarly,
    \begin{align*}
        v(\alpha_2) &= v_o + v_{s,2}.
    \end{align*}
    Thus, we have
    \begin{align}
        v_o = \frac{v(\alpha_1) + v(\alpha_2) - v_{s,1} - v_{s,2}}{2}.
    \end{align}
    We could then enumerate the set $V_O \subseteq [n]$ such that $V_O = V(l_1(\alpha_1) \cup l_2(\alpha_2)) - V(U_H) - K$, which gives at most $\binom{n}{v_o} = \binom{n}{\frac{v(\alpha_1) + v(\alpha_2) - v_{s,1}-v_{s,2} }{2}}$ choices.

    \item Recall that $l_1(\alpha_1) \triangle l_2(\alpha_2) \subseteq H$. Therefore, $l_2(\alpha_2) - l_1(\alpha_1) \subseteq (l_1(\alpha_1) \cup l_2(\alpha_2)) \cap H$, and
    \begin{align*}
        &\quad V(l_1(\alpha_1) \cup l_2(\alpha_2)) - V(U_H)\\
        &= V(l_1(\alpha_1) \cup l_2(\alpha_2)) - V((l_1(\alpha_1) \cup l_2(\alpha_2)) \cap H)\\
        &= V(l_1(\alpha_1)) \cup V(l_2(\alpha_2) - l_1(\alpha_1)) - V((l_1(\alpha_1) \cup l_2(\alpha_2)) \cap H)\\
        &= V(l_1(\alpha_1)) -  V((l_1(\alpha_1) \cup l_2(\alpha_2)) \cap H)\\
        &= V(l_1(\alpha_1)) -  V(U_H).
    \end{align*} Now for fixed $S_1, S_2$, $V_{S,1}$, $V_{S,2}$ and $V_O$, we enumerate the number of $l_1(\alpha_1)$ that satisfies
    \begin{align*}
         V_O \sqcup V_{S,1} &= \left[V(l_1(\alpha_1) \cup l_2(\alpha_2)) - V(U_H) - K\right] \sqcup \left[V(l_1(\alpha_1)) \cap V(U_H) - K\right] \\
         &= \left[V(l_1(\alpha_1)) - V(U_H) - K\right] \sqcup \left[V(l_1(\alpha_1)) \cap V(U_H) - K\right]\\
         &= V(l_1(\alpha_1)) - K.
    \end{align*} In other words, when $V_O$ and $V_{S,1}$ are fixed, the vertex set $V(l_1(\alpha_1))$ of the image $l_1(\alpha_1)$ is fixed, and the number of $l_1(\alpha_1)$ is bounded by
    \[\frac{v(\alpha_1)!}{|\text{Aut}(\alpha_1)|} \le \frac{v(\alpha_1)^{v_{s,1}} v_o!}{|\text{Aut}(\alpha_1)|},\]
    since $v(\alpha_1)!$ is the number of injective maps from $V(\alpha_1) - K$ to $V_O \sqcup V_{S,1}$, and taking into account the automorphism group of $\alpha_1$, the number of distinct $l_1(\alpha_1)$ given $V_{S,1}, V_O$ is bounded by $\frac{v(\alpha_1)!}{|\text{Aut}(\alpha_1)|}$. In the inequality, we used that $v(\alpha_1) = v_o + v_{s,1}$, as we have shown above that $V_O \sqcup V_{S,1} = V(l_1(\alpha_1)) - K$.
    \item Furthermore, we note that once $l_1(\alpha_1), S_1, S_2$ are determined, $l_2(\alpha_2)$ is also determined, since from the condition $l_1(\alpha_1) \triangle l_2(\alpha_2) \subseteq H$ and that $S_1 = l_1(\alpha_1) \cap H, S_2 = l_2(\alpha_2) \cap H$, we know $l_1(\alpha_1) - S_1 = l_2(\alpha_2) - S_2 = (l_1(\alpha_1) \cap l_2(\alpha_2)) - H$. As a result, $\alpha_2$ is also determined by $l_1(\alpha_1), S_1, S_2$. If 
    \[|\hat{f}_{\alpha_2}|\left(\frac{n}{2D}\right)^{\frac{v(\alpha_2)}{2}}|\text{Aut}(\alpha_2)| > |\hat{f}_{\alpha_1}|\left(\frac{n}{2D}\right)^{\frac{v(\alpha_1)}{2}}|\text{Aut}(\alpha_1)|,\]
    then this $\alpha_2$ does not participate in the sum and we abort the current enumeration and backtrack. Otherwise, the following holds:
    \[|\hat{f}_{\alpha_2}|\left(\frac{n}{2D}\right)^{\frac{v(\alpha_2)}{2}}|\text{Aut}(\alpha_2)| \le |\hat{f}_{\alpha_1}|\left(\frac{n}{2D}\right)^{\frac{v(\alpha_1)}{2}}|\text{Aut}(\alpha_1)|.\]
    
    \item For fixed $l_1(\alpha_1), l_2(\alpha_2)$, we finish enumerating the number of $(l_1, l_2) \in \LL_{\alpha_1, \alpha_2}(W)$ that achieve the images $l_1(\alpha_1)$ and $l_2(\alpha_2)$. This number is bounded by \[|\text{Aut}(\alpha_1)||\text{Aut}(\alpha_2)| \le \frac{|\hat{f}_{\alpha_1}|}{|\hat{f}_{\alpha_2}|}\left(\frac{n}{2D}\right)^{\frac{v(\alpha_1) - v(\alpha_2)}{2}}|\text{Aut}(\alpha_1)|^2.\]
    Thus, for fixed $l_1(\alpha_1), l_2(\alpha_2)$ and the fixed choice of parameters in $W$, the contribution to the sum is at most
    \begin{align*}
        &\quad 2\left|\hat{f}_{\alpha_1} \right|\left|\hat{f}_{\alpha_2} \right| |\text{Aut}(\alpha_1)||\text{Aut}(\alpha_2)| \left(\sqrt{\frac{1-q}{q}}\right)^{s + t}\\
        &\le 2\left|\hat{f}_{\alpha_1} \right|\left|\hat{f}_{\alpha_2} \right| \frac{|\hat{f}_{\alpha_1}|}{|\hat{f}_{\alpha_2}|}\left(\frac{n}{2D}\right)^{\frac{v(\alpha_1) - v(\alpha_2)}{2}}|\text{Aut}(\alpha_1)|^2 \left(\sqrt{\frac{1-q}{q}}\right)^{s + t}\\
        &= 2 \hat{f}_{\alpha_1}^2 \left(\frac{n}{2D}\right)^{\frac{v(\alpha_1) - v(\alpha_2)}{2}}|\text{Aut}(\alpha_1)|^2 \left(\sqrt{\frac{1-q}{q}}\right)^{s + t}
    \end{align*}
    \end{itemize}

    In total, combining the number of choices for $U_H, S_1, S_2, V_{S,1}, V_{S,2}, V_O, l_1(\alpha_1), l_2(\alpha_2)$, we conclude that the total contribution to the sum from all  $(\alpha_2, (l_1,l_2) \in \LL_{\alpha_1,\alpha_2}(W))$ is at most
    \begin{align*}
        &\quad \underbrace{L^{c_s} 2^{s-1}}_{U_H}  \underbrace{3^s}_{S_1, S_2} \underbrace{3^{v_s}}_{V_{S,1},V_{S,2}} \underbrace{\binom{n}{\frac{v(\alpha_1) + v(\alpha_2) - v_{s,1} - v_{s,2}}{2}}}_{V_O} \underbrace{\frac{v(\alpha_1)^{v_{s,1}} v_o!}{|\text{Aut}(\alpha_1)|}}_{l_1(\alpha_1),l_2(\alpha_2)} \\
        &\quad \cdot \underbrace{2 \hat{f}_{\alpha_1}^2 \left(\frac{n}{2D}\right)^{\frac{v(\alpha_1) - v(\alpha_2)}{2}}|\text{Aut}(\alpha_1)|^2 \left(\sqrt{\frac{1-q}{q}}\right)^{s + t}}_{\substack{\text{contribution given fixed }\\ U_H, S_1, S_2, V_{S,1}, V_{S,2}, V_O, l_1(\alpha_1),l_2(\alpha_2)}}\\
        &\le 2L^{c_s} 2^{s-1} 3^s 3^{v_s} \left(\binom{n}{\frac{v(\alpha_1) + v(\alpha_2) - v_{s,1} - v_{s,2}}{2}} v_o! \right) v(\alpha_1)^{v_{s,1}} \\
        &\quad \cdot \hat{f}_{\alpha_1}^2\left(\frac{n}{2D}\right)^{\frac{v(\alpha_1) - v(\alpha_2)}{2}} |\text{Aut}(\alpha_1)|\left(\sqrt{\frac{1-q}{q}}\right)^{s + t}\\
        &\le L^{c_s} 6^s 3^{v_s} n^{\frac{v(\alpha_1) + v(\alpha_2) - v_{s,1} - v_{s,2}}{2}}v(\alpha_1)^{v_{s,1}}\hat{f}_{\alpha_1}^2\left(\frac{n}{2D}\right)^{\frac{v(\alpha_1) - v(\alpha_2)}{2}}|\text{Aut}(\alpha_1)|\left(\sqrt{\frac{1-q}{q}}\right)^{s + t},
        \intertext{where we used that $v_o = \frac{v(\alpha_1) + v(\alpha_2) - v_{s,1} - v_{s,2}}{2}$,}
        &= L^{c_s} 6^s 3^{v_s} n^{v(\alpha_1) - \frac{v_{s,1} + v_{s,2}}{2}}v(\alpha_1)^{v_{s,1}} (2D)^{\frac{v(\alpha_2) - v(\alpha_1)}{2}}\hat{f}_{\alpha_1}^2|\text{Aut}(\alpha_1)|\left(\sqrt{\frac{1-q}{q}}\right)^{s + t}\\
        &\le \hat{f}_{\alpha_1}^2 n^{v(\alpha_1)} |\text{Aut}(\alpha_1)|\cdot L^{c_s} 6^s 3^{v_s} n^{-\frac{v_{s,1} + v_{s,2}}{2}} (2D)^{v_{s,1}} (2D)^{\frac{v_{s,2} - v_{s,1}}{2}}\left(\sqrt{\frac{1-q}{q}}\right)^{s + t}
        \intertext{since $v(\alpha_1) = |V(l_1(\alpha_1)) - K| \le 2D$ and that $v(\alpha_1) = v_o + v_{s,1}, v(\alpha_2) = v_o + v_{s,2}$,}
        &= \hat{f}_{\alpha_1}^2 n^{v(\alpha_1)} |\text{Aut}(\alpha_1)|\cdot L^{c_s} 6^s 3^{v_s} n^{-\frac{v_{s,1} + v_{s,2}}{2}} (2D)^{\frac{v_{s,1} + v_{s,2}}{2}}\left(\sqrt{\frac{1-q}{q}}\right)^{s + t}\\
        &= \hat{f}_{\alpha_1}^2 n^{v(\alpha_1)} |\text{Aut}(\alpha_1)|\cdot \left[\left(6\sqrt{\frac{1-q}{q}}\right)^s \left(3\sqrt{\frac{2D}{n}}\right)^{v_s} L^{c_s}\right]\cdot \left[\left(\sqrt{\frac{1-q}{q}}\right)^t \left(\sqrt{\frac{2D}{n}}\right)^{v_t}\right]\\
        &\quad \cdot \left[\left(\sqrt{\frac{2D}{n}}\right)^{v_{s,1} + v_{s,2} - v_s - v_t}\right]
        \intertext{where we note that $v_{s,1} + v_{s,2} \ge v_s + v_t$. We further use $v_s = s + c_s, v_t \ge t$ and get}
        &= \hat{f}_{\alpha_1}^2 n^{v(\alpha_1)} |\text{Aut}(\alpha_1)|\cdot \left[\left(18\sqrt{\frac{2D}{n}\cdot\frac{1-q}{q}}\right)^s \left(3L\sqrt{\frac{2D}{n}}\right)^{c_s} \right]\cdot \left[\left(\sqrt{\frac{2D}{n}\cdot\frac{1-q}{q}}\right)^t \left(\sqrt{\frac{2D}{n}}\right)^{v_t - t}\right]\\
        &\quad \cdot \left[\left(\sqrt{\frac{2D}{n}}\right)^{v_{s,1} + v_{s,2} - v_s - v_t}\right].
    \end{align*}

    Finally, for each fixed $\alpha_1$, we sum over all possible choices of parameters and get
    \begin{align*}
        &\quad 2\sum_{\alpha_1 \in \GG_{\le D}^K} \left|\hat{f}_{\alpha_1}\right| \sum_{\substack{W = (s,t,v_s,v_t,\\ v_{s,1}, v_{s,2},c_1,c_2,c_s) } } \sum_{\substack{\alpha_2 \in \GG_{\le D}^K:\\ \alpha_1 \ne \alpha_2,\\ |\hat{f}_{\alpha_2}|\left(\frac{n}{2D}\right)^{\frac{v(\alpha_2)}{2}}|\text{Aut}(\alpha_2)|\\
        \le |\hat{f}_{\alpha_1}|\left(\frac{n}{2D}\right)^{\frac{v(\alpha_1)}{2}}|\text{Aut}(\alpha_1)| } } \left|\hat{f}_{\alpha_2}\right|  \sum_{(l_1, l_2) \in \LL_{\alpha_1, \alpha_2}(W)} \left(\sqrt{\frac{1-q}{q}}\right)^{s + t}\\
        &\le \sum_{\alpha_1 \in \GG_{\le D}^K}\hat{f}_{\alpha_1}^2 n^{v(\alpha_1)} |\text{Aut}(\alpha_1)| \cdot \sum_{\substack{W = (s,t,v_s,v_t,\\ v_{s,1}, v_{s,2},c_1,c_2,c_s) } } \left[\left(18\sqrt{\frac{2D}{n}\cdot\frac{1-q}{q}}\right)^s \left(3L\sqrt{\frac{2D}{n}}\right)^{c_s} \right]\\
        &\quad \cdot \left[\left(\sqrt{\frac{2D}{n}\cdot\frac{1-q}{q}}\right)^t \left(\sqrt{\frac{2D}{n}}\right)^{v_t - t}\right] \cdot \left[\left(\sqrt{\frac{2D}{n}}\right)^{v_{s,1} + v_{s,2} - v_s - v_t}\right]\\
        &\le \sum_{\alpha_1 \in \GG_{\le D}^K}\hat{f}_{\alpha_1}^2 n^{v(\alpha_1)} |\text{Aut}(\alpha_1)| \cdot \sum_{s \ge 1} \left(18\sqrt{\frac{2D}{n}\cdot\frac{1-q}{q}}\right)^s \sum_{c_s = 0}^s \left(3L\sqrt{\frac{2D}{n}}\right)^{c_s} \\
        &\quad \cdot  \sum_{t = 0}^s \left(\sqrt{\frac{2D}{n}\cdot\frac{1-q}{q}}\right)^t \sum_{v_t \ge t} \left(\sqrt{\frac{2D}{n}}\right)^{v_t - t} \sum_{i \ge 0} \left(\sqrt{\frac{2D}{n}}\right)^{i} \\
        &\quad \cdot \#\{\text{feasible } W, \text{ given } v(\alpha_1), s, c_s, t, v_t, \text{ and } v_{s,1} + v_{s,2} - v_s - v_t = i\},
    \end{align*}
    where in the last step we use that $v_{s,1} + v_{s,2} \ge v_s + v_t$ and that $s = |U_H| = |(l_1(\alpha_1) \cup l_2(\alpha_2)) \cap H| \ge |l_1(\alpha_1) \triangle l_2(\alpha_2)| \ge 1$ for $\alpha_1 \ne \alpha_2$. By feasible $W$ given $v(\alpha_1), s, c_s, t, v_t$ and $v_{s,1} + v_{s,2} - v_s - v_t = i$, we mean the choice of parameters in $W = (s,t,v_s,v_t, v_{s,1}, v_{s,2},c_1,c_2,c_s)$ for which there exists $\alpha_2$ and $(l_1, l_2) \in \LL_{\alpha_1, \alpha_2}(W)$. We have
    \begin{align*}
        c_1, c_2 &\in \{0,1\}\\
        v_s &= s + c_s\\
        v_{s,1} + v_{s,2} &= v_s + v_t + i.
    \end{align*}
    Thus, $v_s$ is determined by $s$ and $c_s$; there are at most $4$ choices for $c_1, c_2$; there are at most $v_s + v_t + i + 1$ choices for $v_{s,1}, v_{s,2}$. In total, the number of feasible choices of $W$ given $v(\alpha_1), s, c_s, t, v_t$ and $i$ is at most $4(v_s + v_t + i + 1) = 4(s + c_s + v_t + i + 1)$. We can thus bound

    \begin{align*}
        &\quad 2\sum_{\alpha_1 \in \GG_{\le D}^K} \left|\hat{f}_{\alpha_1}\right| \sum_{\substack{W = (s,t,v_s,v_t,\\ v_{s,1}, v_{s,2},c_1,c_2,c_s) } } \sum_{\substack{\alpha_2 \in \GG_{\le D}^K:\\ \alpha_1 \ne \alpha_2,\\ |\hat{f}_{\alpha_2}|\left(\frac{n}{2D}\right)^{\frac{v(\alpha_2)}{2}}|\text{Aut}(\alpha_2)|\\
        \le |\hat{f}_{\alpha_1}|\left(\frac{n}{2D}\right)^{\frac{v(\alpha_1)}{2}}|\text{Aut}(\alpha_1)| } } \left|\hat{f}_{\alpha_2}\right|  \sum_{(l_1, l_2) \in \LL_{\alpha_1, \alpha_2}(W)} \left(\sqrt{\frac{1-q}{q}}\right)^{s + t}\\
        &\le \sum_{\alpha_1 \in \GG_{\le D}^K}\hat{f}_{\alpha_1}^2 n^{v(\alpha_1)} |\text{Aut}(\alpha_1)| \cdot \sum_{s \ge 1} \left(18\sqrt{\frac{2D}{n}\cdot\frac{1-q}{q}}\right)^s \sum_{c_s = 0}^s \left(3L\sqrt{\frac{2D}{n}}\right)^{c_s} \sum_{t = 0}^s \left(\sqrt{\frac{2D}{n}\cdot\frac{1-q}{q}}\right)^t\\
        &\quad \cdot  \sum_{v_t \ge t} \left(\sqrt{\frac{2D}{n}}\right)^{v_t - t} \sum_{i \ge 0} \left(\sqrt{\frac{2D}{n}}\right)^i \cdot 4(s + c_s + v_t + i + 1),
    \end{align*}
    which is $O\left(\sqrt{\frac{D(1-q)}{nq}}\cdot \sum_{\alpha \in \GG_{\le D}^K}\hat{f}_{\alpha}^2 n^{v(\alpha)} |\text{Aut}(\alpha)|\right)$ as long as $18\sqrt{\frac{2D}{n}\cdot\frac{1-q}{q}} \le 1 - \eps$ and $3L \sqrt{\frac{2D}{n}} \le 1 - \eps$ for some constant $\eps > 0$.

    Thus, under the assumption that $\frac{D(1-q)}{nq} = o(1)$, $2D < L$, and $L^2 \frac{D}{n} = o(1)$, we have
    \begin{align*}
        &\quad \EE[f(G)^2]\\
        &= \sum_{\alpha\in \GG_{\le D}^K} \hat{f}_{\alpha}^2 \sum_{l_1, l_2 \in \LL_{\alpha}} \EE\left[\allone\{l_1(\alpha) \triangle l_2(\alpha) \subseteq \bH\} \left(\frac{1-q}{q}\right)^{\frac{|l_1(\alpha) \cap \bH|}{2} + \frac{|l_2(\alpha) \cap \bH|}{2}}\right]\\
        &\quad + \sum_{\substack{\alpha_1, \alpha_2\in \GG_{\le D}^K:\\ \alpha_1 \ne \alpha_2}} \hat{f}_{\alpha_1}\hat{f}_{\alpha_2} \sum_{l_1 \in \LL_{\alpha_1}, l_2 \in \LL_{\alpha_2}} \EE\left[\allone\{l_1(\alpha_1) \triangle l_2(\alpha_2) \subseteq \bH\} \left(\frac{1-q}{q}\right)^{\frac{|l_1(\alpha) \cap \bH|}{2} + \frac{|l_2(\alpha) \cap \bH|}{2}}\right]\\
        &\le \sum_{\alpha \in \GG_{\le D}^K} \hat{f}_{\alpha}^2 \sum_{l_1, l_2 \in \LL_{\alpha}} \EE\left[\allone\{l_1(\alpha) \triangle l_2(\alpha) \subseteq \bH\} \left(\frac{1-q}{q}\right)^{\frac{|l_1(\alpha) \cap \bH|}{2}+\frac{|l_2(\alpha) \cap \bH|}{2}}\right]\\
        &\quad + O\left(\sqrt{\frac{D(1-q)}{nq}}\right) \cdot \sum_{\alpha \in \GG_{\le D}^K}\hat{f}_{\alpha}^2 n^{v(\alpha)} |\text{Aut}(\alpha)|\\
        &\le \left(1 + O\left(\sqrt{\frac{D(1-q)}{nq}}\right)\right)\cdot \sum_{\alpha\in \GG_{\le D}^K} \hat{f}_{\alpha}^2 \sum_{l_1, l_2 \in \LL_{\alpha}} \EE\left[\allone\{l_1(\alpha) \triangle l_2(\alpha) \subseteq \bH\} \left(\frac{1-q}{q}\right)^{\frac{|l_1(\alpha) \cap \bH|}{2}+\frac{|l_2(\alpha) \cap \bH|}{2}}\right],
    \end{align*}
    and we also have
    \begin{align*}
        &\quad \EE[f(G)^2]\\
        &\ge \left(1 - O\left(\sqrt{\frac{D(1-q)}{nq}}\right)\right)\cdot \sum_{\alpha\in \GG_{\le D}^K} \hat{f}_{\alpha}^2 \sum_{l_1, l_2 \in \LL_{\alpha}} \EE\left[\allone\{l_1(\alpha) \triangle l_2(\alpha) \subseteq \bH\} \left(\frac{1-q}{q}\right)^{\frac{|l_1(\alpha) \cap \bH|}{2}+\frac{|l_2(\alpha) \cap \bH|}{2}}\right]
    \end{align*}

    Similarly, under the same conditions, we may lower bound the expectation of the cross term:
    \begin{align*}
        &\quad \EE[f(G)f(T_{\rho}(G))]\\
        &= \sum_{\alpha_1, \alpha_2 \in \GG_{\le D}^K} \hat{f}_{\alpha_1} \hat{f}_{\alpha_2}(1-\rho)^{|\alpha_2|} \sum_{l_1 \in \LL_{\alpha_1}, l_2 \in \LL_{\alpha_2}} \EE\left[\chi_{l_1(\alpha_1)}(G)\chi_{l_2(\alpha_2)}(T_{\rho}(G))\right]\\
        &= \sum_{\alpha_1, \alpha_2 \in \GG_{\le D}^K} \hat{f}_{\alpha_1} \hat{f}_{\alpha_2}(1-\rho)^{|\alpha_2|} \sum_{l_1 \in \LL_{\alpha_1}, l_2 \in \LL_{\alpha_2}} \EE\left[\allone\{l_1(\alpha_1) \triangle l_2(\alpha_2) \subseteq \bH\} \left(\frac{1-q}{q}\right)^{\frac{|l_1(\alpha_1) \cap \bH|}{2}+\frac{|l_2(\alpha) \cap \bH|}{2}}\right]\\
        &= \sum_{\alpha \in \GG_{\le D}^K} \hat{f}_{\alpha}^2 (1-\rho)^{|\alpha|} \sum_{l_1, l_2 \in \LL_{\alpha}} \EE\left[\allone\{l_1(\alpha) \triangle l_2(\alpha) \subseteq \bH\} \left(\frac{1-q}{q}\right)^{\frac{|l_1(\alpha) \cap \bH|}{2}+\frac{|l_2(\alpha) \cap \bH|}{2}}\right]\\
        &\quad + \sum_{\substack{\alpha_1, \alpha_2 \in \GG_{\le D}^K:\\ \alpha_1 \ne \alpha_2}} \hat{f}_{\alpha_1} \hat{f}_{\alpha_2} (1 - \rho)^{|\alpha_2|} \sum_{l_1, l_2 \in \LL_{\alpha}} \EE\left[\allone\{l_1(\alpha) \triangle l_2(\alpha) \subseteq \bH\} \left(\frac{1-q}{q}\right)^{\frac{|l_1(\alpha) \cap \bH|}{2}+\frac{|l_2(\alpha) \cap \bH|}{2}}\right]\\
        &\ge (1 - \rho)^D \sum_{\alpha \in \GG_{\le D}^K} \hat{f}_{\alpha}^2  \sum_{l_1, l_2 \in \LL_{\alpha}} \EE\left[\allone\{l_1(\alpha) \triangle l_2(\alpha) \subseteq \bH\} \left(\frac{1-q}{q}\right)^{\frac{|l_1(\alpha) \cap \bH|}{2}+\frac{|l_2(\alpha) \cap \bH|}{2}}\right]\\
        &\quad - O\left(\sqrt{\frac{D(1-q)}{nq}}\right) \cdot \sum_{\alpha \in \GG_{\le D}^K} \hat{f}_{\alpha}^2 n^{v(\alpha)} |\text{Aut}(\alpha)|\\
        &\ge \left((1- \rho)^D - O\left(\sqrt{\frac{D(1-q)}{nq}}\right)\right) \sum_{\alpha \in \GG_{\le D}^K} \hat{f}_{\alpha}^2\\
        &\quad \cdot \sum_{l_1, l_2 \in \LL_{\alpha}} \EE\left[\allone\{l_1(\alpha) \triangle l_2(\alpha) \subseteq \bH\} \left(\frac{1-q}{q}\right)^{\frac{|l_1(\alpha) \cap \bH|}{2}+\frac{|l_2(\alpha) \cap \bH|}{2}}\right]\\
        &\ge \left((1 - \rho)^D - O\left(\sqrt{\frac{D(1-q)}{nq}}\right)\right) \EE[f(G)^2]. \stepcounter{equation}\tag{\theequation}\label{ineq:second-moment-cross}
    \end{align*}
    Under the same conditions and that $q \le 1/2$, we may upper bound the second moment under the noise operator as
    \begin{align*}
        &\quad \EE[f(T_{\rho}(G))^2]\\
        &= \sum_{\alpha_1, \alpha_2 \in \GG_{\le D}^K} \hat{f}_{\alpha_1} \hat{f}_{\alpha_2}(1-\rho)^{|\alpha_2|} \sum_{l_1 \in \LL_{\alpha_1}, l_2 \in \LL_{\alpha_2}} \EE\left[\chi_{l_1(\alpha_1)}(T_{\rho}(G))\chi_{l_2(\alpha_2)}(T_{\rho}(G))\right]\\
        &= \sum_{\alpha_1, \alpha_2 \in \GG_{\le D}^K} \hat{f}_{\alpha_1} \hat{f}_{\alpha_2} \sum_{l_1 \in \LL_{\alpha_1}, l_2 \in \LL_{\alpha_2}}(1-\rho)^{|l_1(\alpha_1) \triangle l_2(\alpha_2)|} \\
        &\quad \cdot \EE\left[\allone\{l_1(\alpha_1) \triangle l_2(\alpha_2) \subseteq \bH\} \left(\frac{1-q}{q}\right)^{\frac{|l_1(\alpha_1)  \triangle l_2(\alpha_2)|}{2}}\left(\rho + (1 - \rho)\left( \frac{1-q}{q}\right)\right)^{|l_1(\alpha_1) \cap l_2(\alpha_2) \cap \bH|}\right]\\
        &= \sum_{\alpha\in \GG_{\le D}^K} \hat{f}_{\alpha}^2 \sum_{l_1 \in \LL_{\alpha}, l_2 \in \LL_{\alpha}}(1-\rho)^{|l_1(\alpha) \triangle l_2(\alpha)|} \\
        &\quad \cdot \EE\left[\allone\{l_1(\alpha) \triangle l_2(\alpha) \subseteq \bH\} \left(\frac{1-q}{q}\right)^{\frac{|l_1(\alpha)  \triangle l_2(\alpha)|}{2}}\left(\rho + (1 - \rho)\left( \frac{1-q}{q}\right)\right)^{|l_1(\alpha) \cap l_2(\alpha) \cap \bH|}\right]\\
        &\quad + \sum_{\substack{\alpha_1, \alpha_2\in \GG_{\le D}^K:\\ \alpha_1 \ne \alpha_2}} \hat{f}_{\alpha_1} \hat{f}_{\alpha_2} \sum_{l_1 \in \LL_{\alpha_1}, l_2 \in \LL_{\alpha_2}}(1-\rho)^{|l_1(\alpha_1) \triangle l_2(\alpha_2)|} \\
        &\quad \cdot \EE\left[\allone\{l_1(\alpha_1) \triangle l_2(\alpha_2) \subseteq \bH\} \left(\frac{1-q}{q}\right)^{\frac{|l_1(\alpha_1)  \triangle l_2(\alpha_2)|}{2}}\left(\rho + (1 - \rho)\left( \frac{1-q}{q}\right)\right)^{|l_1(\alpha_1) \cap l_2(\alpha_2) \cap \bH|}\right]\\
        &\le \sum_{\alpha\in \GG_{\le D}^K} \hat{f}_{\alpha}^2 \sum_{l_1 \in \LL_{\alpha}, l_2 \in \LL_{\alpha}}(1-\rho)^{|l_1(\alpha) \triangle l_2(\alpha)|} \\
        &\quad \cdot \EE\left[\allone\{l_1(\alpha) \triangle l_2(\alpha) \subseteq \bH\} \left(\frac{1-q}{q}\right)^{\frac{|l_1(\alpha)\cap \bH|}{2} + \frac{|l_2(\alpha)\cap \bH|}{2}}\right]\\
        &\quad + \sum_{\substack{\alpha_1, \alpha_2 \in \GG_{\le D}^K:\\ \alpha_1 \ne \alpha_2}} |\hat{f}_{\alpha_1} \hat{f}_{\alpha_2}| \sum_{l_1 \in \LL_{\alpha_1}, l_2 \in \LL_{\alpha_2}}(1-\rho)^{|l_1(\alpha_1) \triangle l_2(\alpha_2)|} \\
        &\quad \cdot \EE\left[\allone\{l_1(\alpha_1) \triangle l_2(\alpha_2) \subseteq \bH\} \left(\frac{1-q}{q}\right)^{\frac{|l_1(\alpha_1)\cap \bH|}{2}+\frac{|l_2(\alpha)\cap \bH|}{2}}\right]
        \intertext{where we used that $\rho + (1-\rho)\left(\frac{1-q}{q}\right) \le \frac{1-q}{q}$ if $q \le \frac{1}{2}$,}
        &\le \sum_{\alpha\in \GG_{\le D}^K} \hat{f}_{\alpha}^2 \sum_{l_1, l_2 \in \LL_{\alpha}} \EE\left[\allone\{l_1(\alpha) \triangle l_2(\alpha) \subseteq \bH\} \left(\frac{1-q}{q}\right)^{\frac{|l_1(\alpha) \cap \bH|}{2}+\frac{|l_2(\alpha)\cap \bH|}{2}}\right]\\
        &\quad + O\left(\sqrt{\frac{D(1-q)}{nq}}\right) \cdot \sum_{\alpha\in \GG_{\le D}^K}\hat{f}_{\alpha}^2 n^{v(\alpha)} |\text{Aut}(\alpha)|\\
        &\le \left(1 + O\left(\sqrt{\frac{D(1-q)}{nq}}\right)\right)\sum_{\alpha\in \GG_{\le D}^K} \hat{f}_{\alpha}^2 \sum_{l_1, l_2 \in \LL_{\alpha}} \EE\left[\allone\{l_1(\alpha) \triangle l_2(\alpha) \subseteq \bH\} \left(\frac{1-q}{q}\right)^{\frac{|l_1(\alpha) \cap \bH|}{2}+\frac{|l_2(\alpha)\cap \bH|}{2}}\right]\\
        &\le \left(1 + O\left(\sqrt{\frac{D(1-q)}{nq}}\right)\right)\EE[f(G)^2]. \stepcounter{equation}\tag{\theequation}\label{ineq:second-moment-noisy}
    \end{align*}

    Combining \eqref{ineq:second-moment-cross} and \eqref{ineq:second-moment-noisy}, we get
\begin{align*}
    &\quad \EE[(f(G) - f(T_{\rho}(G)))^2]\\
    &= \EE[f(G)^2] + \EE[f(T_{\rho}(g))^2] - 2\EE[f(G)f(T_{\rho}(G))]\\
    &\le 2\left(1  - (1 - \rho)^D + O\left(\sqrt{\frac{D(1-q)}{nq}}\right)\right) \EE[f(G)^2].
\end{align*}
\end{proof}

\section{Proofs for Random Linear Code}\label{sec:RLC}

\subsection{Noisy MMSE: Proof of Theorem~\ref{thm:linear-code-noisy-MMSE}}

In this section, we prove a lower bound for the noisy MMSE of the random linear code problem.

Recall that in the noiseless setting, $A \in \{0,1\}^{m \times n}$ and $x \in \{0,1\}^n$ are sampled uniformly at random, and $y = Ax \text{ mod } 2$. In the noisy version, $T_{\rho}$ resamples each coordinate of $y$ from $\text{Bern}(1/2)$ independently with probability $\rho$.

We restated Theorem~\ref{thm:linear-code-noisy-MMSE} below for convenience.

\begin{theorem}[Restatement of Theorem~\ref{thm:linear-code-noisy-MMSE}]
    For any function $f(n) = \omega(1)$ and $m = n + f(n)$, if $\frac{1}{\left(1 - \frac{3}{2}\rho + \frac{3}{4}\rho^2\right)^m} = \omega\left( 2^{2(m-n)}\right)$, then any algorithm $\mathcal{A}: \{0,1\}^{m \times n} \times \{0,1\}^m \to \R^n$ for RLC has a mean squared error at least
    \begin{align*}
        \E\left[\|\mathcal{A}(A, T_{\rho}(y)) - x\|_2^2\right] \ge \left(\frac{1}{4} - o(1)\right)n.
    \end{align*}

\end{theorem}

\begin{proof}
    For $x \in \{0,1\}^n$, we will use $w(x)$ to denote the weight of $x$, i.e., the number of $1$'s in $x$. Let us denote $\hat{y} = T_{\rho}(y)$. Note that independently for each coordinate of $y$, we have
    \begin{align*}
        \hat{y}_i = \begin{cases}
            y_i & \quad \text{ with probability } 1 - \frac{\rho}{2},\\
            1 - y_i & \quad \text{ with probability } \frac{\rho}{2}.
        \end{cases}
    \end{align*}
    Since the prior of $x$ is uniform over $\{0,1\}^n$, given $A$ and $\hat{y}$, the posterior distribution over $\{0,1\}^n$ is given by
    \begin{align*}
        \Pr(x \vert A, \hat{y}) \propto \left(\frac{\frac{\rho}{2}}{1-\frac{\rho}{2}}\right)^{w(Ax - \hat{y})} = \left(\frac{\rho}{2-\rho}\right)^{w(Ax - \hat{y})}.
    \end{align*}

    To minimize the mean squared error, the optimal algorithm is to output the posterior mean. To minimize the average Hamming distance, the optimal algorithm is to output $x$ that maximizes the marginal posterior probability at each coordinate.

    Let us first understand the marginal posterior probabilities. Let us WLOG look at coordinate $1$. Define $S_0 := \{x \in \{0,1\}^n: x_1 = 0\}$ and $S_1 := \{x \in \{0,1\}^n: x_1 = 1\}$.

    We have
    \begin{align*}
        \Pr(x_1 = 0 \vert A,\hat{y}) &\propto \sum_{x \in S_0} \left(\frac{\rho}{2-\rho}\right)^{w(Ax-\hat{y})},\\
        \Pr(x_1 = 1 \vert A,\hat{y}) &\propto \sum_{x \in S_1} \left(\frac{\rho}{2-\rho}\right)^{w(Ax-\hat{y})}.
    \end{align*}
    For convenience, let us denote $L_0 := \sum_{x \in S_0} \left(\frac{\rho}{2-\rho}\right)^{w(Ax-\hat{y})}$ and $L_1 := \sum_{x \in S_1} \left(\frac{\rho}{2-\rho}\right)^{w(Ax-\hat{y})}$. Then,
    \begin{align*}
        \Pr(x_1 = 0 \vert A,\hat{y}) &= \frac{L_0}{L_0 + L_1}\\
        \Pr(x_1 = 1 \vert A,\hat{y}) &= \frac{L_1}{L_0 + L_1}.
    \end{align*}

    We will show that these marginal posterior probabilities are concentrated around $\frac{1}{2}$ with high probability. Recall that $\hat{y} = T_{\rho}(y) = T_{\rho}(Ax^*)$, where use $x^*$ to denote the true message. Moreover, the application of $T_{\rho}$ to $y$ is equivalent to adding a vector $z \in \{0,1\}^m$ whose entries are i.i.d.~$\text{Bern}(\rho/2)$. Thus, we have
    \begin{align*}
        &\quad \E[L_0]\\
        &= \sum_{x \in S_0} \E_{A,x^*}\left[\left(\frac{\rho}{2-\rho}\right)^{w(Ax - T_{\rho}(Ax^*))}\right] \\
        &= \sum_{x \in S_0} \underset{\substack{A,x^*,\\z \sim \text{Bern}(\rho/2)^{\otimes m}}}{\E}\left[\left(\frac{\rho}{2-\rho}\right)^{w(A(x - x^*)-z)}\right]
        \intertext{note that when $x - x^* \ne \boldsymbol{0}$, $A(x-x^*)$ is a uniformly random vector in $\mathbb{F}_2^m$ since $A \in \{0,1\}^{m \times n}$ is uniformly random. Moreover, the sum of any random vector with a uniformly random vector in $\mathbb{F}_2^m$ is a uniformly random vector. Thus, we get}
        &= \sum_{x \in S_0} 
        \left(\left(1 - 2^{-n}\right)\cdot\underset{u \sim \text{Unif}(\mathbb{F}_2^m)}{\E}\left[\left(\frac{\rho}{2-\rho}\right)^{w(u)}\right] +  2^{-n}\cdot\underset{z \sim \text{Bern}(\rho/2)^{\otimes m}}{\E}\left[\left(\frac{\rho}{2-\rho}\right)^{w(z)}\right] \right)\\
        &= 2^{n-1} \cdot 
        \left(\left(1 - 2^{-n}\right)\cdot \left[\frac{1}{2}\cdot 1 + \frac{1}{2} \cdot \frac{\rho}{2-\rho}\right]^m +  2^{-n}\cdot \left[\left(1 - \frac{\rho}{2}\right)\cdot 1 + \frac{\rho}{2} \cdot \frac{\rho}{2-\rho}\right]^m \right)\\
        &= \frac{1}{2} \cdot \left[(2^n - 1) \left(\frac{1}{2 - \rho}\right)^m + \left(\frac{(2-\rho)^2 + \rho^2}{2(2-\rho)}\right)^m\right]. 
    \end{align*}
    Similarly,
    \begin{align*}
        \E[L_1] = \E[L_0].
    \end{align*}
    We may also compute
    \begin{align*}
        &\quad \E[L_0^2]\\
        &= \sum_{x, x' \in S_0} \underset{\substack{A,x^*,\\z \sim \text{Bern}(\rho/2)^{\otimes m}}}{\E}\left[\left(\frac{\rho}{2-\rho}\right)^{w(A(x - x^*)-z)} \left(\frac{\rho}{2-\rho}\right)^{w(A(x' - x^*)-z)}\right]\\
        &= \sum_{x\in S_0} \underset{\substack{A,x^*,\\z \sim \text{Bern}(\rho/2)^{\otimes m}}}{\E}\left[\left(\frac{\rho}{2-\rho}\right)^{2\cdot w(A(x - x^*)-z)} \right]\\
        &\quad + \sum_{\substack{x,x' \in S_0:\\ x \ne x'}} \underset{\substack{A,x^*,\\z \sim \text{Bern}(\rho/2)^{\otimes m}}}{\E}\left[\left(\frac{\rho}{2-\rho}\right)^{w(A(x - x^*)-z)} \left(\frac{\rho}{2-\rho}\right)^{w(A(x' - x^*)-z)}\right]
        \intertext{Recall that $A(x - x^*) - z$ follows $\text{Unif}(\mathbb{F}_2)^{\otimes m}$ or $\text{Bern}(\rho/2)^{\otimes m}$ depending on whether $x = x^*$ or not. Moreover, it is easy to see that for $x \ne x'$ and fixed $x^*$, the distribution of $A(x - x^*) - z$ and $A(x' - x^*) - z$ are independent for random $A \sim \text{Unif}\left(\{0,1\}^{m \times n}\right)$ and $z \sim \text{Bern}(\rho/2)^{\otimes m}$. Thus,}
        &= 2^{n-1} \cdot 
        \left(\left(1 - 2^{-n}\right)\cdot \left[\frac{1}{2}\cdot 1 + \frac{1}{2} \cdot \left(\frac{\rho}{2-\rho}\right)^2\right]^m +  2^{-n}\cdot \left[\left(1 - \frac{\rho}{2}\right)\cdot 1 + \frac{\rho}{2} \cdot \left(\frac{\rho}{2-\rho}\right)^2\right]^m \right)\\
        &\quad + 2^{n-1}(2^{n-1} - 1) \cdot \Bigg((1 - 2\cdot 2^{-n}) \cdot \left[\frac{1}{2}\cdot 1 + \frac{1}{2} \cdot \frac{\rho}{2-\rho}\right]^{2m} \\
        &\quad+ 2\cdot 2^{-n} \cdot \left[\frac{1}{2}\cdot 1 + \frac{1}{2} \cdot \frac{\rho}{2-\rho}\right]^m \left[\left(1-\frac{\rho}{2}\right)\cdot 1 + \frac{\rho}{2} \cdot \frac{\rho}{2-\rho}\right]^m\Bigg)\\
        &= \frac{1}{2}\cdot \frac{2^n - 1}{2^m(2-\rho)^{2m}}\cdot \left((2-\rho)^2 + \rho^2\right)^m + \frac{1}{2} \cdot \frac{1}{2^m(2-\rho)^{2m}} \cdot \left((2-\rho)^3 + \rho^3\right)^m\\
        &\quad + \frac{(2^{n-1} - 1)^2}{(2-\rho)^{2m}} + \frac{2^{n-1} - 1}{2^m (2-\rho)^{2m}} \cdot \left((2-\rho)^2 + \rho^2\right)^m\\
        &= \frac{(2^{n-1} - 1)^2}{(2-\rho)^{2m}} + \frac{2^{n} - \frac{3}{2}}{2^m (2-\rho)^{2m}} \cdot \left((2-\rho)^2 + \rho^2\right)^m + \frac{1}{2} \cdot \frac{1}{2^m(2-\rho)^{2m}} \cdot \left((2-\rho)^3 + \rho^3\right)^m.
    \end{align*}

    To show concentration of $L_0$ around $\EE[L_0]$, we will show $\Var[L_0] = o(\EE[L_0]^2)$, or equivalently $\EE[L_0^2] = (1 + o(1))\EE[L_0]^2$. We have
    \begin{align*}
        &\quad \frac{\E[L_0^2]}{\E[L_0]^2}\\
        &= \frac{\frac{(2^{n-1} - 1)^2}{(2-\rho)^{2m}} + \frac{2^{n} - \frac{3}{2}}{2^m (2-\rho)^{2m}} \cdot \left((2-\rho)^2 + \rho^2\right)^m + \frac{1}{2} \cdot \frac{1}{2^m(2-\rho)^{2m}} \cdot \left((2-\rho)^3 + \rho^3\right)^m}{\frac{1}{4} \cdot \left[(2^n - 1) \left(\frac{1}{2 - \rho}\right)^m + \left(\frac{(2-\rho)^2 + \rho^2}{2(2-\rho)}\right)^m\right]^2}\\
        &= \frac{\frac{(2^{n-1} - 1)^2}{2^{2m} } + \frac{2^{n} - \frac{3}{2}}{2^m } \cdot \left(\left(1-\frac{\rho}{2}\right)^2 + \left(\frac{\rho}{2}\right)^2\right)^m + \frac{1}{2} \cdot \left(\left(1-\frac{\rho}{2}\right)^3 + \left(\frac{\rho}{2}\right)^3\right)^m}{\frac{1}{4}\cdot \left(\frac{(2^n - 1)^2}{2^{2m}} + 2\cdot \frac{2^n - 1}{2^m} \cdot \left(\left(1-\frac{\rho}{2}\right)^2 + \left(\frac{\rho}{2}\right)^2\right)^m + \left(\left(1-\frac{\rho}{2}\right)^2 + \left(\frac{\rho}{2}\right)^2\right)^{2m}\right)}\\
        &= \frac{(1 + O(2^{n-2m}))A^2 + (4 + O(2^{-m})) AB + 2C}{(1-O(2^{n-2m}))A^2 + (2 - O(2^{-m}))AB + B^2},
        \intertext{where $A = \frac{2^n}{2^m}$, $B = \left(\left(1-\frac{\rho}{2}\right)^2 + \left(\frac{\rho}{2}\right)^2\right)^m$, and $C = \left(\left(1-\frac{\rho}{2}\right)^3 + \left(\frac{\rho}{2}\right)^3\right)^m$,}
        &= (1 + o(1))\frac{1+ 4\frac{B}{A} + 2\frac{C}{A^2}}{1 + 2\frac{B}{A} + \frac{B^2}{A^2}}.
    \end{align*}
    As long as $\frac{B}{A} = o(1)$ and $\frac{C}{A^2} = o(1)$, then we can conclude that $\Var[L_0] = o(\EE[L_0]^2)$. We further note that for any $\rho \in [0,1]$, $\left(1-\frac{\rho}{2}\right)^3 + \left(\frac{\rho}{2}\right)^3 = 1 - \frac{3}{2}\rho + \frac{3}{4}\rho^2 \ge \left(\left(1-\frac{\rho}{2}\right)^2 + \left(\frac{\rho}{2}\right)^2\right)^2$. Thus, $C \ge B^2$. As long as $\frac{C}{A^2} = o(1)$, we also get $\frac{B}{A} \le \frac{C^{\frac{1}{2}}}{A} = o(1)$.

    Therefore, if $\frac{1}{C} = \frac{1}{\left(1 - \frac{3}{2}\rho + \frac{3}{4}\rho^2\right)^m} \gg \frac{1}{A^2} = 2^{2(m-n)}$, then we have $\frac{B}{A} = o(1), \frac{C}{A^2} = o(1)$ and
    \begin{align*}
        \frac{\Var[L_0]}{\E[L_0]^2} &= \frac{\E[L_0^2]}{\E[L_0]^2} - 1\\
        &= (1+o(1))\frac{1+4\frac{B}{A} + 2\frac{C}{A^2}}{1 + 2\frac{B}{A}+\frac{B^2}{A^2}} - 1\\
        &\le o(1).
    \end{align*}

    By a symmetric argument, we get $\Var[L_1] = o(\EE[L_1]^2)$. Therefore, by Chebyshev's inequality, with probability $1 - o(1)$ over the random choice of $A$ and $x^*$, we have
    \begin{align*}
        L_0 &= (1 \pm o(1))\E[L_0],\\
        L_1 &= (1 \pm o(1))\E[L_1],
    \end{align*}
    Since $\E[L_0] = \E[L_1]$, with probability $1 - o(1)$, we have
    \begin{align*}
        \Pr(x_1 = 0 \vert A,\hat{y}) &= \frac{L_0}{L_0 + L_1}\\
        &=  \frac{1}{2} \pm o(1),\\
        \Pr(x_1 = 1 \vert A,\hat{y}) &= \frac{L_1}{L_0 + L_1}\\
        &=  \frac{1}{2} \pm o(1).
    \end{align*}
    Conditioning on this high probability event $E$, the conditional MMSE is at least 
    \begin{align*}
        &\quad \E[(x_1^* - \E[x_1 \vert A, \hat{y}])^2 \vert E ] \\
        &\ge \E\left[\left(x_1^* - \frac{1}{2} \pm o(1)\right)^2 \vert E \right] \\
        &= \left(\frac{1}{2} - o(1)\right)^2\\
        &= \frac{1}{4} - o(1),
    \end{align*}
    and thus,
    \begin{align*}
        \E[(x_1^* - \E[x_1 \vert A, \hat{y}])^2] &\ge \Pr(E) \cdot \E[(x_1^* - \E[x_1 \vert A, \hat{y}])^2 \vert E ] = \frac{1}{4} - o(1).
    \end{align*}
    Due to the symmetry of the random linear code problem, for any $i \in [n]$, we have
    \begin{align*}
        \E[(x_i^* - \E[x_i \vert A, \hat{y}])^2] = \E[(x_1^* - \E[x_1 \vert A, \hat{y}])^2] &\ge \frac{1}{4} - o(1).
    \end{align*}
    Finally, we conclude that the noisy MMSE is at least 
    \begin{align*}
        &\quad \text{MMSE}_{\rho}\\
        &= \E[\|x^* - \E[x \vert A, T_{\rho}(y)]\|_2^2]\\
        &= \E[\|x^* - \E[x \vert A, \hat{y}]\|_2^2]\\
        &= \sum_{i=1}^n E[(x_i^* - \E[x_i \vert A, \hat{y}])^2]\\
        &\ge  \left(\frac{1}{4} - o(1)\right)n.
    \end{align*}
    
\end{proof}

\subsection{Stability of Low-Degree Polynomials: Proof of Theorem~\ref{thm:low-deg-stability-rlc}}

In this section, we prove the stability parameters of polynomials for the random linear code problem. Recall that to prove Theorem~\ref{thm:low-deg-stability-rlc}, we need to show that for any $D$ such that $2D \le n$, a polynomial $f: \{0,1\}^{m \times n} \times \{0,1\}^m \to \R$ of degree at most $D$ is $(\rho, 2(1 - (1-\rho)^D))$-stable for the random linear code problem.

\begin{proof}[Theorem~\ref{thm:low-deg-stability-rlc}]
    For a polynomial $f: \{0,1\}^{m \times n} \times \{0,1\}^m \to \R$ of degree at most $D$, we may write it uniquely as
    \begin{align*}
        f(A,y) = \sum_{\substack{S \subseteq [m] \times [n],\\ T \subseteq [m]:\\ |S| + |T| \le D}} \hat{f}_{S,T} \chi_{S,T}(A,y),
    \end{align*}
    where $\chi_{S,T}(A,y)$ denotes
    \begin{align*}
        \chi_{S,T}(A,y) = \prod_{\{i,j\} \in S} (2A_{i,j} - 1) \prod_{k \in T} (2y_k - 1).
    \end{align*}
    One important fact that we will use is that for any $S \subseteq [m] \times [n]$ and $T \subseteq [m]$, as long as there does not exist an index $i \in [m]$ such that $\{i\} \times [n] \subseteq S$ and $i \in T$, then the marginal distribution of $(A^S, y^T)$ is uniform $\text{Bern}(1/2)^{\otimes S} \times \text{Bern}(1/2)^{\otimes T}$ for $(A,y)$ sampled in the random linear code problem. This follows from $y = Ax$ where both $A$ and $x$ are drawn uniformly at random, and thus $y^T \vert A^S$ is distributed as $\text{Bern}(1/2)^{\otimes T}$ as long as $A^S$ does not contain an entire row of $A$ corresponding to some entry of $y$.
    
    Now, we have
    \begin{align*}
        \E[(f(A,y) - f(A,T_{\rho}(y)))^2] &= \E[f(A,y)^2] + \E[f(A,T_{\rho}(y))^2] - 2\E[f(A,y)f(A,T_{\rho}(y))]. 
    \end{align*}
    We will analyze each of the three terms separately. 

    First,
    \begin{align*}
        \E[f(A,y)^2] &= \sum_{\substack{S_1, T_1, S_2, T_2:\\ |S_1| + |T_1| \le D,\\
        |S_2| + |T_2| \le D}} \hat{f}_{S_1, T_1} \hat{f}_{S_2, T_2} \E[\chi_{S_1, T_1}(A,y) \chi_{S_2,T_2}(A,y) ]\\
        &= \sum_{\substack{S_1, T_1, S_2, T_2:\\ |S_1| + |T_1| \le D,\\
        |S_2| + |T_2| \le D}} \hat{f}_{S_1, T_1} \hat{f}_{S_2, T_2} \E[\chi_{S_1 \triangle S_2, T_1\triangle T_2 }(A,y) ]
        \intertext{Note that since $2D \le n$, the marginal distribution of $(A^{S_1 \triangle S_2}, y^{T_1 \triangle T_2})$ is uniform, and thus the expectation of $\chi_{S_1 \triangle S_2, T_1\triangle T_2 }$ is $0$ whenever $(S_1, T_1) \ne (S_2, T_2)$.}
        &= \sum_{\substack{S, T:\\ |S| + |T| \le D}} \hat{f}_{S, T}^2. \stepcounter{equation}\tag{\theequation}\label{eq:rlc-low-deg-second-moment}
    \end{align*}

    Second, 
    \begin{align*}
        \E[f(A,T_{\rho}(y))^2] &= \sum_{\substack{S_1, T_1, S_2, T_2:\\ |S_1| + |T_1| \le D,\\
        |S_2| + |T_2| \le D}} \hat{f}_{S_1, T_1} \hat{f}_{S_2, T_2} \E[\chi_{S_1, T_1}(A,T_{\rho}(y)) \chi_{S_2,T_2}(A,T_{\rho}(y)) ]\\
        &= \sum_{\substack{S_1, T_1, S_2, T_2:\\ |S_1| + |T_1| \le D,\\
        |S_2| + |T_2| \le D}} \hat{f}_{S_1, T_1} \hat{f}_{S_2, T_2} \E[\chi_{S_1 \triangle S_2, T_1\triangle T_2 }(A,T_{\rho}(y)) ]
        \intertext{Note that $T_{\rho}$ resamples each coordinate of $y$ from $\text{Bern}(1/2)$ with probability $\rho$, and thus when $2D \le n$, the marginal distribution of $(A^{S_1 \triangle S_2}, T_{\rho}(y)^{T_1 \triangle T_2})$ is uniform, and the expectation of $\chi_{S_1 \triangle S_2, T_1\triangle T_2 }$ is $0$ whenever $(S_1, T_1) \ne (S_2, T_2)$.}
        &= \sum_{\substack{S, T:\\ |S| + |T| \le D}} \hat{f}_{S, T}^2. \stepcounter{equation}\tag{\theequation}\label{eq:rlc-low-deg-noisy-second-moment}
    \end{align*}

    Third, we have
    \begin{align*}
        \E[f(A,y)f(A, T_{\rho}(y))] &= \sum_{\substack{S_1, T_1, S_2, T_2:\\ |S_1| + |T_1| \le D,\\
        |S_2| + |T_2| \le D}} \hat{f}_{S_1, T_1} \hat{f}_{S_2, T_2} \E[\chi_{S_1, T_1}(A,y) \chi_{S_2,T_2}(A,T_{\rho}(y)) ]
        \intertext{Note that whenever some entry of $y^{T_2}$ is resampled from $\text{Bern}(1/2)$ by $T_{\rho}$, then the conditional expectation of $\chi_{S_1, T_1}(A,y) \chi_{S_2,T_2}(A,T_{\rho}(y))$ is $0$. Thus,}
        &= \sum_{\substack{S_1, T_1, S_2, T_2:\\ |S_1| + |T_1| \le D,\\
        |S_2| + |T_2| \le D}} \hat{f}_{S_1, T_1} \hat{f}_{S_2, T_2} \cdot (1 - \rho)^{|T_2|} \cdot \E[\chi_{S_1, T_1}(A,y) \chi_{S_2,T_2}(A,y) ]\\
        &= \sum_{\substack{S, T:\\ |S| + |T| \le D}} \hat{f}_{S, T}^2 (1 - \rho)^{|T|}\\
        &\ge (1 - \rho)^D \sum_{\substack{S, T:\\ |S| + |T| \le D}} \hat{f}_{S, T}^2. \stepcounter{equation}\tag{\theequation}\label{eq:rlc-low-degree-cross}
    \end{align*}

    Combining \eqref{eq:rlc-low-deg-second-moment}, \eqref{eq:rlc-low-deg-noisy-second-moment}, and \eqref{eq:rlc-low-degree-cross}, we have
    \begin{align*}
        \E[(f(A,y) - f(A,T_{\rho}(y)))^2] &= \E[f(A,y)^2] + \E[f(A,T_{\rho}(y))^2] - 2\E[f(A,y)f(A,T_{\rho}(y))]\\
        &\le 2(1 - (1-\rho)^D) \sum_{\substack{S, T:\\ |S| + |T| \le D}} \hat{f}_{S, T}^2.\\
        &= 2(1 - (1-\rho)^D) \cdot \E[f(A,y)^2].
    \end{align*}
    This finishes the proof that if $2D \le n$, then $f$ is $(\rho, 2(1 - (1 - \rho)^D)$-stable.
\end{proof}

\section{Proofs for Gaussian Subset Sum}\label{sec:GSS_proofs}
\subsection{Deferred proof for MMSE instability}

\begin{proof}[\ref{prop:aon_GSS}]
    The GSS setting is identical to the sparse regression setting of \cite{reeves2019all} with one sample and we are in the noiseless case (in their notation) $\sigma=0$. In the noisy case, the OU operator on GSS, corresponds to the sparse regression setting again with $n=1$ and noise $\sigma=\rho/\sqrt{1-\rho^2}.$ 
    
    Now, in the sparse regression setting for noise $\sigma=\rho/\sqrt{1-\rho^2}$  the information theoretic threshold for AoN according to \cite{reeves2019all} is given by the sample size $n^*=2k\log(N/k)/\log(1+k/\sigma^2)=2k\log(N/k)/\log(1+k(1-\rho^2)/\rho^2).$ In particular, if $n \leq 0.99 n^*$ and $k/\sigma^2=k(1-\rho^2)/\rho^2 \geq C$ for some universal constant $C>0,$ we have from \cite[Theorem 3, arXiv version]{reeves2019all} that the MMSE is ``trivial" i.e., $\text{MMSE}_{\rho} \geq (1-o(1))k$. But it is easy to see that one can choose some $\rho=\exp(-\Theta(k \log (N/k)))$,  we can guarantee $n^* \geq 30$, making sure ``trivial" MMSE holds for $n=1$ for this $\rho>0$ for the GSS setting.
\end{proof}
\subsection{Stability of Low-Degree Polynomials}

In this section, we prove the stability parameters of low-degree polynomials for GSS. We will make use of the following theorem regarding the expectation of products of Hermite polynomials, the proof of which is deferred to Section~\ref{sec:deferred-proof-wick-product}.

\begin{theorem}[Diagram Formula for Expectation of Products of Hermite Polynomials]\label{thm:wick-product-formula}
    
    Suppose $x_1, \dots, x_k$ are identically distributed as $N(0,1)$. Let $R_{ij} = \E[x_ix_j]$ denote their correlations. Given $\alpha \in \N^k$, we define the following graph $G(\alpha) = (V, E)$ on $|\alpha|$ vertices. For each $i \in [k]$, create $\alpha_i$ vertices corresponding to $x_i$, and let $j(v) = i$ for each of the $\alpha_i$ vertices $v$ associated with $x_i$. For $u, v \in V$, add an edge $\{u,v\}$ if $j(u) \ne j(v)$.
Then, we have
\begin{align*}
    \E\left[\prod_{i=1}^k h_{\alpha_i}(x_i)\right] = \frac{1}{\sqrt{\alpha!}} \sum_{M \in \MM(\alpha)} \prod_{\{a,b\} \in M} R_{j(a)j(b)},
\end{align*}
where $\MM(\alpha)$ denotes the collection of perfect matchings in the graph $G(\alpha)$ described above.
\end{theorem}

Next, we prove that if $k = o(n)$ and $D = o(\min\{k^{1/4}, (n/k)^{1/5}\})$, then polynomials of degree at most $D$ are $(\rho, 2(1 - (\sqrt{1-\rho^2})^D + o(1)))$-stable for GSS.

\begin{proof}[Theorem~\ref{thm:low-deg-stability-gss}]
    
Consider a polynomial $f: \R^n \times \R \to \R$ of degree at most $D$. We may write it uniquely as
\begin{align*}
    f(X, Y) &= \sum_{\substack{\alpha \in \N^n, t\in \N:\\ |\alpha| + t \le D}} \hat{f}_{\alpha,t} h_{\alpha}(X)h_t(y),
\end{align*}
where $y = \frac{Y}{\sqrt{k}}$.

Note that since $Y = \sum_{i \in S} X_i$ for $S$ that is a uniform random $k$-subset of $[n]$, we may express $y = \langle u, X\rangle$ where $u := \frac{1}{\sqrt{k}}\one_S$ and $\one_S$ is the indicator vector of $S$. We may compute the second moment of $f$ as
\begin{align*}
    &\quad \E\left[f(X,Y)^2\right]\\
    &= \sum_{\substack{\alpha_1, \alpha_2 \in \N^n, t_1, t_2\in \N:\\ |\alpha_1| + t_1 \le D, |\alpha_2| + t_2 \le D}} \hat{f}_{\alpha_1, t_1} \hat{f}_{\alpha_2, t_2} \E\left[h_{\alpha_1}(X)h_{\alpha_2}(X)h_{t_1}(y)h_{t_2}(y)\right]\\
    &= \sum_{\substack{\alpha_1, \alpha_2 \in \N^n, t_1, t_2\in \N:\\ |\alpha_1| + t_1 \le D, |\alpha_2| + t_2 \le D}} \hat{f}_{\alpha_1, t_1} \hat{f}_{\alpha_2, t_2} \E_S \E_X\left[h_{\alpha_1}(X)h_{\alpha_2}(X)h_{t_1}(y)h_{t_2}(y)\right].
\end{align*}

Now we apply Wick's formula to analyze the inner expectation. We have
\begin{align*}
    &\quad \E_S \E_X\left[h_{\alpha_1}(X)h_{\alpha_2}(X)h_{t_1}(y)h_{t_2}(y)\right]\\
    &= \E_S\left[\frac{1}{\sqrt{\alpha_1!\alpha_2!t_1!t_2!}} \sum_{M \in \MM(\alpha_1, \alpha_2, t_1, t_2)} \prod_{\{a,b\} \in M} R_{j(a)j(b)}\right].
\end{align*}
From this formula, we see that all the inner expectations are nonnegative. We next analyze the perfect matchings in the graph $G(\alpha_1, \alpha_2, t_1, t_2)$. There are 4 types of vertices in $G$, which are $X$-vertices coming from $h_{\alpha_1}$, $X$-vertices coming from $h_{\alpha_2}$, $y$-vertices coming from $h_{t_1}$, and $y$-vertices coming from $h_{t_2}$. Let us denote the $X$-vertices coming from $h_{\alpha_1}$ as $U$, the $X$-vertices coming from $h_{\alpha_2}$ as $V$, $y$-vertices coming from $h_{t_1}$ as $R$, and $y$-vertices coming from $h_{t_2}$ as $T$. Note that for fixed choice of $S$, the covariance between a vertex in $U \sqcup V$ corresponding to $X_i$ and a $y$-vertex is $\frac{1}{\sqrt{k}}\allone\{i \in S\}$. The covariance between vertices in $U \sqcup V$ corresponding to $X_i$ and $X_j$ is $\allone\{i = j\}$. The covariance between vertices in $R \sqcup T$ is $1$.

Thus, taking expectation of $S$, the matchings that contribute to the sum are the ones consisting of a perfect matching between a subset of $U$ and a subset of $V$, a perfect matching between a subset of $R$ and a subset of $T$, and then a perfect matching between the rest of $U \sqcup V$ and the rest of $R \sqcup T$. Moreover, the perfect matching between a subset of $U$ and a subset of $V$ needs to be further a perfect matching when restricted to the vertices corresponding to $X_i$ for any $i \in [n]$, since the fact that $X_i$ and $X_j$ are uncorrelated for $i\ne j$ prevents contributing matchings to have edges connecting vertices corresponding to different $X_i$'s.

For a contributing matching $M \in \MM(\alpha_1, \alpha_2, t_1, t_2)$, let us call the edges $\{a,b\} \in M$ connecting vertices in $U \sqcup V$ as $X$-edges, the edges connecting vertices in $R \sqcup T$ as $y$-edges, and the edges connecting vertices between $U \sqcup V$ and $R \sqcup T$ as cross-edges. We know that each $X$-edge and each $y$-edge contribute $1$ to the product of $M$. The contribution of the cross-edges depends on two quantities: (1) the number of cross-edges, and (2) the number of vertices corresponding to distinct variables $X_i$'s that the cross-edges are incident to. Let $m$ denote the number of cross-edges, and $z$ denote the number of distinct $X_i$'s that the cross-edges are incident to. We note the obvious relation that $m \ge z$. The contribution of the cross-edges, under the expectation of $S$, is equal to 
\begin{align*}
    &\quad \Pr(\text{a fixed }z\text{-subset of }[n] \subseteq S) \cdot \left(\frac{1}{\sqrt{k}}\right)^m\\
    &= (1 + O(1/n))\left(\frac{k}{n}\right)^z \cdot \left(\frac{1}{\sqrt{k}}\right)^m.
\end{align*}
Next, we give an estimate $C(m,z; \alpha_1, \alpha_2, t_1, t_2)$ on the number of contributing perfect matchings with $m$ cross-edges and $z$ distinct $X_i$'s that the cross-edges are incident to. To enumerate such a matching, we may first enumerate $m$ vertices in $U \sqcup V$ and $m$ vertices in $R \sqcup T$, and then a perfect matching between the enumerated subset of $U \sqcup V$ and the enumerated subset of $R \sqcup T$. This gives at most \[\binom{|U|+|V|}{m} \cdot \binom{|R| + |T|}{m} \cdot m! \le \left(\frac{e(|U|+|V|)}{m}\right)^m \left(\frac{e(|R|+|T|)}{m}\right)^m m^m\] choices. Next, we enumerate the $X$-edges, i.e., a perfect matching between the remaining vertices in $U \sqcup V$. This gives at most \[\min\{\alpha_1!, \alpha_2!\} \le \sqrt{\alpha_1!\alpha_2!}\]
choices. Finally, we enumerate the $y$-edges, i.e., a perfect matching between the remaining vertices in $R \sqcup T$. This gives at most
\[\min\{t_1!, t_2!\} \le \sqrt{t_1!t_2!}\]
choices. Consequently, we have
\begin{align*}
    &\quad C(m,z; \alpha_1, \alpha_2, t_1, t_2) \\
    &\le \sqrt{\alpha_1!\alpha_2! t_1!t_2!} \left(\frac{e(|U|+|V|)}{m}\right)^m \left(\frac{e(|R|+|T|)}{m}\right)^m m^m\\
    &= \sqrt{\alpha_1!\alpha_2! t_1!t_2!} \left(\frac{e^2D^2}{m}\right)^m,
\end{align*}
where we use that  $(|U|+|V|)(|R|+|T|) \le D^2$ since $|U|+|V|+|R|+|T| = |\alpha_1| + t_1 + |\alpha_2| + t_2 \le 2D$.

Additionally, we observe that we must have $z \ge |\alpha_1 \triangle \alpha_2|_0$, since any coordinate $i$ on which $\alpha_1$ and $\alpha_2$ differs has some vertex corresponding to $X_i$ incident to cross-edges in a contributing perfect matching. We also have $m \ge |\alpha_1 \triangle \alpha_2|$, since there are at least $|\alpha_1(i) - \alpha_2(i)|$ cross-edges incident to the vertices corresponding to $X_i$. Similarly, we have $m \ge |t_1 - t_2|$ since there are at least $|t_1 - t_2|$ cross-edges incident to the vertices corresponding to $y$.

In total, for fixed $\alpha_1, \alpha_2, t_1, t_2$, the total contribution of the inner expectation is at most
\begin{align*}
    &\quad \E_S \E_X\left[h_{\alpha_1}(X)h_{\alpha_2}(X)h_{t_1}(y)h_{t_2}(y)\right]\\
    &= \E_S\left[\frac{1}{\sqrt{\alpha_1!\alpha_2!t_1!t_2!}} \sum_{M \in \MM(\alpha_1, \alpha_2, t_1, t_2)} \prod_{\{a,b\} \in M} R_{j(a)j(b)}\right]\\
    &\le \frac{1}{\sqrt{\alpha_1!\alpha_2!t_1!t_2!}} \sum_{z \ge |\alpha_1 \triangle \alpha_2|_0} \sum_{m \ge \max\{|\alpha_1 \triangle \alpha_2|, |t_1 - t_2|\}} C(m,z; \alpha_1, \alpha_2, t_1, t_2) \cdot (1 + O(1/n))\left(\frac{k}{n}\right)^z \cdot \left(\frac{1}{\sqrt{k}}\right)^m\\
    &\le (1 + O(1/n)) \sum_{z \ge |\alpha_1 \triangle \alpha_2|_0} \left(\frac{k}{n}\right)^z\sum_{m \ge \max\{|\alpha_1 \triangle \alpha_2|, |t_1 - t_2|\}} \left(\frac{e^2D^2}{m\sqrt{k}}\right)^m. \stepcounter{equation}\tag{\theequation}\label{ineq:inner-expectation}
\end{align*}

We also observe that for the diagonal terms with $\alpha_1 = \alpha_2 = \alpha$ and $t_1 = t_2 = t$, this expectation \eqref{ineq:inner-expectation} is at least
\begin{align*}
    &\quad \E_S \E_X\left[h_{\alpha}(X)^2h_{t}(y)^2\right]\\
    &= \E_S\left[\frac{1}{\sqrt{\alpha!\alpha!t!t!}} \sum_{M \in \MM(\alpha, \alpha, t, t)} \prod_{\{a,b\} \in M} R_{j(a)j(b)}\right]\\
    &\ge \frac{1}{\alpha! t!} \cdot \alpha!t!\\
    &= 1,
\end{align*}
where in the last inequality we take the sum over perfect matchings without cross-edges, giving a total of $\alpha!$ choices for matching between $U$ and $V$, times $t!$ choices for matching between $R$ and $T$.

Now we have
\begin{align*}
    &\quad \E[f(X,Y)^2]\\
    &= \sum_{\substack{\alpha_1, \alpha_2 \in \N^n, t_1, t_2\in \N:\\ |\alpha_1| + t_1 \le D, |\alpha_2| + t_2 \le D}} \hat{f}_{\alpha_1, t_1} \hat{f}_{\alpha_2, t_2} \E_S \E_X\left[h_{\alpha_1}(X)h_{\alpha_2}(X)h_{t_1}(y)h_{t_2}(y)\right]\\
    &= \sum_{\substack{\alpha\in \N^n, t\in \N:\\ |\alpha| + t \le D}} \hat{f}_{\alpha, t}^2 \E_S \E_X\left[h_{\alpha}(X)^2h_{t}(y)^2\right]\\
    &\quad + \sum_{\substack{\alpha_1, \alpha_2 \in \N^n, t_1, t_2\in \N:\\ |\alpha_1| + t_1 \le D, |\alpha_2| + t_2 \le D,\\ (\alpha_1, t_1) \ne (\alpha_2, t_2)}} \hat{f}_{\alpha_1, t_1} \hat{f}_{\alpha_2, t_2} \E_S \E_X\left[h_{\alpha_1}(X)h_{\alpha_2}(X)h_{t_1}(y)h_{t_2}(y)\right].
\end{align*}

We know that the contribution of the diagonal terms is at least
\begin{align*}
    &\quad \sum_{\substack{\alpha\in \N^n, t\in \N:\\ |\alpha| + t \le D}} \hat{f}_{\alpha, t}^2 \E_S \E_X\left[h_{\alpha}(X)^2h_{t}(y)^2\right]\\
    &\ge \sum_{\substack{\alpha\in \N^n, t\in \N:\\ |\alpha| + t \le D}} \hat{f}_{\alpha, t}^2.
\end{align*}

Now let us consider the off-diagonal terms. For fixed $\alpha_1, \alpha_2 \in \N^n$, we will denote
\begin{align}
    \beta &= \alpha_1 \wedge \alpha_2\\
    a &= |\alpha_1 - \beta|_0\\
    b &= |\alpha_2 - \beta|_0\\
    c &= |\alpha_1 - \beta|\\
    d &= |\alpha_2 - \beta|.
\end{align}
For fixed $\alpha_1 \in \N^n$ and fixed $a,b,c,d\in \N$, we will use $W_{\alpha_1}(a,b,c,d)$ to denote the collection of $\alpha_2 \in \N^n$ such that $\alpha_2$ satisfies the parameters above and $\alpha_2 \ne \alpha_1$. Note that if $\alpha_1 \ne \alpha_2$, we have $a+b \ge 1$. Then, the contribution of the off-diagonal terms is at most
\begin{align*}
    &\quad \left|\sum_{\substack{\alpha_1, \alpha_2 \in \N^n, t_1, t_2\in \N:\\ |\alpha_1| + t_1 \le D, |\alpha_2| + t_2 \le D,\\ (\alpha_1, t_1) \ne (\alpha_2, t_2)}} \hat{f}_{\alpha_1, t_1} \hat{f}_{\alpha_2, t_2} \E_S \E_X\left[h_{\alpha_1}(X)h_{\alpha_2}(X)h_{t_1}(y)h_{t_2}(y)\right]\right|\\
    &\le \sum_{\substack{\alpha_1, \alpha_2 \in \N^n, t_1, t_2\in \N:\\ |\alpha_1| + t_1 \le D, |\alpha_2| + t_2 \le D,\\ (\alpha_1, t_1) \ne (\alpha_2, t_2)}} |\hat{f}_{\alpha_1, t_1}| |\hat{f}_{\alpha_2, t_2}| \cdot |\E_S \E_X\left[h_{\alpha_1}(X)h_{\alpha_2}(X)h_{t_1}(y)h_{t_2}(y)\right]|\\
    &= \sum_{\substack{\alpha \in \N^n, t_1, t_2\in \N:\\ |\alpha| + t_1 \le D, |\alpha| + t_2 \le D,\\ t_1 \ne t_2}} |\hat{f}_{\alpha, t_1}| |\hat{f}_{\alpha, t_2}| \cdot |\E_S \E_X\left[h_{\alpha}(X)^2h_{t_1}(y)h_{t_2}(y)\right]|\\
    &\quad + \sum_{\substack{\alpha_1, \alpha_2 \in \N^n, t_1, t_2\in \N:\\ |\alpha_1| + t_1 \le D, |\alpha_2| + t_2 \le D,\\ \alpha_1 \ne \alpha_2}} |\hat{f}_{\alpha_1, t_1}| |\hat{f}_{\alpha_2, t_2}| \cdot |\E_S \E_X\left[h_{\alpha_1}(X)h_{\alpha_2}(X)h_{t_1}(y)h_{t_2}(y)\right]|.
\end{align*}

The first term involving sum over $t_1 \ne t_2$ is at most
\begin{align*}
    &\quad \sum_{\substack{\alpha \in \N^n, t_1, t_2\in \N:\\ |\alpha| + t_1 \le D, |\alpha| + t_2 \le D,\\ t_1 \ne t_2}} |\hat{f}_{\alpha, t_1}| |\hat{f}_{\alpha, t_2}| \cdot |\E_S \E_X\left[h_{\alpha}(X)^2h_{t_1}(y)h_{t_2}(y)\right]|\\
    &\le 2\sum_{\alpha \in \N^n, t_1 \in \N} |\hat{f}_{\alpha, t_1}| \sum_{\substack{t_2 \ne t_1,\\ |\hat{f}_{\alpha, t_2}| \le |\hat{f}_{\alpha, t_1}|}} |\hat{f}_{\alpha, t_2}| \cdot \E_S \E_X\left[h_{\alpha}(X)^2h_{t_1}(y)h_{t_2}(y)\right]
    \intertext{where the inequality holds because every ordered pairs of $(t_1, t_2)$ may be reordered so that $|\hat{f}_{\alpha, t_2}| \le |\hat{f}_{\alpha, t_1}|$. Now, we apply \eqref{ineq:inner-expectation} and get}
    &\le 2(1 + O(1/n))\sum_{\alpha \in \N^n, t_1 \in \N}  \hat{f}_{\alpha, t_1}^2 \sum_{t_2 \ne t_1}   \sum_{z \ge 0} \left(\frac{k}{n}\right)^z\sum_{m \ge |t_1 - t_2|} \left(\frac{e^2D^2}{m\sqrt{k}}\right)^m\\
    &\le 2\left(1 + O(k/n) + O\left(D^2/\sqrt{k}\right)\right)\sum_{\alpha \in \N^n, t_1 \in \N}  \hat{f}_{\alpha, t_1}^2 \sum_{t_2 \ne t_1}   \left(\frac{e^2D^2}{|t_1 - t_2|\sqrt{k}}\right)^{|t_1 - t_2|}\\
    &\le 2\left(1 + O(k/n) + O\left(D^2/\sqrt{k}\right)\right)\sum_{\alpha \in \N^n, t_1 \in \N}  \hat{f}_{\alpha, t_1}^2 \cdot   2\left(\frac{e^2D^2}{\sqrt{k}}\right)\\
    &= o\left(\sum_{\alpha \in \N^n, t \in \N}  \hat{f}_{\alpha, t}^2\right),
\end{align*}
provided that $k = o(n)$ and $D = o(k^{1/4})$.

The second term involving sum over $\alpha_1 \ne \alpha_2$ is at most
\begin{align*}
    &\quad \sum_{\substack{\alpha_1, \alpha_2 \in \N^n, t_1, t_2\in \N:\\ |\alpha_1| + t_1 \le D, |\alpha_2| + t_2 \le D,\\ \alpha_1 \ne \alpha_2}} |\hat{f}_{\alpha_1, t_1}| |\hat{f}_{\alpha_2, t_2}| \cdot |\E_S \E_X\left[h_{\alpha_1}(X)h_{\alpha_2}(X)h_{t_1}(y)h_{t_2}(y)\right]|\\
    &\le 2 \sum_{\substack{\alpha_1 \in \N^n, t_1 \in \N:\\ |\alpha_1| + t_1 \le D}} |\hat{f}_{\alpha_1, t_1}| \sum_{\substack{a,b,c,d \ge 0,\\a+b\ge 1}} \sum_{\substack{\alpha_2 \in \N^n, t_2 \in \N:\\ |\alpha_2| + t_2 \le D,\\
    \alpha_2 \in W_{\alpha_1}(a,b,c,d)\\|\hat{f}_{\alpha_2, t_2}|n^{\frac{b}{2}}D^{\frac{a}{2}}\\ \le|\hat{f}_{\alpha_1, t_1}|n^{\frac{a}{2}}D^{\frac{b}{2}} }} |\hat{f}_{\alpha_2,t_2}| \cdot \E_S \E_X\left[h_{\alpha_1}(X)h_{\alpha_2}(X)h_{t_1}(y)h_{t_2}(y)\right]
    \intertext{where the inequality holds because every ordered pairs of $((\alpha_1,t_1), (\alpha_2,t_2))$ may be reordered so that $|\hat{f}_{\alpha_2, t_2}|n^{\frac{b}{2}}D^{\frac{a}{2}}\le|\hat{f}_{\alpha_1, t_1}|n^{\frac{a}{2}}D^{\frac{b}{2}}$ }
    &\le 2 \sum_{\substack{\alpha_1 \in \N^n, t_1 \in \N:\\ |\alpha_1| + t_1 \le D}} \hat{f}_{\alpha_1, t_1}^2 \sum_{\substack{a,b,c,d \ge 0,\\a+b\ge 1}} n^{\frac{a-b}{2}} D^{\frac{b-a}{2} } \sum_{\substack{\alpha_2 \in \N^n, t_2 \in \N:\\ |\alpha_2| + t_2 \le D,\\
    \alpha_2 \in W_{\alpha_1}(a,b,c,d)\\|\hat{f}_{\alpha_2, t_2}|n^{\frac{b}{2}}D^{\frac{a}{2}}\\ \le|\hat{f}_{\alpha_1, t_1}|n^{\frac{a}{2}}D^{\frac{b}{2}} }}  \E_S \E_X\left[h_{\alpha_1}(X)h_{\alpha_2}(X)h_{t_1}(y)h_{t_2}(y)\right]\\
    &\le 2(1 + O(1/n)) \sum_{\substack{\alpha_1 \in \N^n, t_1 \in \N:\\ |\alpha_1| + t_1 \le D}} \hat{f}_{\alpha_1, t_1}^2 \sum_{\substack{a,b,c,d \ge 0,\\a+b\ge 1}} n^{\frac{a-b}{2}} D^{\frac{b-a}{2} } \\
    &\quad \sum_{\substack{\alpha_2 \in \N^n, t_2 \in \N:\\ |\alpha_2| + t_2 \le D,\\
    \alpha_2 \in W_{\alpha_1}(a,b,c,d)\\|\hat{f}_{\alpha_2, t_2}|n^{\frac{b}{2}}D^{\frac{a}{2}}\\ \le|\hat{f}_{\alpha_1, t_1}|n^{\frac{a}{2}}D^{\frac{b}{2}} }} \sum_{z \ge |\alpha_1 \triangle \alpha_2|_0} \left(\frac{k}{n}\right)^z\sum_{m \ge \max\{|\alpha_1 \triangle \alpha_2|, |t_1 - t_2|\}} \left(\frac{e^2D^2}{m\sqrt{k}}\right)^m\\
    &\le 2(1 + O(1/n)) \sum_{\substack{\alpha_1 \in \N^n, t_1 \in \N:\\ |\alpha_1| + t_1 \le D}} \hat{f}_{\alpha_1, t_1}^2 \sum_{\substack{a,b,c,d \ge 0,\\a+b\ge 1}} n^{\frac{a-b}{2}} D^{\frac{b-a}{2} } \\
    &\quad \sum_{\substack{\alpha_2 \in \N^n, t_2 \in \N:\\ |\alpha_2| + t_2 \le D,\\
    \alpha_2 \in W_{\alpha_1}(a,b,c,d) }} \sum_{z \ge a+b} \left(\frac{k}{n}\right)^z\sum_{m \ge \max\{c+d, |t_1 - t_2|\}} \left(\frac{e^2D^2}{m\sqrt{k}}\right)^m.
\end{align*}
Now, we proceed  to bound the number of $\alpha_2 \in W_{\alpha_1}(a,b,c,d)$. To enumerate such $\alpha_2$, we may first enumerate the $b$ indices such that $\alpha_2(i) > \alpha_1(i)$, giving at most $\binom{n}{b}$ choices. Then, we enumerate the $a$ indices such that $\alpha_1(i) > \alpha_2(i)$, giving at most $\binom{|\alpha_1|_0}{a}$ choices. Next, we enumerate the choices of $\alpha_2(i)$ for the coordinates $\alpha_1(i) > \alpha_2(i)$. Recall that $\beta = \alpha_1 \wedge \alpha_2$. Since \[c = |\alpha_1 - \beta| = \sum_{i: \alpha_1(i) > \alpha_2(i)} (\alpha_1(i) - \alpha_2(i)),\]
by standard stars-and-bars argument, the number of choices of $\alpha_2$ on the coordinates $\alpha_1(i) > \alpha_2(i)$ is at most $\binom{c-1}{a-1} \le \binom{c}{a}$, if $a \ge 1$, and $1 \le \binom{c}{a}$ if $a = 0$. Similarly, the number of choices of $\alpha_2(i)$ for the coordinates $\alpha_1(i) < \alpha_2(i)$ is at most $\binom{d}{b}$. Thus, $\binom{n}{b} \binom{|\alpha_1|_0}{a} \binom{c}{a} \binom{d}{b}$ is an upper bound on the number of $\alpha_2 \in W_{\alpha_1}(a,b,c,d)$. Returning to our bound of the second term, we have

\begin{align*}
    &\quad \sum_{\substack{\alpha_1, \alpha_2 \in \N^n, t_1, t_2\in \N:\\ |\alpha_1| + t_1 \le D, |\alpha_2| + t_2 \le D,\\ \alpha_1 \ne \alpha_2}} |\hat{f}_{\alpha_1, t_1}| |\hat{f}_{\alpha_2, t_2}| \cdot |\E_S \E_X\left[h_{\alpha_1}(X)h_{\alpha_2}(X)h_{t_1}(y)h_{t_2}(y)\right]|\\
    &\le 2(1 + O(1/n)) \sum_{\substack{\alpha_1 \in \N^n, t_1 \in \N:\\ |\alpha_1| + t_1 \le D}} \hat{f}_{\alpha_1, t_1}^2 \sum_{\substack{a,b,c,d \ge 0,\\a+b\ge 1}} n^{\frac{a-b}{2}} D^{\frac{b-a}{2} } \\
    &\quad \sum_{\substack{\alpha_2 \in \N^n, t_2 \in \N:\\ |\alpha_2| + t_2 \le D,\\
    \alpha_2 \in W_{\alpha_1}(a,b,c,d) }} \sum_{z \ge a+b} \left(\frac{k}{n}\right)^z\sum_{m \ge \max\{c+d, |t_1 - t_2|\}} \left(\frac{e^2D^2}{m\sqrt{k}}\right)^m\\
    &\le 2(1 + O(1/n)) \sum_{\substack{\alpha_1 \in \N^n, t_1 \in \N:\\ |\alpha_1| + t_1 \le D}} \hat{f}_{\alpha_1, t_1}^2 \sum_{\substack{a,b,c,d \ge 0,\\a+b\ge 1}} n^{\frac{a-b}{2}} D^{\frac{b-a}{2} } \\
    &\quad \sum_{t_2 \in \N}\binom{n}{b} \binom{|\alpha_1|_0}{a} \binom{c}{a} \binom{d}{b} \sum_{z \ge a+b} \left(\frac{k}{n}\right)^z\sum_{m \ge \max\{c+d, |t_1 - t_2|\}} \left(\frac{e^2D^2}{m\sqrt{k}}\right)^m\\
    &\le 2(1 + O(1/n)) \sum_{\substack{\alpha_1 \in \N^n, t_1 \in \N:\\ |\alpha_1| + t_1 \le D}} \hat{f}_{\alpha_1, t_1}^2 \sum_{\substack{a,b \ge 0,\\a+b\ge 1,\\ c\ge a, d\ge b}} n^{\frac{a-b}{2}} D^{\frac{b-a}{2} } n^b D^a \cdot\frac{c^a}{a!} \frac{d^b}{b!}\\
    &\quad  \sum_{z \ge a+b} \left(\frac{k}{n}\right)^z\sum_{t_2 \in \N} \sum_{m \ge \max\{c+d, |t_1 - t_2|\}} \left(\frac{e^2D^2}{m\sqrt{k}}\right)^m\\
    &\le 2(1 + O(k/n)) \sum_{\substack{\alpha_1 \in \N^n, t_1 \in \N:\\ |\alpha_1| + t_1 \le D}} \hat{f}_{\alpha_1, t_1}^2 \sum_{\substack{a,b \ge 0,\\a+b\ge 1,\\ c\ge a, d\ge b}} n^{\frac{a-b}{2}} D^{\frac{b-a}{2} } n^b D^a \cdot\frac{c^a}{a!} \frac{d^b}{b!}\\
    &\quad  \left(\frac{k}{n}\right)^{a+b}\left[(2(c+d)+1) \sum_{m \ge c+d} \left(\frac{e^2D^2}{m\sqrt{k}}\right)^m + 2 \sum_{i \ge 1} \sum_{m \ge c+d+i} \left(\frac{e^2D^2}{m\sqrt{k}}\right)^m\right]\\
    &\le 2\left(1 + O(k/n) + O\left(D^2/\sqrt{k}\right)\right) \sum_{\substack{\alpha_1 \in \N^n, t_1 \in \N:\\ |\alpha_1| + t_1 \le D}} \hat{f}_{\alpha_1, t_1}^2 \sum_{\substack{a,b \ge 0,\\a+b\ge 1,\\ c\ge a, d\ge b}} n^{\frac{a-b}{2}} D^{\frac{b-a}{2} } n^b D^a \cdot\frac{c^a}{a!} \frac{d^b}{b!}\\
    &\quad  \left(\frac{k}{n}\right)^{a+b}\left[(2(c+d)+1)  \left(\frac{e^2D^2}{(c+d)\sqrt{k}}\right)^{c+d} + 2 \sum_{i \ge 1} \left(\frac{e^2D^2}{(c+d+i)\sqrt{k}}\right)^{c+d+i}\right]\\
    &\le 2\left(1 + O(k/n) + O\left(D^2/\sqrt{k}\right)\right) \sum_{\substack{\alpha_1 \in \N^n, t_1 \in \N:\\ |\alpha_1| + t_1 \le D}} \hat{f}_{\alpha_1, t_1}^2 \sum_{\substack{a,b \ge 0,\\a+b\ge 1}} n^{\frac{a-b}{2}} D^{\frac{b-a}{2} } n^b D^a \left(\frac{k}{n}\right)^{a+b} \\
    &\quad  \sum_{c\ge a, d\ge b}\frac{(2(c+d)+1)c^a d^b}{a!b!(c+d)^{c+d}}  \left(\frac{e^2D^2}{\sqrt{k}}\right)^{c+d}\\
    &\le 3\left(1 + O(k/n) + O\left(D^2/\sqrt{k}\right)\right) \sum_{\substack{\alpha_1 \in \N^n, t_1 \in \N:\\ |\alpha_1| + t_1 \le D}} \hat{f}_{\alpha_1, t_1}^2 \sum_{\substack{a,b \ge 0,\\a+b\ge 1}} \left(\frac{k \sqrt{D}}{\sqrt{n}}\right)^{a+b} \\
    &\quad  \sum_{c\ge a, d\ge b}\frac{2(c+d)c^a d^b}{a!b!\binom{c+d}{c}c^cd^d}  \left(\frac{e^2D^2}{\sqrt{k}}\right)^{c+d}
    \intertext{Now we inspect the term $\frac{2(c+d)c^ad^b}{a!b!\binom{c+d}{c} c^c d^d}$. Since $c \ge a$ and $d \ge b$, if both $c \ge 1$ and $d \ge 1$, we have $\frac{2(c+d)c^ad^b}{a!b!\binom{c+d}{c} c^c d^d} \le 2$. On the other hand, suppose $c = 0$, then we must have $a = 0$ and $b \ge 1$, in which case $\frac{2(c+d)c^ad^b}{a!b!\binom{c+d}{c} c^c d^d} \le \frac{2d \cdot d^b}{b! d^d} \le 2$. By symmetry, we extend the argument to the case when $d = 0$. Thus, we conclude that this term is always upper bounded by $2$, and we get }
    &\le 6\left(1 + O(k/n) + O\left(D^2/\sqrt{k}\right)\right) \sum_{\substack{\alpha_1 \in \N^n, t_1 \in \N:\\ |\alpha_1| + t_1 \le D}} \hat{f}_{\alpha_1, t_1}^2 \sum_{\substack{a,b \ge 0,\\a+b\ge 1}} \left(\frac{k \sqrt{D}}{\sqrt{n}}\right)^{a+b} \sum_{c\ge a, d\ge b} \left(\frac{e^2D^2}{\sqrt{k}}\right)^{c+d}\\
    &\le 6\left(1 + O(k/n) + O\left(D^2/\sqrt{k}\right)\right) \sum_{\substack{\alpha_1 \in \N^n, t_1 \in \N:\\ |\alpha_1| + t_1 \le D}} \hat{f}_{\alpha_1, t_1}^2 \sum_{i \ge 1} (i+1)\left(\frac{k \sqrt{D}}{\sqrt{n}}\right)^{i} \sum_{j \ge i} (j+1)\left(\frac{e^2D^2}{\sqrt{k}}\right)^{j}\\
    &\le 6\left(1 + O(k/n) + O\left(D^2/\sqrt{k}\right)\right) \sum_{\substack{\alpha_1 \in \N^n, t_1 \in \N:\\ |\alpha_1| + t_1 \le D}} \hat{f}_{\alpha_1, t_1}^2 \sum_{i \ge 1} (i+1)\left(\frac{k \sqrt{D}}{\sqrt{n}}\right)^{i} \cdot (i+1)\left(\frac{e^2D^2}{\sqrt{k}}\right)^{i}\\
    &= 6\left(1 + O(k/n) + O\left(D^2/\sqrt{k}\right)\right) \sum_{\substack{\alpha_1 \in \N^n, t_1 \in \N:\\ |\alpha_1| + t_1 \le D}} \hat{f}_{\alpha_1, t_1}^2 \sum_{i \ge 1} (i+1)^2\left(\frac{e^2\sqrt{k} D^{\frac{5}{2}}}{\sqrt{n}}\right)^{i}\\
    &\le 24\left(1 + O(k/n) + O\left(D^2/\sqrt{k}\right) + O\left(D^{\frac{5}{2}}\sqrt{k/n}\right)\right) \sum_{\substack{\alpha_1 \in \N^n, t_1 \in \N:\\ |\alpha_1| + t_1 \le D}} \hat{f}_{\alpha_1, t_1}^2 \cdot \frac{e^2\sqrt{k} D^{\frac{5}{2}}}{\sqrt{n}}\\
    &= o\left(\sum_{\substack{\alpha \in \N^n, t \in \N:\\ |\alpha| + t \le D}} \hat{f}_{\alpha, t}^2\right),
\end{align*}
provided that $k = o(n)$ and $D = o\left(\min\left\{k^{1/4},  \left(n/k\right)^{1/5}\right\}\right)$.

Combining the bound on the first and the second term, the contribution of the off-diagonal terms is at most
\begin{align*}
    &\quad \left|\sum_{\substack{\alpha_1, \alpha_2 \in \N^n, t_1, t_2\in \N:\\ |\alpha_1| + t_1 \le D, |\alpha_2| + t_2 \le D,\\ (\alpha_1, t_1) \ne (\alpha_2, t_2)}} \hat{f}_{\alpha_1, t_1} \hat{f}_{\alpha_2, t_2} \E_S \E_X\left[h_{\alpha_1}(X)h_{\alpha_2}(X)h_{t_1}(y)h_{t_2}(y)\right]\right|\\
    &\le \sum_{\substack{\alpha \in \N^n, t_1, t_2\in \N:\\ |\alpha| + t_1 \le D, |\alpha| + t_2 \le D,\\ t_1 \ne t_2}} |\hat{f}_{\alpha, t_1}| |\hat{f}_{\alpha, t_2}| \cdot |\E_S \E_X\left[h_{\alpha}(X)^2h_{t_1}(y)h_{t_2}(y)\right]|\\
    &\quad + \sum_{\substack{\alpha_1, \alpha_2 \in \N^n, t_1, t_2\in \N:\\ |\alpha_1| + t_1 \le D, |\alpha_2| + t_2 \le D,\\ \alpha_1 \ne \alpha_2}} |\hat{f}_{\alpha_1, t_1}| |\hat{f}_{\alpha_2, t_2}| \cdot |\E_S \E_X\left[h_{\alpha_1}(X)h_{\alpha_2}(X)h_{t_1}(y)h_{t_2}(y)\right]|\\
    &\le o\left(\sum_{\substack{\alpha \in \N^n, t \in \N:\\ |\alpha| + t \le D}} \hat{f}_{\alpha, t}^2\right),
\end{align*}
and thus the second moment of $f$ is
\begin{align*}
    &\quad \E\left[f(X,Y)^2\right]\\
    &= \sum_{\substack{\alpha_1, \alpha_2 \in \N^n, t_1, t_2\in \N:\\ |\alpha_1| + t_1 \le D, |\alpha_2| + t_2 \le D}} \hat{f}_{\alpha_1, t_1} \hat{f}_{\alpha_2, t_2} \E_S \E_X\left[h_{\alpha_1}(X)h_{\alpha_2}(X)h_{t_1}(y)h_{t_2}(y)\right]\\
    &\ge \sum_{\substack{\alpha \in \N^n, t \in \N:\\ |\alpha| + t \le D}} \hat{f}_{\alpha, t}^2 \cdot \E_S \E_X\left[h_{\alpha}(X)^2h_{t}(y)^2\right] + o\left(\sum_{\substack{\alpha \in \N^n, t \in \N:\\ |\alpha| + t \le D}} \hat{f}_{\alpha, t}^2\right). \stepcounter{equation}\tag{\theequation}\label{ineq:gss-low-deg-second-moment}
\end{align*}

We next analyze the effect of the noise operator $T_{\rho}$ on $y$. Recall that $T_{\rho}$ acts on $Y$ by averaging $Y$ with a Gaussian variable of variance $k$. The equivalent operator for $y = \frac{1}{\sqrt{k}} Y$ acts on $y$ by taking the average with a standard Gaussian variable $N(0,1)$. More precisely, $T_{\rho}(y) = \sqrt{1 - \rho^2} y + \rho z$, where $z \sim N(0,1)$ is an independent Gaussian. Our goal is now to analyze the second moment under the noisy model $\E[f(X, T_{\rho}(Y))^2]$ and the correlation $\E[f(X,Y)f(X, T_{\rho}(Y))]$.

To do so, we need to understand how the inner expectations behave under the noise operator. We have
\begin{align*}
    &\quad \E_S \E_X\left[h_{\alpha_1}(X)h_{\alpha_2}(X)h_{t_1}(T_{\rho}(y))h_{t_2}(T_{\rho}(y))\right]\\
    &= \E_S\left[\frac{1}{\sqrt{\alpha_1!\alpha_2!t_1!t_2!}} \sum_{M \in \MM(\alpha_1, \alpha_2, t_1, t_2)} \prod_{\{a,b\} \in M} R_{j(a)j(b)}\right],
\end{align*}
where the correlations $R_{ij}$ are different from before due to the application of $T_{\rho}$. We now have
\begin{align*}
    \E_X\left[T_{\rho}(y)^2\right] &= 1\\
    \E_X\left[X_i T_{\rho}(y)\right] &= \frac{\sqrt{1-\rho^2}}{\sqrt{k}} \allone\{i \in S\}.
\end{align*}
Thus, all the edge contributions for the $X$-edges and $Y$-edges stay the same, whereas the contribution of the cross-edges are scaled by a factor of $\sqrt{1-\rho^2}$. In particular, this implies that
\begin{align*}
    \left(\sqrt{1-\rho^2}\right)^D\le \frac{\E_S \E_X\left[h_{\alpha_1}(X)h_{\alpha_2}(X)h_{t_1}(T_{\rho}(y))h_{t_2}(T_{\rho}(y))\right]}{\E_S \E_X\left[h_{\alpha_1}(X)h_{\alpha_2}(X)h_{t_1}(y)h_{t_2}(y)\right]} \le 1,
\end{align*}
for all $(\alpha_1, t_1), (\alpha_2, t_2) \in \N^n \times \N$ such that $|\alpha_1| + t_1 \le D$ and $|\alpha_2| + t_2 \le D$, and we get
\begin{align*}
    &\quad \E[f(X, T_{\rho}(Y))^2]\\
    &= \sum_{\substack{\alpha_1, \alpha_2 \in \N^n, t_1, t_2\in \N:\\ |\alpha_1| + t_1 \le D, |\alpha_2| + t_2 \le D}} \hat{f}_{\alpha_1, t_1} \hat{f}_{\alpha_2, t_2} \E_S \E_X\left[h_{\alpha_1}(X)h_{\alpha_2}(X)h_{t_1}(T_{\rho}(y))h_{t_2}(T_{\rho}(y))\right]\\
    &\le  \sum_{\substack{\alpha \in \N^n, t \in \N:\\ |\alpha| + t \le D}} \hat{f}_{\alpha, t}^2 \cdot \E_S \E_X\left[h_{\alpha}(X)^2h_{t}(T_{\rho}(y))^2\right]\\
    &\quad + \sum_{\substack{\alpha_1, \alpha_2 \in \N^n, t_1, t_2\in \N:\\ |\alpha_1| + t_1 \le D, |\alpha_2| + t_2 \le D,\\ (\alpha_1, t_1) \ne (\alpha_2, t_2)}} |\hat{f}_{\alpha_1, t_1}| |\hat{f}_{\alpha_2, t_2}| \cdot \E_S \E_X\left[h_{\alpha_1}(X)h_{\alpha_2}(X)h_{t_1}(T_{\rho}(y))h_{t_2}(T_{\rho}(y))\right]\\
    &\le \sum_{\substack{\alpha \in \N^n, t \in \N:\\ |\alpha| + t \le D}} \hat{f}_{\alpha, t}^2 \cdot \E_S \E_X\left[h_{\alpha}(X)^2h_{t}(y)^2\right] + o\left(\sum_{\substack{\alpha \in \N^n, t \in \N:\\ |\alpha| + t \le D}} \hat{f}_{\alpha, t}^2\right)\\
    &\le (1 + o(1))\E\left[f(X,Y)^2\right],\stepcounter{equation}\tag{\theequation}\label{ineq:gss-low-deg-noisy-second-moment}
\end{align*}
where the last inequality follows from $\E_S \E_X\left[h_{\alpha}(X)^2h_{t}(y)^2\right] \ge 1$.

Similarly, we proceed to analyze the correlation term $\E[f(X,Y)f(X,T_{\rho}(Y))]$. We have
\begin{align*}
    &\quad \E_S \E_X\left[h_{\alpha_1}(X)h_{\alpha_2}(X)h_{t_1}(y)h_{t_2}(T_{\rho}(y))\right]\\
    &= \E_S\left[\frac{1}{\sqrt{\alpha_1!\alpha_2!t_1!t_2!}} \sum_{M \in \MM(\alpha_1, \alpha_2, t_1, t_2)} \prod_{\{a,b\} \in M} R_{j(a)j(b)}\right],
\end{align*}
where the correlations $R_{ij}$ are again different. We now have
\begin{align*}
    \E_X\left[yT_{\rho}(y)\right] &= \sqrt{1-\rho^2}\\
    \E_X\left[X_i y\right] &= \frac{1}{\sqrt{k}} \allone\{i \in S\}\\
    \E_X\left[X_i T_{\rho}(y)\right] &= \frac{\sqrt{1-\rho^2}}{\sqrt{k}} \allone\{i \in S\}.
\end{align*}
Thus, all the edge contributions for the $X$-edges and the cross-edges between vertices corresponding to $X_i$ and $y$ stay the same, whereas the contribution of the $Y$-edges and the cross-edges between vertices corresponding to $X_i$ and $T_{\rho}(y)$ are scaled by a factor of $\sqrt{1-\rho^2}$. We again have \begin{align*}
    \left(\sqrt{1-\rho^2}\right)^D\le \frac{\E_S \E_X\left[h_{\alpha_1}(X)h_{\alpha_2}(X)h_{t_1}(y)h_{t_2}(T_{\rho}(y))\right]}{\E_S \E_X\left[h_{\alpha_1}(X)h_{\alpha_2}(X)h_{t_1}(y)h_{t_2}(y)\right]} \le 1,
\end{align*}
for all $(\alpha_1, t_1), (\alpha_2, t_2) \in \N^n \times \N$ such that $|\alpha_1| + t_1 \le D$ and $|\alpha_2| + t_2 \le D$, and therefore, 
\begin{align*}
    &\quad \E\left[f(X,Y)f(X,T_{\rho}(Y))\right]\\
    &= \sum_{\substack{\alpha_1, \alpha_2 \in \N^n, t_1, t_2\in \N:\\ |\alpha_1| + t_1 \le D, |\alpha_2| + t_2 \le D}} \hat{f}_{\alpha_1, t_1} \hat{f}_{\alpha_2, t_2} \E_S \E_X\left[h_{\alpha_1}(X)h_{\alpha_2}(X)h_{t_1}(y)h_{t_2}(T_{\rho}(y))\right]\\
    &\ge  \sum_{\substack{\alpha \in \N^n, t \in \N:\\ |\alpha| + t \le D}} \hat{f}_{\alpha, t}^2 \cdot \E_S \E_X\left[h_{\alpha}(X)^2h_{t}(y)h_{t}(T_{\rho}(y))\right]\\
    &\quad - \sum_{\substack{\alpha_1, \alpha_2 \in \N^n, t_1, t_2\in \N:\\ |\alpha_1| + t_1 \le D, |\alpha_2| + t_2 \le D,\\ (\alpha_1, t_1) \ne (\alpha_2, t_2)}} |\hat{f}_{\alpha_1, t_1}| |\hat{f}_{\alpha_2, t_2}| \cdot \E_S \E_X\left[h_{\alpha_1}(X)h_{\alpha_2}(X)h_{t_1}(y)h_{t_2}(T_{\rho}(y))\right]\\
    &\ge \left(\sqrt{1-\rho^2}\right)^D\sum_{\substack{\alpha \in \N^n, t \in \N:\\ |\alpha| + t \le D}} \hat{f}_{\alpha, t}^2 \cdot \E_S \E_X\left[h_{\alpha}(X)^2h_{t}(y)^2\right] - o\left(\sum_{\substack{\alpha \in \N^n, t \in \N:\\ |\alpha| + t \le D}} \hat{f}_{\alpha, t}^2\right)\\
    &\ge \left(\left(\sqrt{1-\rho^2}\right)^D - o(1)\right)\E\left[f(X,Y)^2\right],\stepcounter{equation}\tag{\theequation}\label{ineq:gss-low-deg-correlation}
\end{align*}
where we again use $ \E_S \E_X\left[h_{\alpha}(X)^2h_{t}(y)^2\right] \ge 1$ in the last inequality.

Combining \eqref{ineq:gss-low-deg-second-moment}, \eqref{ineq:gss-low-deg-noisy-second-moment}, and \eqref{ineq:gss-low-deg-correlation}, we get
\begin{align*}
    &\quad \E\left[\left(f(X,Y) - f(X, T_{\rho}(Y))\right)^2\right]\\
    &= \E\left[f(X,Y)^2\right] + \E\left[f(X,T_{\rho}(Y))^2\right] - 2\E\left[f(X,Y)f(X, T_{\rho}(Y))\right]\\
    &\le 2\left(1 - \left(\sqrt{1-\rho^2}\right)^D + o(1)\right)\E\left[f(X,Y)^2\right].
\end{align*}
This confirms that polynomial $f$ of degree at most $D$ is $\left(\rho, 2\left(1 - \left(\sqrt{1-\rho^2}\right)^D + o(1)\right)\right)$-stable for the GSS, provided that $D = o\left(\min\left\{k^{1/4},  \left(n/k\right)^{1/5}\right\}\right)$.

\end{proof}

\subsection{Deferred Proof for Theorem~\ref{thm:wick-product-formula}} \label{sec:deferred-proof-wick-product}

\begin{proof}
    Clearly, if $|\alpha| = \sum_i \alpha_i$ is odd, then the corresponding product of Hermite polynomials have zero expectation. Therefore, from now on we assume $|\alpha|$ is even. Recall that the generating function of the probabilist's Hermite polynomials is
    \begin{align*}
        \exp\left(tx - \frac{1}{2}t^2\right) = \sum_{n=0}^\infty \text{He}_n(x) \frac{t^n}{n!}.
    \end{align*}

    Let $x_1, \dots, x_k$ be Gaussian variables identically distributed as $N(0,1)$, whose covariances are $R_{ij} = \E[x_ix_j]$. Consider the product of generating functions
    \begin{align*}
        Z(t_1, \dots, t_k) &= \exp\left(\sum_{i=1}^k \left(t_i x_i - \frac{1}{2} t_i^2\right)\right)\\
        &= \exp\left(- \frac{1}{2} \sum_{i=1}^k t_i^2\right) \exp\left(\sum_{i=1}^k t_ix_i\right).
    \end{align*}
    Taking the expectation of $Z$, we get
    \begin{align*}
        &\E\left[Z(t_1, \dots, t_k)\right]\\
        &=\exp\left(- \frac{1}{2} \sum_{i=1}^k t_i^2\right) \E\left[\exp\left(\sum_{i=1}^k t_ix_i\right)\right]\\
        &= \exp\left(- \frac{1}{2} \sum_{i=1}^k t_i^2\right) \exp\left(\frac{1}{2}\E\left[\left(\sum_{i=1}^k t_i x_i\right)^2\right]\right),
    \end{align*}
    using the moment generating function of Gaussians. Expanding the expectation inside the exponential using the covariances $R_{ij}$, we get
    \begin{align*}
        &\E\left[Z(t_1, \dots, t_k)\right]\\
        &= \exp\left(- \frac{1}{2} \sum_{i=1}^k t_i^2\right) \exp\left(\frac{1}{2}\E\left[\left(\sum_{i=1}^k t_i x_i\right)^2\right]\right)\\
        &= \exp\left(- \frac{1}{2} \sum_{i=1}^k t_i^2\right) \exp\left(\frac{1}{2}\sum_{i=1}^k t_i^2 + \sum_{1 \le i < j \le k} t_i t_j R_{ij}\right)\\
        &= \exp\left(\sum_{1 \le i < j \le k} t_i t_j R_{ij}\right)\\
        &= \sum_{n=0}^\infty \frac{1}{n!}\left(\sum_{1 \le i < j \le k} t_i t_j R_{ij}\right)^n.
    \end{align*}

    Differentiating the expectation and evaluating at $0$, we get
    \begin{align*}
        &\frac{\partial^{\alpha_1}}{\partial t_1^{\alpha_1}} \frac{\partial^{\alpha_2}}{\partial t_2^{\alpha_2}} \dots \frac{\partial^{\alpha_k}}{\partial t_k^{\alpha_k}} \E\left[Z(t_1, \dots, t_k)\right] \Bigg\vert_{t_1 = \dots = t_k = 0}\\
        &= \prod_{i=1}^k \alpha_i! \cdot \frac{1}{(|\alpha|/2)!} \sum_{\substack{(a_i, b_i)_{i=1}^{|\alpha|/2}:\\ 1\le a_i < b_i \le k,\\ \forall j \in [n], \alpha_j = \\ \sum_{i=1}^k \allone\{a_i = j\} + \allone\{b_i = j\}}} \prod_{i=1}^k R_{a_i, b_i}\\
        &= \sum_{M \in \MM(\alpha)} \prod_{\{a,b\} \in M} R_{j(a), j(b)},
    \end{align*}
    where $\MM(\alpha)$ is the set of perfect matchings in the constructed graph corresponding to $\alpha$.

    Finally, since we work with the orthonormal Hermite polynomials $h_n(x)$ normalized by $h_n(x) := \frac{1}{\sqrt{n!}}\text{He}_n(x)$, we have that
    \begin{align*}
        &\E\left[\prod_{i=1}^k h_{\alpha_i}(x_i)\right]\\
        &= \frac{1}{\sqrt{\alpha!}} \E\left[\prod_{i=1}^k \text{He}_{\alpha_i}(x_i)\right]\\
        &= \frac{1}{\sqrt{\alpha!}} \E\left[\prod_{i=1}^k \left[\frac{\partial^{\alpha_i}}{\partial t_i^{\alpha_i}} \exp\left(t_i x_i - \frac{1}{2}t_i^2\right)\right] \Big \vert_{t_i = 0}\right]\\
        &= \frac{1}{\sqrt{\alpha!}} \frac{\partial^{\alpha_1}}{\partial t_1^{\alpha_1}}\frac{\partial^{\alpha_2}}{\partial t_2^{\alpha_2}} \dots \frac{\partial^{\alpha_k}}{\partial t_k^{\alpha_k}}\E\left[\prod_{i=1}^k  \exp\left(t_i x_i - \frac{1}{2}t_i^2\right)\right] \Bigg \vert_{t_1 = \dots = t_k = 0}\\
        &= \frac{1}{\sqrt{\alpha!}} \frac{\partial^{\alpha_1}}{\partial t_1^{\alpha_1}}\frac{\partial^{\alpha_2}}{\partial t_2^{\alpha_2}} \dots \frac{\partial^{\alpha_k}}{\partial t_k^{\alpha_k}}\E\left[Z(t_1, \dots, t_k)\right] \Bigg \vert_{t_1 = \dots = t_k = 0}\\
        &= \frac{1}{\sqrt{\alpha!}}\sum_{M \in \MM(\alpha)} \prod_{\{a,b\} \in M} R_{j(a), j(b)},
    \end{align*}
    as desired. This finishes the proof.
\end{proof}
\newpage

\bibliographystyle{alpha}
\bibliography{main}

\newpage
\appendix

\section{Stable Algorithms Failure for Sparse Tensor PCA}

Based on our results so far, it is natural to wonder whether our technique is only applicable for ``counterexamples" problems, where there is a polynomial-time method but stable algorithms are failing in the task. In this section, we show that this is not true, by proving via our technique that stable methods are failing for the \emph{standard sparse tensor PCA setting} in the ``computational-statistical gap" regime where no polynomial-time method is expected to work. On the way to establishing this result, we obtain a tight refinement on the window size of the all-or-nothing phenomenon exhibited by the model, improving upon the prior work \cite{niles2020all} with a different method.

The model is as follows.

\begin{definition}
    Let $k,d, n \in \N$ and $\lambda \in \R$ with $k \le n$. In the $k$-Sparse Tensor PCA problem, we observe $Y = \sqrt{\lambda} x^{\otimes d} + W$, where $x$ is drawn uniformly at random from the set of $k$-sparse vectors in $\{0, 1\sqrt{k}\}^n \subseteq \R^n$ with exactly $k$ nonzero entries, and $W \in (\R^n)^{\otimes d}$ is a tensor with i.i.d.~$N(0,1)$ entries. The goal of the statistician is to estimate $x$ from $Y$.
\end{definition}

The noise operator we consider is again the OU operator.

\begin{remark}[The effect of OU operator for Sparse Tensor PCA]
    It is easy to see that applying the OU noise operator $T_{\rho}$ to Sparse Tensor PCA at $\lambda$ is equivalent to observing a new Sparse Tensor PCA model with the same $k, d, n$, but with a new $\Tilde{\lambda} = \lambda (1 - \rho^2)$.
\end{remark}

While it is known that Sparse Tensor PCA exhibits the AoN phenomenon at the sharp threshold $\lambda_c = 2\log \binom{n-k}{k}$ \cite{niles2020all}, the size of the critical window from prior work is only of size $o(\lambda_c)$, which is too weak for us to apply our framework. Therefore, we first obtain the following result on the size of the critical window for certain Sparse Tensor PCA model, the proof of which is deferred to Section~\ref{sec:critical-window-PCA}.

\begin{proposition}\label{prop:critical-window-PCA}
    Fix $\eps > 0$. Suppose $d \ge 3$ and $k \le \min\left\{n^{\frac{1}{2} - \eps}, n^{\frac{d-2}{d+2} - \eps}\right\}$. Then, for any $\Delta > 0$ that satisfies
    \begin{align*}
        \Delta = \omega\left(\sqrt{\log \binom{n-k}{k}}\right),
    \end{align*}
    it holds in the Sparse Tensor PCA model that
    \begin{itemize}
        \item if $\lambda = 2\log \binom{n-k}{k} + \Delta$, then $\text{MMSE} \to 0$,
        \item if $\lambda = 2\log \binom{n-k}{k} - \Delta$, then $\text{MMSE} \to 1$.
    \end{itemize}
\end{proposition}

In addition to showing the stable algorithm separation, we also obtain hardness for low-degree polynomials by showing their stability. 

\begin{theorem}\label{thm:low-deg-stability-PCA}
    Let $f: \R^{[n]^d} \to \R$ be a polynomial of degree at most $D$. Suppose that $D = o\left(\min \left\{\left(\frac{k^d}{\lambda}\right)^{\frac{1}{d+\frac{1}{2}}}, \frac{n}{k^2}\right\}\right)$. Then, $f$ is $\left(\rho, 2\left(1 - \left(\sqrt{1-\rho^2}\right)^D + o(1)\right)\right)$-stable for the Sparse Tensor PCA problem parametrized by $n,k,d$ and $\lambda$.
\end{theorem}

The result again follows by a intricate combinatorial argument together with a diagram formula for the expectation of Hermite polynomials. The proof is deferred to Section~\ref{sec:low-deg-stability-PCA}.

As a corollary, we get the following low-degree MMSE lower bound for Sparse Tensor PCA when $\lambda$ is close to the critical window of the model.

\begin{corollary}
    Fix $\eps > 0$. Suppose $d \ge 3$, $k \le \min\left\{n^{\frac{1}{2} - \eps}, n^{\frac{d-2}{d+2} - \eps}\right\}$, and $\Delta > 0$ such that
    \begin{align*}
        \Delta = \omega\left(\sqrt{\log \binom{n-k}{k}}\right).
    \end{align*}
    Set $\lambda = 2\log \binom{n-k}{k} + \Delta$. Note that by Proposition~\ref{prop:critical-window-PCA}, the MMSE goes to $0$ for this $\lambda$.
    If $D = o\left(\min \left\{\left(\frac{k^d}{\lambda}\right)^{\frac{1}{d+\frac{1}{2}}}, \frac{n}{k^2}\right\}\right)$ and $\rho \ge \sqrt{\frac{2\Delta}{\lambda}}$, then any degree-$D$ polynomial $f: \R^{[n]^d} \to \R^n$ has a mean squared error at least
    \begin{align*}
        \EE\left[\|f(Y) - x\|_2^2\right] \ge 1 - O\left(\sqrt{1 - \left(\sqrt{1-\rho^2}\right)^2}\right) - o(1).
    \end{align*}
    In particular, if $D \le \frac{1}{\rho^2}$, we have
    \begin{align*}
        \EE\left[\|f(Y) - x\|_2^2\right] \ge 1 - O\left(\rho^2 D\right) - o(1).
    \end{align*}
\end{corollary}
\subsection{Proofs}

We will need the following theorem, whose proof is almost identical to Theorem~\ref{thm:wick-product-formula}.

\begin{theorem}[Diagram Formula for Expectation of Products of Hermite Polynomials]\label{thm:wick-product-formula-mean}
    
    Suppose $x_1, \dots, x_k$ are Gaussian variables distributed as $x_i \sim N(\mu_i,1)$. Let $R_{ij} = \E[x_ix_j]$ denote their correlations. Given $\alpha \in \N^k$, we define the following graph $G(\alpha) = (V, E)$ on $|\alpha|$ vertices. For each $i \in [k]$, create $\alpha_i$ vertices corresponding to $x_i$, and let $j(v) = i$ for each of the $\alpha_i$ vertices $v$ associated with $x_i$. For $u, v \in V$, add an edge $\{u,v\}$ if $j(u) \ne j(v)$.
Then, we have
\begin{align*}
    \E\left[\prod_{i=1}^k h_{\alpha_i}(x_i)\right] = \frac{1}{\sqrt{\alpha!}} \sum_{M \in \MM(\alpha)} \prod_{v \not\in V(M)} \mu_{j(v)}\prod_{\{a,b\} \in M} R_{j(a)j(b)},
\end{align*}
where $\MM(\alpha)$ denotes the collection of (partial) matchings in the graph $G(\alpha)$ described above.
\end{theorem}

\subsubsection{Stability of Low-degree Polynomails} \label{sec:low-deg-stability-PCA}

\begin{proof}[Proof of Theorem~\ref{thm:low-deg-stability-PCA}]
    Consider a polynomial $f$ of degree at most $D$, expressed uniquely as
    \begin{align*}
        f(Y) = \sum_{\alpha \in \N^{n^d}: |\alpha| \le D} \hat{f}_{\alpha} h_{\alpha}(Y).
    \end{align*}

    We will lower bound $\EE[f(Y)^2]$ and $\EE[f(Y)f(T_{\rho}(Y))]$, and upper bound $\EE[f(T_{\rho}(Y))^2]$.

    First we apply the diagram formula to compute for $\alpha_1, \alpha_2 \in \N^{[n]^d}$:
    \begin{align*}
        &\quad \EE[h_{\alpha_1}(Y)h_{\alpha_2}(Y)]\\
        &= \EE_x \EE[h_{\alpha_1}(Y)h_{\alpha_2}(Y) \big\vert x]\\
        &= \EE_x\left[\frac{1}{\sqrt{\alpha_1!\alpha_2!}} \sum_{M \in \MM(\alpha_1, \alpha_2)} \prod_{v \not\in V(M)} \mu_{j(v)} \prod_{\{a,b\} \in M} R_{j(a)j(b)} \Bigg\vert x\right].
    \end{align*}
    From this formula, we see that all the inner expectations are nonnegative. Next, we analyze the matchings in the graph $G(\alpha_1, \alpha_2)$. Let $S \subset [n]$ denote the support of the random vector $x$. Note that in the sparse PCA, $x$ is uniquely determined by $S$ and vice versa.

    For a fixed $x$, we see that each entry $Y = \sqrt{\lambda} x^{\otimes d} + W$ is independently distributed as $Y_{\overline{i}} \sim N(X_{\overline{i}}, 1)$, where $X_{\overline{i}} = \sqrt{\lambda} (x^{\otimes d})_{\overline{i}}$ for $\overline{i} \in [n]^d$. As a result, a contributing matching $M \in \MM(\alpha_1, \alpha_2)$ can only contain edges connecting one vertex corresponding to $Y_{\overline{i}}$ from $h_{\alpha_1}$ and the other vertex corresponding to the same $Y_{\overline{i}}$ from $h_{\alpha_2}$, since for a fixed $x$, $Y_{\overline{i}}$ and $Y_{\overline{j}}$ are uncorrelated if $\overline{i} \ne \overline{j}$. Thus, the inner expectation above tensorizes into a product over each coordinate:
    \begin{align*}
        &\quad \EE_x\left[\frac{1}{\sqrt{\alpha_1!\alpha_2!}} \sum_{M \in \MM(\alpha_1, \alpha_2)} \prod_{v \not\in V(M)} \mu_{j(v)} \prod_{\{a,b\} \in M} R_{j(a)j(b)} \Bigg\vert x\right]\\
        &= \EE_x \left[\frac{1}{\sqrt{\alpha_1!\alpha_2!}}\prod_{\overline{i} \in [n]^d} \sum_{M \in \MM(\alpha_1(\overline{i}), \alpha_2(\overline{i}))} \prod_{v \not\in V(M)} X_{\overline{i}} \Bigg \vert x\right]\\
        &= \EE_x\left[\frac{1}{\sqrt{\alpha_1!\alpha_2!}}\prod_{\overline{i} \in [n]^d} \sum_{t = 0}^{\alpha_1(\overline{i}) \wedge \alpha_2(\overline{i})} \binom{\alpha_1(\overline{i})}{t} \binom{\alpha_2(\overline{i})}{t}t! \cdot X_{\overline{i}}^{\alpha_1(\overline{i}) + \alpha_2(\overline{i}) - 2t} \Bigg \vert x\right].
    \end{align*}
    
    The diagonal terms with $\alpha_1 = \alpha_2 = \alpha$ is at least
    \begin{align*}
        &\quad \EE[h_{\alpha}(Y)^2]\\
        &= \EE_x\left[\frac{1}{\alpha!} \prod_{\overline{i} \in [n]^d} \sum_{t = 0}^{\alpha(\overline{i})} \binom{\alpha(\overline{i})}{t}^2 t! \cdot X_{\overline{i}}^{2\alpha(\overline{i}) - 2t} \Bigg \vert x\right]\\
        &\ge \EE_x\left[\frac{1}{\alpha!} \prod_{\overline{i} \in [n]^d} \alpha(\overline{i})! \Bigg \vert x\right]\\
        &= 1.
    \end{align*}
    The off-diagonal terms with $\alpha_1 \ne \alpha_2$ can be upper bounded by
    \begin{align*}
        &\quad \EE[h_{\alpha_1}(Y)h_{\alpha_2}(Y)]\\
        &= \EE_x\left[\frac{1}{\sqrt{\alpha_1!\alpha_2!}}\prod_{\overline{i} \in [n]^d} \sum_{t = 0}^{\alpha_1(\overline{i}) \wedge \alpha_2(\overline{i})} \binom{\alpha_1(\overline{i})}{t} \binom{\alpha_2(\overline{i})}{t}t! \cdot X_{\overline{i}}^{\alpha_1(\overline{i}) + \alpha_2(\overline{i}) - 2t} \Bigg \vert x\right]
        \intertext{Now we observe that the ratio between the consecutive terms of the inner sum of $t$ is $(\binom{\alpha_1(\overline{i})}{t+1} \binom{\alpha_2(\overline{i})}{t+1}(t+1)! \cdot X_{\overline{i}}^{\alpha_1(\overline{i}) + \alpha_2(\overline{i}) - 2(t+1)})/(\binom{\alpha_1(\overline{i})}{t} \binom{\alpha_2(\overline{i})}{t}t! \cdot X_{\overline{i}}^{\alpha_1(\overline{i}) + \alpha_2(\overline{i}) - 2t}) =  \frac{(\alpha_1(\overline{i})-t)(\alpha_2(\overline{i})-t)}{X_{\overline{i}}^2 (t+1)} \ge (X_{\overline{i}}^2 (t+1))^{-1}$. Note that $t+1 \le D$ and $X_{\overline{i}}^2 = (1/\sqrt{k})^{2d}= 1/k^d$ when $X_{\overline{i}}$ is nonzero. Therefore, as long as $D = o(k^d)$, the inner sum is dominated by the last term at $t = \alpha_1(\overline{i}) \wedge \alpha_2(\overline{i})$, and we have}
        &\le \EE_x\left[\frac{1}{\sqrt{\alpha_1!\alpha_2!}}\prod_{\overline{i} \in [n]^d} \left(1 + O\left(\frac{k^d}{D}\right)\right) \binom{\alpha_1(\overline{i}) \vee \alpha_2(\overline{i})}{\alpha_1(\overline{i}) \wedge \alpha_2(\overline{i})}(\alpha_1(\overline{i}) \wedge \alpha_2(\overline{i}))! X_{\overline{i}}^{2|\alpha_1(\overline{i}) - \alpha_2(\overline{i})|} \Bigg \vert x\right]\\
        &= \EE_x \left[\left(1 + O\left(\frac{k^d}{D}\right)\right) \frac{1}{\sqrt{\alpha_1!\alpha_2!}}\prod_{\overline{i} \in [n]^d}  \frac{(\alpha_1(\overline{i}) \vee \alpha_2(\overline{i}))!}{|\alpha_1(\overline{i})-\alpha_2(\overline{i})|!} X_{\overline{i}}^{2|\alpha_1(\overline{i}) - \alpha_2(\overline{i})|}\Bigg \vert x\right]\\
        &= \EE_x \left[\left(1 + O\left(\frac{k^d}{D}\right)\right) \prod_{\overline{i} \in [n]^d}  \sqrt{\frac{(\alpha_1(\overline{i}) \vee \alpha_2(\overline{i}))!}{(\alpha_1(\overline{i}) \wedge \alpha_2(\overline{i}))!}}\frac{1}{|\alpha_1(\overline{i})-\alpha_2(\overline{i})|!} X_{\overline{i}}^{2|\alpha_1(\overline{i}) - \alpha_2(\overline{i})|}\Bigg \vert x\right]\\
        &\le \EE_x \left[\left(1 + O\left(\frac{k^d}{D}\right)\right) \prod_{\substack{\overline{i} \in [n]^d:\\ \alpha_1(\overline{i}) \ne \alpha_2(\overline{i})}}  \left(\frac{e\sqrt{\alpha_1(\overline{i}) \wedge \alpha_2(\overline{i})}}{|\alpha_1(\overline{i}) - \alpha_2(\overline{i})|} \cdot X_{\overline{i}}^2\right)^{|\alpha_1(\overline{i}) - \alpha_2(\overline{i})|} \vast \vert x\right]\\
        &\le \EE_x \left[\left(1 + O\left(\frac{k^d}{D}\right)\right) \prod_{\substack{\overline{i} \in [n]^d:\\ \alpha_1(\overline{i}) \ne \alpha_2(\overline{i})}}  \left( e\sqrt{D}X_{\overline{i}}^2\right)^{|\alpha_1(\overline{i}) \triangle \alpha_2(\overline{i})|} \vast \vert x\right]\\
        &\le \left(1 + O\left(\frac{k^d}{D}\right)\right) (e\sqrt{D})^{|\alpha_1 - \alpha_2|} \EE_x \left[\prod_{\substack{\overline{i} \in [n]^d:\\ \alpha_1(\overline{i}) \ne \alpha_2(\overline{i})}}  X_{\overline{i}}^{2|\alpha_1(\overline{i}) - \alpha_2(\overline{i})|} \vast\vert x\right]
        \intertext{Note that $X_{\overline{i}} = \sqrt{\lambda} (x^{\otimes d})_{\overline{i}}$ is nonzero only when $\overline{i} \in S^d$, where $S \subseteq [n]$ is the support of $x$, and whenever $\overline{i} \in S^d$, $X_{\overline{i}}= \sqrt{\lambda}/ k^{d/2}$. Thus, we get}
        &= \left(1 + O\left(\frac{k^d}{D}\right)\right) (e\sqrt{D})^{|\alpha_1 \triangle \alpha_2|} \left(\frac{\lambda}{k^d}\right)^{|\alpha_1 \triangle \alpha_2|} \Pr(V(\alpha_1 \triangle\alpha_2) \subseteq S)\\
        &\le \left(1 + O\left(\frac{k^d}{D}\right)\right) \left(\frac{e\sqrt{D}\lambda}{k^d}\right)^{|\alpha_1 \triangle \alpha_2|} \left(\frac{k}{n}\right)^{v(\alpha_1 \triangle \alpha_2)}.
    \end{align*}

    Now, let us lower bound $\EE[f(Y)^2]$. We have
    \begin{align*}
        &\quad \EE[f(Y)^2]\\
        &= \sum_{\substack{\alpha_1, \alpha_2 \in \N^{[n]^d}:\\ |\alpha_1|, |\alpha_2| \le D}} \hat{f}_{\alpha_1}\hat{f}_{\alpha_2} \EE\left[h_{\alpha_1}(Y)h_{\alpha_2}(Y)\right]\\
        &= \sum_{\substack{\alpha \in \N^{[n]^d}: |\alpha| \le D}} \hat{f}_{\alpha}^2 \EE\left[h_{\alpha}(Y)^2\right] + \sum_{\substack{\alpha_1, \alpha_2 \in \N^{[n]^d}:\\ |\alpha_1|, |\alpha_2| \le D,\\ \alpha_1 \ne \alpha_2}} \hat{f}_{\alpha_1}\hat{f}_{\alpha_2} \EE\left[h_{\alpha_1}(Y)h_{\alpha_2}(Y)\right].
    \end{align*}
    We know that the contribution of the diagonal terms is at least
    \begin{align*}
        &\quad \sum_{\substack{\alpha \in \N^{[n]^d}: |\alpha| \le D}} \hat{f}_{\alpha}^2 \EE\left[h_{\alpha}(Y)^2\right]\\
        &\ge \sum_{\substack{\alpha \in \N^{[n]^d}: |\alpha| \le D}} \hat{f}_{\alpha}^2.
    \end{align*}
    Now let us consider the off-diagonal terms. For fixed $\alpha_1, \alpha_2 \in \N^{[n]^d}$, we denote
    \begin{align}
        v_1 &= |V(\alpha_1 \triangle \alpha_2) - V(\alpha_2)|\\
        v_2 &= |V(\alpha_1 \triangle \alpha_2) - V(\alpha_1)|\\
        m &= |\alpha_1 \triangle \alpha_2|.
    \end{align}
    Clearly, we have $v_1 + v_2 \le |V(\alpha_1 \triangle \alpha_2)|$, and if $\alpha_1 \ne \alpha_2$, then $m \ge 1$. Let $W_{\alpha_1}(v_1, v_2, m)$ denote the collection of $\alpha_2 \in \N^{[n]^d}$ that satisfies the parameter above and $\alpha_2 \ne \alpha_1$. Then, the contribution of the off-diagonal terms is at most
    \begin{align*}
        &\quad \left|\sum_{\substack{\alpha_1, \alpha_2 \in \N^{[n]^d}:\\ |\alpha_1|, |\alpha_2| \le D,\\ \alpha_1 \ne \alpha_2}} \hat{f}_{\alpha_1}\hat{f}_{\alpha_2} \EE\left[h_{\alpha_1}(Y)h_{\alpha_2}(Y)\right]\right|\\
        &\le \sum_{\substack{\alpha_1, \alpha_2 \in \N^{[n]^d}:\\ |\alpha_1|, |\alpha_2| \le D,\\ \alpha_1 \ne \alpha_2}} |\hat{f}_{\alpha_1}||\hat{f}_{\alpha_2}|\cdot \left| \EE\left[h_{\alpha_1}(Y)h_{\alpha_2}(Y)\right]\right|\\
        &\le 2\sum_{\substack{\alpha_1 \in \N^{[n]^d}:\\|\alpha_1| \le D}} |\hat{f}_{\alpha_1}| \sum_{\substack{v_1, v_2 \ge 0,\\m \ge 1}}\sum_{\substack{\alpha_2 \in W_{\alpha_1}(v_1, v_2, m): \\ |\hat{f}_{\alpha_2}|n^{\frac{v_2}{2}}(dD)^{\frac{v_1}{2}}\\ \le |\hat{f}_{\alpha_1}|n^{\frac{v_1}{2}}(dD)^{\frac{v_2}{2}}}} |\hat{f}_{\alpha_2}| \cdot \EE\left[h_{\alpha_1}(Y)h_{\alpha_2}(Y)\right]
        \intertext{where the inequality holds because every ordered pair of $(\alpha_1, \alpha_2)$ may be reordered so that $|\hat{f}_{\alpha_2}|n^{\frac{v_2}{2}}(dD)^{\frac{v_1}{2}} \le |\hat{f}_{\alpha_1}|n^{\frac{v_1}{2}}(dD)^{\frac{v_2}{2}}$,}
        &\le 2\sum_{\substack{\alpha_1 \in \N^{[n]^d}:\\|\alpha_1| \le D}} \hat{f}_{\alpha_1}^2 \sum_{\substack{v_1, v_2 \ge 0,\\m \ge 1}} n^{\frac{v_1 - v_2} {2}}(dD)^{\frac{v_2 - v_1}{2}}\sum_{\substack{\alpha_2 \in W_{\alpha_1}(v_1, v_2, m): \\ |\hat{f}_{\alpha_2}|n^{\frac{v_2}{2}}(dD)^{\frac{v_1}{2}}\\ \le |\hat{f}_{\alpha_1}|n^{\frac{v_1}{2}}(dD)^{\frac{v_2}{2}}}} \left(1 + O\left(\frac{k^d}{D}\right)\right) \left(\frac{e\sqrt{D}\lambda}{k^d}\right)^{|\alpha_1 \triangle \alpha_2|} \left(\frac{k}{n}\right)^{v(\alpha_1 \triangle \alpha_2)}\\
        &\le \left(2 + O\left(\frac{k^d}{D}\right)\right) \sum_{\substack{\alpha_1 \in \N^{[n]^d}:\\|\alpha_1| \le D}} \hat{f}_{\alpha_1}^2 \sum_{\substack{v_1, v_2 \ge 0,\\m \ge 1}} n^{\frac{v_1 - v_2}{2}}(dD)^{\frac{v_2 - v_1}{2}}\sum_{\substack{\alpha_2 \in W_{\alpha_1}(v_1, v_2, m): \\ |\hat{f}_{\alpha_2}|n^{\frac{v_2}{2}}(dD)^{\frac{v_1}{2}}\\ \le |\hat{f}_{\alpha_1}|n^{\frac{v_1}{2}}(dD)^{\frac{v_2}{2}}}}\left(\frac{e\sqrt{D}\lambda}{k^d}\right)^{m} \left(\frac{k}{n}\right)^{v(\alpha_1 \triangle \alpha_2)}.
    \end{align*}
    Next, we bound the number of $\alpha_2 \in W_{\alpha_1}(v_1, v_2, m)$. To enumerate such $\alpha_2$, we first enumerate the $v_2$ indices in $V(\alpha_1 \triangle \alpha_2) - V(\alpha_1)$, which gives at most $n^{v_2}$ choices. Then, we enumerate the $v_1$ indices in $V(\alpha_1 \triangle \alpha_2) - V(\alpha_2)$, which gives at most $(dD)^{v_1}$ choices, since $V(\alpha_1 \triangle \alpha_2) - V(\alpha_2) \subseteq V(\alpha_1)$ and $|V(\alpha_1)| \le d|\alpha_1| \le dD$. Finally, given the previous choices that uniquely determine $V(\alpha_1) \cup V(\alpha_2)$, we enumerate $\alpha_1 \triangle \alpha_2$ on $V(\alpha_1) \cup V(\alpha_2)$. Note that the number of choices of $\alpha_1 \triangle \alpha_2$ on $V(\alpha_1) \cup V(\alpha_2)$ is upper bounded by the number of ways to choose $m = |\alpha_1 \triangle \alpha_2|$ tuples of $\binom{V(\alpha_1) \cup V(\alpha_2)}{d}$ with replacement, giving at most $\binom{|V(\alpha_1) \cup V(\alpha_2)|}{d}^m \le \binom{d(|\alpha_1| + |\alpha_2|)}{d}^m \le \binom{2dD}{d}^m$ choices. Thus, $n^{v_2} (dD)^{v_1} \binom{2dD}{d}^m$ is an upper bound on the number of $\alpha_2 \in W_{\alpha_1}(v_1, v_2, m)$. Returning to the bound on the off-diagonal terms, we get
    \begin{align*}
        &\quad \left(2 + O\left(\frac{k^d}{D}\right)\right) \sum_{\substack{\alpha_1 \in \N^{[n]^d}:\\|\alpha_1| \le D}} \hat{f}_{\alpha_1}^2 \sum_{\substack{v_1, v_2 \ge 0,\\m \ge 1}} n^{\frac{v_1 - v_2}{2}}(dD)^{\frac{v_2 - v_1}{2}}\sum_{\substack{\alpha_2 \in W_{\alpha_1}(v_1, v_2, m): \\ |\hat{f}_{\alpha_2}|n^{\frac{v_2}{2}}(dD)^{\frac{v_1}{2}}\\ \le |\hat{f}_{\alpha_1}|n^{\frac{v_1}{2}}(dD)^{\frac{v_2}{2}}}} \left(\frac{e\sqrt{D}\lambda}{k^d}\right)^{m} \left(\frac{k}{n}\right)^{v(\alpha_1 \triangle \alpha_2)}\\
        &\le \left(2 + O\left(\frac{k^d}{D}\right)\right) \sum_{\substack{\alpha_1 \in \N^{[n]^d}:\\|\alpha_1| \le D}} \hat{f}_{\alpha_1}^2 \sum_{\substack{v_1, v_2 \ge 0,\\m \ge 1}} n^{\frac{v_1 - v_2}{2}}(dD)^{\frac{v_2 - v_1}{2}}n^{v_2} (dD)^{v_1} \binom{2dD}{d}^m \left(\frac{e\sqrt{D}\lambda}{k^d}\right)^{m} \left(\frac{k}{n}\right)^{v_1 + v_2}
        \intertext{where we used that the number of choices of $\alpha_2 \in W_{\alpha_1}(v_1, v_2, m)$ is at most $n^{v_2} (dD)^{v_1} \binom{2dD}{d}^m$, and that $v(\alpha_1 \triangle \alpha_2) \ge v_1 + v_2$.}
        &\le \left(2 + O\left(\frac{k^d}{D}\right)\right) \sum_{\substack{\alpha_1 \in \N^{[n]^d}:\\|\alpha_1| \le D}} \hat{f}_{\alpha_1}^2 \sum_{\substack{v_1, v_2 \ge 0,\\m \ge 1}}  \left(\frac{e\sqrt{D}\lambda}{k^d} \binom{2dD}{d}\right)^{m} \left(\frac{k\sqrt{dD}}{\sqrt{n}}\right)^{v_1 + v_2}\\
        &= \left(2 + O\left(\frac{k^d}{D}\right)\right) \sum_{\substack{\alpha_1 \in \N^{[n]^d}:\\|\alpha_1| \le D}} \hat{f}_{\alpha_1}^2 \sum_{m \ge 1}  \left(\frac{e\sqrt{D}\lambda}{k^d} \binom{2dD}{d}\right)^{m} \sum_{t \ge 0}(t+1)\left(\frac{k\sqrt{dD}}{\sqrt{n}}\right)^{t}\\
        &= \left(2 + O\left(\frac{k^d}{D}\right)\right) \sum_{\substack{\alpha_1 \in \N^{[n]^d}:\\|\alpha_1| \le D}} \hat{f}_{\alpha_1}^2   \cdot O\left(\frac{D^{d+ \frac{1}{2}} \lambda}{k^d}\right) \left(1 + O\left(\frac{D^{d+ \frac{1}{2}} \lambda}{k^d}\right)+O\left(\frac{k\sqrt{dD}}{\sqrt{n}}\right)\right)\\
        &= o\left(\sum_{\substack{\alpha \in \N^{[n]^d}:\\|\alpha| \le D}} \hat{f}_{\alpha}^2\right), \stepcounter{equation}\tag{\theequation}\label{ineq:sparse-PCA-off-diag-second-moment}
    \end{align*}
    provided that $D = o\left(\min\left\{\left(\frac{k^d}{\lambda}\right)^{\frac{1}{d+\frac{1}{2}}}, \frac{n}{k^2} \right\}\right)$.

    Combining the contribution of the diagonal terms and off-diagonal terms, we get
    \begin{align*}
        &\quad \EE[f(Y)^2]\\
        &=\sum_{\substack{\alpha \in \N^{[n]^d}: |\alpha| \le D}} \hat{f}_{\alpha}^2 \EE\left[h_{\alpha}(Y)^2\right] + \sum_{\substack{\alpha_1, \alpha_2 \in \N^{[n]^d}:\\ |\alpha_1|, |\alpha_2| \le D,\\ \alpha_1 \ne \alpha_2}} \hat{f}_{\alpha_1}\hat{f}_{\alpha_2} \EE\left[h_{\alpha_1}(Y)h_{\alpha_2}(Y)\right]\\
        &\ge \sum_{\substack{\alpha \in \N^{[n]^d}:\\|\alpha| \le D}} \hat{f}_{\alpha}^2\EE\left[h_{\alpha}(Y)^2\right] - o\left(\sum_{\substack{\alpha \in \N^{[n]^d}:\\|\alpha| \le D}} \hat{f}_{\alpha}^2\right)\\
        &= (1 - o(1))\sum_{\substack{\alpha \in \N^{[n]^d}:\\|\alpha| \le D}} \hat{f}_{\alpha}^2 \EE\left[h_{\alpha}(Y)^2\right],
    \end{align*}
    where we used that $\EE\left[h_{\alpha}(Y)^2\right] \ge 1$ for any $\alpha \in \N^{[n]^d}$.

    Next, we upper bound $\EE[f(T_{\rho}(Y))^2]$. Note that conditioned on $x$, each entry of $T_{\rho}(Y)$ is independently distributed as $T_{\rho}(Y)_{\overline{i}} \sim N(\sqrt{1-\rho^2}X_{\overline{i}}, 1)$, where $X = \sqrt{\lambda} x^{\otimes d}$. Thus, we can expand $\EE[f(T_{\rho}(Y))^2]$ and analyze the diagonal terms and off-diagonal terms similarly using the diagram formula.
    
    We have
    \begin{align*}
        &\quad \EE[f(T_{\rho}(Y))^2]\\
        &= \sum_{\substack{\alpha \in \N^{[n]^d}: |\alpha| \le D}} \hat{f}_{\alpha}^2 \EE\left[h_{\alpha}(T_{\rho}(Y))^2\right] + \sum_{\substack{\alpha_1, \alpha_2 \in \N^{[n]^d}:\\ |\alpha_1|, |\alpha_2| \le D,\\ \alpha_1 \ne \alpha_2}} \hat{f}_{\alpha_1}\hat{f}_{\alpha_2} \EE\left[h_{\alpha_1}(T_{\rho}(Y))h_{\alpha_2}(T_{\rho}(Y))\right]. \stepcounter{equation}\tag{\theequation}\label{eq:sparse-PCA-noisy-second-moment}
    \end{align*}
    For $\alpha_1, \alpha_2 \in \N^{[n]^d}$, the diagram formula says
    \begin{align*}
        &\quad \EE[h_{\alpha_1}(T_{\rho}(Y))h_{\alpha_2}(T_{\rho}(Y))]\\
        &= \EE_x \EE[h_{\alpha_1}(T_{\rho}(Y))h_{\alpha_2}(T_{\rho}(Y)) \big\vert x]\\
        &= \EE_x\left[\frac{1}{\sqrt{\alpha_1!\alpha_2!}} \sum_{M \in \MM(\alpha_1, \alpha_2)} \prod_{v \not\in V(M)} \mu_{j(v)} \prod_{\{a,b\} \in M} R_{j(a)j(b)} \Bigg\vert x\right]
        \intertext{Here, conditioned on $x$, the covariance structure of $T_{\rho}(Y)$ is unchanged compared to $Y$. The mean of $T_{\rho}(Y)$ is scaled by $\sqrt{1-\rho^2}$ compared to $Y$. Hence, using the same argument as before, we get}
        &= \EE_x \left[\frac{1}{\sqrt{\alpha_1!\alpha_2!}}\prod_{\overline{i} \in [n]^d} \sum_{M \in \MM(\alpha_1(\overline{i}), \alpha_2(\overline{i}))} \prod_{v \not\in V(M)} \sqrt{1-\rho^2}X_{\overline{i}} \Bigg \vert x\right]\\
        &= \EE_x\left[\frac{1}{\sqrt{\alpha_1!\alpha_2!}}\prod_{\overline{i} \in [n]^d} \sum_{t = 0}^{\alpha_1(\overline{i}) \wedge \alpha_2(\overline{i})} \binom{\alpha_1(\overline{i})}{t} \binom{\alpha_2(\overline{i})}{t}t! \cdot \left(\sqrt{1-\rho^2} X_{\overline{i}}\right)^{\alpha_1(\overline{i}) + \alpha_2(\overline{i}) - 2t} \Bigg \vert x\right],
    \end{align*}
    where $X = \sqrt{\lambda} x^{\otimes d}$. We note that every summand of the above conditional expectation is always nonnegative. Therefore, since the each inner summand is scaled by a nonnegative power of $\sqrt{1-\rho^2}$ compared to the formula of $\EE[h_{\alpha_1}(Y)h_{\alpha_2}(Y)]$, we have
    \begin{align*}
        \EE[h_{\alpha_1}(T_{\rho}(Y))h_{\alpha_2}(T_{\rho}(Y))] \le \EE[h_{\alpha_1}(Y)h_{\alpha_2}(Y)].
    \end{align*}
    Applying this inequality in \eqref{eq:sparse-PCA-noisy-second-moment}, we get
    \begin{align*}
        &\quad \EE[f(T_{\rho}(Y))^2]\\
        &= \sum_{\substack{\alpha \in \N^{[n]^d}: |\alpha| \le D}} \hat{f}_{\alpha}^2 \EE\left[h_{\alpha}(T_{\rho}(Y))^2\right] + \sum_{\substack{\alpha_1, \alpha_2 \in \N^{[n]^d}:\\ |\alpha_1|, |\alpha_2| \le D,\\ \alpha_1 \ne \alpha_2}} \hat{f}_{\alpha_1}\hat{f}_{\alpha_2} \EE\left[h_{\alpha_1}(T_{\rho}(Y))h_{\alpha_2}(T_{\rho}(Y))\right]\\
        &\le \sum_{\substack{\alpha \in \N^{[n]^d}: |\alpha| \le D}} \hat{f}_{\alpha}^2 \EE\left[h_{\alpha}(Y)^2\right] + \sum_{\substack{\alpha_1, \alpha_2 \in \N^{[n]^d}:\\ |\alpha_1|, |\alpha_2| \le D,\\ \alpha_1 \ne \alpha_2}} |\hat{f}_{\alpha_1}||\hat{f}_{\alpha_2}| \EE\left[h_{\alpha_1}(Y)h_{\alpha_2}(Y)\right]\\
        &\le (1+o(1))\sum_{\substack{\alpha \in \N^{[n]^d}: |\alpha| \le D}} \hat{f}_{\alpha}^2 \EE\left[h_{\alpha}(Y)^2\right]
        \intertext{where we used the bound \eqref{ineq:sparse-PCA-off-diag-second-moment}.}
        &\le (1+o(1))\EE[f(Y)^2].\stepcounter{equation}\tag{\theequation}\label{ineq:sparse-PCA-noisy-second-moment}
    \end{align*}
    Finally, we lower bound $\EE[f(Y)f(T_{\rho}(Y))]$.
    \begin{align*}
        &\quad \EE[f(Y)f(T_{\rho}(Y))]\\
        &= \sum_{\substack{\alpha \in \N^{[n]^d}: |\alpha| \le D}} \hat{f}_{\alpha}^2 \EE\left[h_{\alpha}(Y)h_{\alpha}(T_{\rho}(Y))\right] + \sum_{\substack{\alpha_1, \alpha_2 \in \N^{[n]^d}:\\ |\alpha_1|, |\alpha_2| \le D,\\ \alpha_1 \ne \alpha_2}} \hat{f}_{\alpha_1}\hat{f}_{\alpha_2} \EE\left[h_{\alpha_1}(Y)h_{\alpha_2}(T_{\rho}(Y))\right]. \stepcounter{equation}\tag{\theequation}\label{eq:sparse-PCA-correlated-second-moment}
    \end{align*}

    For $\alpha_1, \alpha_2 \in \N^{[n]^d}$, the diagram formula gives
    \begin{align*}
        &\quad \EE[h_{\alpha_1}(Y)h_{\alpha_2}(T_{\rho}(Y))]\\
        &= \EE_x \EE[h_{\alpha_1}(Y)h_{\alpha_2}(T_{\rho}(Y)) \big\vert x]\\
        &= \EE_x\left[\frac{1}{\sqrt{\alpha_1!\alpha_2!}} \sum_{M \in \MM(\alpha_1, \alpha_2)} \prod_{v \not\in V(M)} \mu_{j(v)} \prod_{\{a,b\} \in M} R_{j(a)j(b)} \Bigg\vert x\right]
        \intertext{Conditioned on $x$, $Y_{\overline{i}}$ and $T_{\rho}(Y)_{\overline{i}}$ has covariance $\sqrt{1-\rho^2}$, and $Y_{\overline{i}}$ and $T_{\rho}(Y)_{\overline{j}}$ are uncorrelated for $\overline{i} \ne \overline{j}$. The mean of $T_{\rho}(Y)$ is scaled by $\sqrt{1-\rho^2}$ compared to $Y$. Hence, the inner conditional expectation again tensorizes into a product over each coordinate:}
        &= \EE_x \left[\frac{1}{\sqrt{\alpha_1!\alpha_2!}}\prod_{\overline{i} \in [n]^d} \sum_{M \in \MM(\alpha_1(\overline{i}), \alpha_2(\overline{i}))} \prod_{v \not\in V(M)} \mu_{j(v)} \prod_{\{a,b\}\in M }\sqrt{1-\rho^2} \Bigg \vert x\right]\\
        &= \EE_x\left[\frac{1}{\sqrt{\alpha_1!\alpha_2!}}\prod_{\overline{i} \in [n]^d} \sum_{t = 0}^{\alpha_1(\overline{i}) \wedge \alpha_2(\overline{i})} \binom{\alpha_1(\overline{i})}{t} \binom{\alpha_2(\overline{i})}{t}t! \cdot \left(\sqrt{1-\rho^2}\right)^{\alpha_2(\overline{i})}X_{\overline{i}}^{\alpha_1(\overline{i}) + \alpha_2(\overline{i}) - 2t} \Bigg \vert x\right],
    \end{align*}
    where $X = \sqrt{\lambda} x^{\otimes d}$, and we used that $\sqrt{1-\rho^2}$ is raised to an exponent equal to the number of edges in the matching $M \in \MM(\alpha_1(\overline{i}), \alpha_2(\overline{i}))$ plus the number of unmatched vertices among the $\alpha_2(\overline{i})$ vertices corresponding to $h_{\alpha_2(\overline{i})}$, and this exponent is exactly $\alpha_2(\overline{i})$. We note that every summand of the above conditional expectation is always nonnegative, and is scaled by a factor $\prod_{\overline{i}\in [n]^d}\left(\sqrt{1-\rho^2}\right)^{\alpha_2(i)} = \left(\sqrt{1-\rho^2}\right)^{|\alpha_2|}$ compared to the that of $\EE[h_{\alpha_1}(Y)h_{\alpha_2}(Y)]$, which is between $\left(\sqrt{1-\rho^2}\right)^{D}$ and $1$. Thus, we have
    \begin{align}
        \left(\sqrt{1-\rho^2}\right)^{D} \le \frac{\EE\left[h_{\alpha_1}(Y)h_{\alpha_2}(T_{\rho}(Y))\right]}{\EE\left[h_{\alpha_1}(Y)h_{\alpha_2}(Y)\right]} \le 1. \label{ineq:correlated-expectation-bound}
    \end{align}
    Applying the lower bound of \eqref{ineq:correlated-expectation-bound} to the diagonal terms in \eqref{eq:sparse-PCA-correlated-second-moment} and the upper bound \eqref{ineq:correlated-expectation-bound} to the off-diagonal terms in \eqref{eq:sparse-PCA-correlated-second-moment}, we get
    \begin{align*}
        &\quad \EE[f(Y)f(T_{\rho}(Y))]\\
        &= \sum_{\substack{\alpha \in \N^{[n]^d}: |\alpha| \le D}} \hat{f}_{\alpha}^2 \EE\left[h_{\alpha}(Y)h_{\alpha}(T_{\rho}(Y))\right] + \sum_{\substack{\alpha_1, \alpha_2 \in \N^{[n]^d}:\\ |\alpha_1|, |\alpha_2| \le D,\\ \alpha_1 \ne \alpha_2}} \hat{f}_{\alpha_1}\hat{f}_{\alpha_2} \EE\left[h_{\alpha_1}(Y)h_{\alpha_2}(T_{\rho}(Y))\right]\\
        &\ge \sum_{\substack{\alpha \in \N^{[n]^d}: |\alpha| \le D}} \hat{f}_{\alpha}^2 \EE\left[h_{\alpha}(Y)h_{\alpha}(T_{\rho}(Y))\right] - \sum_{\substack{\alpha_1, \alpha_2 \in \N^{[n]^d}:\\ |\alpha_1|, |\alpha_2| \le D,\\ \alpha_1 \ne \alpha_2}} |\hat{f}_{\alpha_1}||\hat{f}_{\alpha_2}|\EE\left[h_{\alpha_1}(Y)h_{\alpha_2}(T_{\rho}(Y))\right]\\
        &\ge \left(\sqrt{1-\rho^2}\right)^D\sum_{\substack{\alpha \in \N^{[n]^d}: |\alpha| \le D}} \hat{f}_{\alpha}^2 \EE\left[h_{\alpha}(Y)^2\right] - \sum_{\substack{\alpha_1, \alpha_2 \in \N^{[n]^d}:\\ |\alpha_1|, |\alpha_2| \le D,\\ \alpha_1 \ne \alpha_2}} |\hat{f}_{\alpha_1}||\hat{f}_{\alpha_2}|\EE\left[h_{\alpha_1}(Y)h_{\alpha_2}(Y)\right]\\
        &\ge \left(\sqrt{1-\rho^2}\right)^D\sum_{\substack{\alpha \in \N^{[n]^d}: |\alpha| \le D}} \hat{f}_{\alpha}^2 \EE\left[h_{\alpha}(Y)^2\right] - o\left(\sum_{\substack{\alpha \in \N^{[n]^d}: |\alpha| \le D}} \hat{f}_{\alpha}^2 \right)
        \intertext{where we used the bound \eqref{ineq:sparse-PCA-off-diag-second-moment}}
        &\ge \left(\left(\sqrt{1-\rho^2}\right)^D - o(1)\right)\EE\left[f(Y)^2\right].\stepcounter{equation}\tag{\theequation}\label{ineq:sparse-PCA-correlated-second-moment}
    \end{align*}
    Combining \eqref{ineq:sparse-PCA-noisy-second-moment} and \eqref{ineq:sparse-PCA-correlated-second-moment}, we get
    \begin{align*}
        &\quad \EE\left[\left(f(Y)-f(T_{\rho}(Y))\right)^2\right]\\
        &= \EE[f(Y)^2] + \EE[f(T_{\rho}(Y))^2] - 2\EE\left[f(Y)f(T_{\rho}(Y))\right]\\
        &\le 2\left(1 - \left(\sqrt{1-\rho^2}\right)^D + o(1)\right)\EE[f(Y)^2].
    \end{align*}
    Thus, degree-$D$ polynomials are $\left(\rho, 2\left(1 - \left(\sqrt{1-\rho^2}\right)^D + o(1)\right)\right)$-stable for the $k$-sparse PCA provided that $D = o\left(\min\left\{\left(\frac{k^d}{\lambda}\right)^{\frac{1}{d+\frac{1}{2}}}, \frac{n}{k^2} \right\}\right)$.
    
\end{proof}

\subsubsection{Deferred Proof for Critical Window/MMSE Instability of Sparse Tensor PCA} \label{sec:critical-window-PCA}

    Consider the $k$-Sparse Tensor PCA model at SNR $\lambda$. Its MMSE is
    \begin{align*}
        &\quad \EE\left[\|\EE[x \vert Y] - x\|^2\right]\\
        &= \EE[\|x\|^2] - \EE\left[\langle x, \EE[x\vert Y]\rangle\right].
    \end{align*}
    Let $(x,Y) \sim \PP$ denote the distribution of $k$-sparse tensor PCA, and $x \sim \PP(\cdot \vert Y)$ denote the posterior distribution. Then, we have
    \begin{align*}
        &\quad \EE\left[\|\EE[x \vert Y] - x\|^2\right]\\
        &= \EE[\|x\|^2] - \EE_{(x,Y) \sim \PP} \EE_{x' \sim \PP(\cdot \vert Y)}\left[\langle x, x'\rangle\right].
    \end{align*}
    Therefore, it suffices to understand the expected overlap between the prior $x$ and the posterior $x'$.

    Let $p_i \coloneqq \Pr(\langle x,x'\rangle = i/k)$ for $0 \le i \le k$. Then, we have
    \begin{align*}
        &\quad \EE\left[\|\EE[x \vert Y] - x\|^2\right]\\
        &= \EE[\|x\|^2] - \EE_{(x,Y) \sim \PP} \EE_{x' \sim \PP(\cdot \vert Y)}\left[\langle x, x'\rangle\right]\\
        &= 1 - \EE\left[\sum_{i=0}^k p_i\cdot \frac{i}{k}\right].
    \end{align*}

    Fix $Y = \sqrt{\lambda} x^{\otimes d} + W$. Let $\Omega_{n,k}$ denote set of $k$-sparse vectors in $\{0, 1/\sqrt{k}\}^n$. The posterior distribution satisfies that for $x' \in \Omega_{n,k}$,
    \begin{align*}
        \PP(x' \vert Y) &\propto \PP(Y \vert x') \PP(x')\\
        &\propto \PP(Y \vert x')\\
        &= \exp\left(- \|Y - \sqrt{\lambda} x'^{\otimes d}\|^2/2\right)\\
        &\propto \exp\left( \sqrt{\lambda}\langle Y, x'^{\otimes d}\rangle - \lambda \|x'^{\otimes d}\|^2/2\right)\\
        &= \exp\left( \sqrt{\lambda}\langle \sqrt{\lambda} x^{\otimes d} + W, x'^{\otimes d}\rangle - \lambda/2\right)\\
        &\propto \exp\left(\lambda \langle x, x' \rangle^d + \sqrt{\lambda} \langle W, x'^{\otimes d} \rangle\right). 
    \end{align*}
    Therefore, if we define 
    \begin{align*}
        S_i &\coloneqq \sum_{\substack{x'\in \Omega_{n,k}:\\ \langle x, x' \rangle = i/k}} \exp\left(\lambda \langle x, x' \rangle^d + \sqrt{\lambda} \langle W, x'^{\otimes d} \rangle\right)\\
        &= \sum_{\substack{x'\in \Omega_{n,k}:\\ \langle x, x' \rangle = i/k}} \exp\left(\lambda \left(\frac{i}{k}\right)^d + \sqrt{\lambda} \langle W, x'^{\otimes d} \rangle\right),
    \end{align*}
    then we have $p_i = \frac{S_i}{\sum_{j=0}^k S_j}$. Before we start analyzing the critical window of sparse tensor PCA, let us state a few useful lemmas.

    \begin{lemma}\label{lem:class-size-bound}
        For $0 \le i \le k$, let $N_i \coloneqq \binom{k}{i}\binom{n-k}{k-i}$. Then,
        \begin{align*}
            \frac{N_i}{N_0} \le \frac{1}{i!}\left(\frac{k^2}{n-2k}\right)^i.
        \end{align*}
    \end{lemma}
    \begin{proof}[Proof of Lemma~\ref{lem:class-size-bound}]
        \begin{align*}
            \frac{N_i}{N_0} &= \binom{k}{i} \frac{\binom{n-k}{k-i}}{\binom{n-k}{k}}\\
            &\le \frac{k^i}{i!}\cdot\frac{k^i}{(n-2k)^i}\\
            &= \frac{1}{i!}\left(\frac{k^2}{n-2k}\right)^i.
        \end{align*}
    \end{proof}

    \begin{lemma}\label{lem:log-sum-exp-tail}
        Let $Z_1, \dots, Z_m$ be Gaussians, each distributed as $N(0, \sigma^2)$. Then, for any $t \ge 0$ such that \[\sqrt{2\sigma^2 (t+\log m)} \le \sigma^2,\]
        we have
        \begin{align*}
            \Pr\left(\log\left(\sum_{i=1}^m \exp(Z_i)\right) \le \sqrt{2\sigma^2 (t+ \log m)} +\log\left(1 +\sqrt{2\sigma^2 (t+ \log m)}\right)\right) \ge 1 - 2\exp(-t).
        \end{align*}
    \end{lemma}

    \begin{proof}[Proof of Lemma~\ref{lem:log-sum-exp-tail}]
        By union bound, we have
        \begin{align*}
            \Pr(\max_{i} Z_i \ge s) &\le m \Pr(Z_1 \ge s)\\
            &\le m \exp\left(-\frac{s^2}{2\sigma^2}\right).
        \end{align*}
        Thus,
        \begin{align*}
            \Pr(\max_{i} Z_i \ge \sqrt{2\sigma^2 (t+\log m)}) &\le m \Pr(Z_1 \ge \sqrt{2\sigma^2 (t+\log m)})\\
            &\le m \exp\left(-\frac{2\sigma^2 (t+\log m)}{2\sigma^2}\right)\\
            &= \exp(-t).
        \end{align*}

        Denote $E = \{\max_{i} Z_i \le \sqrt{2\sigma^2 (t+\log m)}\}$, and this event takes place with probability at least $1 - \exp(-t)$. We have
        \begin{align*}
            \EE\left[\sum_{i=1}^m \left(\exp(Z_i)-1\right) \cdot \allone_E\right] &= \sum_{i=1}^m \EE\left[\left(\exp(Z_i)-1\right) \cdot \allone_E\right]\\
            &\le \sum_{i=1}^m \EE\left[\left(\exp(Z_i)-1\right) \cdot \allone\{Z_i \le \sqrt{2\sigma^2 (t+\log m)}\}\right]\\
            &\le \sum_{i=1}^m \int_{0}^{\sqrt{2\sigma^2 (t+\log m)}} \exp(s)\Pr\left(Z_i \ge s\right)ds\\
            &\le \sum_{i=1}^m \int_{0}^{\sqrt{2\sigma^2 (t+\log m)}} \exp(s)\exp\left(-\frac{s^2}{2\sigma^2}\right)ds\\
            &\le m \sqrt{2\sigma^2 (t+\log m)} \exp\left(\sqrt{2\sigma^2 (t+\log m)}\right) \exp\left(-\left(t+\log m\right)\right)
            \intertext{where we used that $\sqrt{2\sigma^2 (t+\log m)} \le \sigma^2$, so that $\exp(s)\exp\left(-\frac{s^2}{2\sigma^2}\right)$ is increasing for $s \in [0, \sqrt{2\sigma^2 (t+\log m)}]$.}
            &\le \sqrt{2\sigma^2 (t+\log m)} \exp\left(\sqrt{2\sigma^2 (t+\log m)} - t\right).
        \end{align*}
        By Markov's inequality,
        \begin{align*}
            &\quad \Pr\left(\sum_{i=1}^m \left(\exp(Z_i)-1\right) \cdot \allone_E \ge \sqrt{2\sigma^2 (t+\log m)} \exp\left(\sqrt{2\sigma^2 (t+\log m)}\right)\right)\\
            &\le \frac{\EE\left[\sum_{i=1}^m \left(\exp(Z_i)-1\right) \cdot \allone_E\right]}{\sqrt{2\sigma^2 (t+\log m)} \exp\left(\sqrt{2\sigma^2 (t+\log m)}\right)}\\
            &\le \exp\left(-t\right).
        \end{align*}
        Note that since $\sqrt{2\sigma^2 (t+\log m)}\le \sigma^2$, we have
        \begin{align*}
            t + \log m &= \frac{\left(\sqrt{2\sigma^2 (t+ \log m)}\right)^2}{2\sigma^2}\\
            &\le \frac{1}{2}\sqrt{2\sigma^2 (t+ \log m)}.\\
            m &\le \exp\left(\sqrt{2\sigma^2 (t+ \log m)}\right)
        \end{align*}
        Thus, condition on the event $E$ and that $\sum_{i=1}^m \left(\exp(Z_i)-1\right) \cdot \allone_E \le \sqrt{2\sigma^2 (t+\log m)} \exp\left(\sqrt{2\sigma^2 (t+\log m)}\right)$, which takes place with probability at least $1 - 2\exp\left(-t\right)$, we have
        \begin{align*}
            \log\left(\sum_{i=1}^m \exp(Z_i)\right) &\le \log \left(m + \sqrt{2\sigma^2 (t+\log m)} \exp\left(\sqrt{2\sigma^2 (t+\log m)}\right)\right)\\
            &\le \log \left(\exp\left(\sqrt{2\sigma^2 (t+ \log m)}\right) + \sqrt{2\sigma^2 (t+\log m)} \exp\left(\sqrt{2\sigma^2 (t+\log m)}\right)\right)\\
            &= \sqrt{2\sigma^2 (t+ \log m)} +\log\left(1 +\sqrt{2\sigma^2 (t+ \log m)}\right),
        \end{align*}
        and this finishes the proof.
    \end{proof}

    \begin{lemma}\label{lem:correlated-gaussian-bound}
        Let $X,Y$ be centered Gaussain variables with variance $1$ and covariance $\rho$. Then, there exists an absolute constant $C > 0$ such that for all $u \ge 0$,
        \begin{align*}
            \Pr(X \ge u, Y \ge u) \le \frac{C}{\sqrt{1 - \rho^2}} \exp\left(-\frac{u^2}{1 + \rho}\right).
        \end{align*}
    \end{lemma}

    \begin{proof}[Proof of Lemma~\ref{lem:correlated-gaussian-bound}]
        The density of $(X,Y) = (x,y)$ is 
        \[\frac{1}{2\pi\sqrt{1-\rho^2}} \exp\left(-\frac{x^2 - 2\rho xy + y^2}{2(1-\rho^2)}\right).\]
        Integrating along $x \ge u$ and $y \ge u$, we get
        \begin{align*}
            &\quad \int_{u}^{\infty} \int_{u}^{\infty} \frac{1}{2\pi\sqrt{1-\rho^2}} \exp\left(-\frac{x^2 - 2\rho xy + y^2}{2(1-\rho^2)}\right) dy dx\\
            &= \int_{0}^{\infty} \int_{0}^{\infty} \frac{1}{2\pi\sqrt{1-\rho^2}} \exp\left(-\frac{(u+t)^2 - 2\rho (u+t)(u+s) + (u+s)^2}{2(1-\rho^2)}\right) ds dt\\
            &= \int_{0}^{\infty} \int_{0}^{\infty} \frac{1}{2\pi\sqrt{1-\rho^2}} \exp\left(-\frac{2u^2(1-\rho) + 2u(1-\rho)(s+t) + s^2 + t^2 - 2\rho st}{2(1-\rho^2)}\right) ds dt \\
            &\le \int_{0}^{\infty} \int_{0}^{\infty} \frac{1}{2\pi\sqrt{1-\rho^2}} \exp\left(-\frac{2u^2(1-\rho) + 2u(1-\rho)(s+t)}{2(1-\rho^2)}\right) ds dt \\
            &= \frac{1}{2\pi\sqrt{1-\rho^2}} \exp\left(-\frac{u^2}{1 + \rho}\right)\int_{0}^{\infty} \int_{0}^{\infty} \exp\left(-\frac{u(s+t)}{1+\rho}\right) ds dt\\
            &= \frac{(1+\rho)^2}{2\pi u^2\sqrt{1-\rho^2}} \exp\left(-\frac{u^2}{1 + \rho}\right).
        \end{align*}
    \end{proof}

    First, we prove the upper edge of the critical window of the Sparse Tensor PCA.

\begin{proof}[Proof of the upper edge of Proposition~\ref{prop:critical-window-PCA}]
    Suppose $d \ge 3$, $k \le \min\left\{n^{\frac{1}{2} - \eps}, n^{\frac{d-2}{d+2} - \eps}\right\}$ for some constant $\eps > 0$, and $\Delta > 0$ satisfies
    \begin{align*}
        \Delta = \omega\left(\sqrt{\log \binom{n-k}{k}}\right).
    \end{align*}
    Suppose $\lambda = 2\log \binom{n-k}{k} + \Delta$.

    Fix an arbitrary $\delta > 0$. We will show that $\sum_{i=0}^{(1-\delta)k}p_i \to 0$. For convenience, let $N_i \coloneqq \binom{k}{i}\binom{n-k}{k-i}$. Note that 
    \[N_i = \left|\{x'\in \Omega_{n,k}: \langle x, x' \rangle = i/k\}\right|.\]
    We will break the analysis into two parts, one for small $i$ and one for large $i$. Since $k \le n^{\frac{1}{2} - \varepsilon}$ , we know $\log(k^2/(n-2k)) \le - c_{\varepsilon} \log(n/k)$ for some constant $c_{\varepsilon} > 0$. Now choose $\eta > 0$ so that $3 \eta^{d-1} < \frac{c_{\varepsilon}}{2}$.

    \paragraph{\textbf{Small overlap analysis:}} First, consider $i \le \eta k$.

    Recall that
    \begin{align*}
        S_i &= \sum_{\substack{x'\in \Omega_{n,k}:\\ \langle x, x' \rangle = i/k}} \exp\left(\lambda \left(\frac{i}{k}\right)^d + \sqrt{\lambda} \langle W, x'^{\otimes d} \rangle\right)\\
        &= \exp\left(\lambda \left(\frac{i}{k}\right)^d + \left(\frac{i}{k}\right)^d\sqrt{\lambda} \langle W, x^{\otimes d} \rangle\right)\sum_{\substack{x'\in \Omega_{n,k}:\\ \langle x, x' \rangle = i/k}} \exp\left( \sqrt{\lambda} \left\langle W, x'^{\otimes d} - \left(i/k\right)^d x^{\otimes d}  \right\rangle\right),
    \end{align*}
    Note that $\langle W, x^{\otimes d} \rangle$ and $\left\langle W, x'^{\otimes d} - \left(i/k\right)^d x^{\otimes d}  \right\rangle$ are independent, and $\left\langle W, x'^{\otimes d} - \left(i/k\right)^d x^{\otimes d}  \right\rangle$ is distributed as $N(0, 1 - (i/k)^{2d})$. Thus, we have
    \begin{align*}
        &\quad \EE\left[\frac{S_i}{S_k} \bigg\vert \langle W, x^{\otimes d} \rangle\right]\\
        &= \exp\left(\lambda \left(\left(\frac{i}{k}\right)^d - 1\right) + \left(\left(\frac{i}{k}\right)^d - 1\right)\sqrt{\lambda} \langle W, x^{\otimes d} \rangle\right)\\
        &\quad \cdot\EE\left[\sum_{\substack{x'\in \Omega_{n,k}:\\ \langle x, x' \rangle = i/k}} \exp\left( \sqrt{\lambda} \left\langle W, x'^{\otimes d} - \left(i/k\right)^d x^{\otimes d}  \right\rangle\right)\right]\\
        &= \exp\left(\lambda \left(\left(\frac{i}{k}\right)^d - 1\right) + \left(\left(\frac{i}{k}\right)^d - 1\right)\sqrt{\lambda} \langle W, x^{\otimes d} \rangle\right) N_i \exp\left(\frac{\lambda}{2}\left(1 - \left(\frac{i}{k}\right)^{2d}\right)\right)\\
        &= N_i \exp\left(\lambda \left(\left(\frac{i}{k}\right)^d - 1\right) + \left(\left(\frac{i}{k}\right)^d - 1\right)\sqrt{\lambda} \langle W, x^{\otimes d} \rangle + \frac{\lambda}{2}\left(1 - \left(\frac{i}{k}\right)^{2d}\right)\right).
    \end{align*}

    Since $\Delta = \omega\left(\sqrt{\log\binom{n-k}{k}}\right)$, we may choose $b = \sqrt{\frac{\Delta }{\sqrt{\log\binom{n-k}{k}}}}$ so that $b = \omega(1)$ and $b\sqrt{\log\binom{n-k}{k}} = o(\Delta)$. Consider the event $E_1 = \{\langle W, x^{\otimes d} \rangle \ge -b\}$. We have $\Pr(E_1) \ge 1 - \exp\left(-b^2/2\right) = 1 - o(1)$. Thus, condition on the event $E_1$, we have
    \begin{align*}
        &\quad \EE\left[\frac{S_i}{S_k} \bigg\vert E_1\right]\\
        &\le N_i \exp\left(\lambda \left(\left(\frac{i}{k}\right)^d - 1\right) -b \left(\left(\frac{i}{k}\right)^d - 1\right)\sqrt{\lambda}  + \frac{\lambda}{2}\left(1 - \left(\frac{i}{k}\right)^{2d}\right)\right)\\
        &\le \exp\left(\log\binom{n-k}{k} +  i\log\left(\frac{k^2}{n-2k}\right)  - \frac{\lambda}{2} + b\sqrt{\lambda} + \lambda\left(\frac{i}{k}\right)^d\right)
        \intertext{since $N_i \le \binom{n-k}{k} \frac{1}{i!}\left(\frac{k^2}{n-2k}\right)^i$ by Lemma~\ref{lem:class-size-bound},}
        &= \exp\left( i\log\left(\frac{k^2}{n-2k}\right)  - \frac{\Delta}{2} + b\sqrt{2 \log \binom{n-k}{k} + \Delta} + \left(2 \log \binom{n-k}{k} + \Delta\right)\left(\frac{i}{k}\right)^d\right)
        \intertext{by plugging in $\lambda = 2 \log \binom{n-k}{k} + \Delta,$}
        &\le \exp\left( -c_{\varepsilon} i\log\left( n/k\right)  - \frac{\Delta}{2} + o(\Delta) + 3i \log(n/k)\left(\frac{i}{k}\right)^{d-1}\right)
        \intertext{by using $\log(k^2/(n-2k)) \le - c_{\varepsilon} \log(n/k)$, $\Delta \le \log\binom{n-k}{k}$, and $b \sqrt{\log \binom{n-k}{k}} = o(\Delta)$}
        &\le \exp\left(3i \log(n/k)\eta^{d-1} -c_{\varepsilon} i\log\left( n/k\right)  - \frac{\Delta}{2} + o(\Delta)\right)
        \intertext{since $i \le \eta k$. Finally we use that $3\eta^{d-1} \le c_{\varepsilon}$ and get}
        &\le \exp\left(-\frac{\Delta}{4}\right).
    \end{align*}
    Summing over $i \le \eta k$, we get
    \begin{align*}
        \EE\left[\sum_{i=0}^{\eta k} \frac{S_i}{S_k} \bigg\vert E\right] &\le \eta k \exp\left(-\frac{\Delta}{4}\right)\\
        &= o(1)
    \end{align*}
    since $\Delta = \omega\left(\sqrt{\log \binom{n-k}{k}}\right)$.
    Thus, condition on the event $E_1$, by Markov's inequality, we have
    \begin{align*}
        &\quad \sum_{i=0}^{\eta k} p_i\\
        &\le \sum_{i=0}^{\eta k} \frac{S_i}{S_k}\\
        &\to 0.
    \end{align*}
    Since $E_1$ has probability $1 - o(1)$, we conclude that $\sum_{i=0}^{\eta k} p_i \to 0$ when $\lambda = 2\log \binom{n-k}{k} + \Delta$.

    \paragraph{\textbf{Large overlap analysis:}}
    Next, we consider $\eta k \le i \le (1-\delta)k$.

    Recall that
    \begin{align*}
        S_i &= \exp\left(\lambda \left(\frac{i}{k}\right)^d + \left(\frac{i}{k}\right)^d\sqrt{\lambda} \langle W, x^{\otimes d} \rangle\right)\sum_{\substack{x'\in \Omega_{n,k}:\\ \langle x, x' \rangle = i/k}} \exp\left( \sqrt{\lambda} \left\langle W, x'^{\otimes d} - \left(i/k\right)^d x^{\otimes d}  \right\rangle\right).
    \end{align*}
    If we denote
    \[Y_i = \sum_{\substack{x'\in \Omega_{n,k}:\\ \langle x, x' \rangle = i/k}} \exp\left( \sqrt{\lambda} \left\langle W, x'^{\otimes d} - \left(i/k\right)^d x^{\otimes d}  \right\rangle\right),\]
    we have \begin{align*}
        S_i &= \exp\left(\lambda \left(\frac{i}{k}\right)^d + \left(\frac{i}{k}\right)^d\sqrt{\lambda} \langle W, x^{\otimes d} \rangle\right)Y_i.
    \end{align*}
    
    Fix some $i \in [\eta k, (1-\delta)k]$. By Lemma~\ref{lem:log-sum-exp-tail}, with probability at least $1 - 2\exp(-t)$, we have
    \begin{align*}
        \Pr\left(\log(Y_i) \ge B_i + \log(1 + B_i)\right)  \le 2\exp(-t_i),
    \end{align*}
    where $B_i = \sqrt{2\lambda\left(1 - \left(i/k\right)^{2d}\right)\left(t_i + \log N_i\right)}$ for some $t_i \ge 0$, provided that
    \[B_i \le \lambda\left(1 - \left(i/k\right)^{2d}\right).\]
    Now we choose $t_i = 2\log k$. We may verify that for $i \le (1-\delta)k$,
    \begin{align*}
        B_i &= \sqrt{2\lambda\left(1 - \left(i/k\right)^{2d}\right)\left(2\log k + \log N_i\right)}\\
        &= \sqrt{2\lambda\left(1 - \left(i/k\right)^{2d}\right)\left(2\log k + \log \binom{k}{i} + \log \binom{n-k}{k-i}\right)}\\
        &\le \sqrt{2\lambda\left(1 - \left(i/k\right)^{2d}\right)\left(2\log k + k + (k-i)\log (e(n-k)/(k-i)) \right)}
        \intertext{Now we use that $k - i \ge \delta k$ and get}
        &\le \sqrt{2\lambda\left(1 - \left(i/k\right)^{2d}\right) \left((k-i)\log((n-k)/k) + O(k)\right)}\\
        &\le \sqrt{2\lambda\left(1 - \left(i/k\right)^{2d}\right) \left(\left(1 - \frac{i}{k}\right)\log \binom{n-k}{k} + O(k)\right)}\\
        &\le \lambda\left(1 - \left(i/k\right)^{2d}\right) \sqrt{\frac{\left(2\log \binom{n-k}{k}\right)\left(1- (i/k) + o(1)\right)}{\lambda\left(1 - \left(i/k\right)^{2d}\right)}}\\
        &\le \lambda\left(1 - \left(i/k\right)^{2d}\right),
    \end{align*}
    since $\lambda \ge 2\log \binom{n-k}{k}$ and $1 - (i/k) + o(1) \le 1 - (i/k)^{2d}$ for $\eta k \le i \le (1-\delta)k$. Consequently, for any $\eta k\le i \le (1-\delta)k$, with probability at least $ 1- 2\exp(-2\log k) = 1 - \frac{2}{k^2}$, we have
    \begin{align*}
        \log(Y_i) \le B_i + \log(1 + B_i),
    \end{align*}
    where $B_i = \sqrt{2\lambda\left(1 - \left(i/k\right)^{2d}\right)\left(2\log k + \log N_i\right)}$. Taking a union bound over all $i \le (1 - \delta)k$, we have with probability at least $1 - \frac{2}{k}$ that
    \begin{align*}
        &\quad \sum_{i=\eta k}^{(1-\delta) k} S_i\\
        &= \sum_{i=\eta k}^{(1-\delta) k} \exp\left(\lambda \left(\frac{i}{k}\right)^d + \left(\frac{i}{k}\right)^d\sqrt{\lambda} \langle W, x^{\otimes d} \rangle\right)Y_i\\
        &\le \sum_{i=\eta k}^{(1-\delta) k}\exp\left(\lambda \left(\frac{i}{k}\right)^d + \left(\frac{i}{k}\right)^d\sqrt{\lambda} \langle W, x^{\otimes d} \rangle\right)\exp\left(B_i + \log(1+B_i)\right)\\
        &= S_k \sum_{i=\eta k}^{(1-\delta) k}\frac{1}{S_k}\exp\left(\lambda \left(\frac{i}{k}\right)^d + \left(\frac{i}{k}\right)^d\sqrt{\lambda} \langle W, x^{\otimes d} \rangle\right)\exp\left(B_i + \log(1+B_i)\right)
        \intertext{Now we use $S_k = \exp(\lambda + \sqrt{\lambda} \langle W, x^{\otimes d} \rangle)$ and get}
        &= S_k \sum_{i=\eta k}^{(1-\delta) k}\exp\left(\left(\left(\frac{i}{k}\right)^d-1\right)\left[\lambda  + \sqrt{\lambda} \langle W, x^{\otimes d} \rangle\right]\right)\exp\left(B_i + \log(1+B_i)\right)\\
        &= S_k \sum_{i=\eta k}^{(1-\delta) k}\exp\left(\left(\left(\frac{i}{k}\right)^d-1\right)\sqrt{\lambda} \langle W, x^{\otimes d} \rangle\right)\exp\left(\left(\left(\frac{i}{k}\right)^d-1\right)\lambda+ B_i + \log(1+B_i)\right). \stepcounter{equation}\tag{\theequation}\label{ineq:upper-branch-posterior}
    \end{align*}
    Call the above event $E_2$, which takes place with probability at least $1 - \frac{2}{k}$.

    Again, let us condition on the event $E_1 = \{\langle W, x^{\otimes d} \rangle \ge -b\}$, which takes place with probability $1 - o(1)$. Recall here $b = \sqrt{\frac{\Delta }{\sqrt{\log\binom{n-k}{k}}}}$ so that $b = \omega(1)$ and $b\sqrt{\log\binom{n-k}{k}} = o(\Delta)$. Conditioning on $E_1$ and $E_2$, let us analyze the sum in \eqref{ineq:upper-branch-posterior}. We have
    \begin{align*}
        &\quad \sum_{i=\eta k}^{(1-\delta) k}\exp\left(\left(\left(\frac{i}{k}\right)^d-1\right)\sqrt{\lambda} \langle W, x^{\otimes d} \rangle\right)\exp\left(\left(\left(\frac{i}{k}\right)^d-1\right)\lambda+ B_i + \log(1+B_i)\right)\\
        &\le \sum_{i=\eta k}^{(1-\delta) k}\exp\left(-b\left(\left(\frac{i}{k}\right)^d-1\right)\sqrt{\lambda} + \left(\left(\frac{i}{k}\right)^d-1\right)\lambda+ B_i(1+o(1))\right)\\
        &\le \sum_{i=\eta k}^{(1-\delta) k}\exp\Bigg(-b\left(\left(\frac{i}{k}\right)^d-1\right)\sqrt{\lambda} + \left(\left(\frac{i}{k}\right)^d-1\right)\lambda\\
        &\quad + (1+o(1))\sqrt{2\lambda\left(1 - \left(\frac{i}{k}\right)^{2d}\right)\left(2\log k + \log N_i\right)}\Bigg)
        \intertext{Now we reuse the previous calculation that $2\log k + \log N_i \le (1 - (i/k))\log \binom{n-k}{k} + O(k)$ and the bound $b \sqrt{\log \binom{n-k}{k}} = o(\Delta)$ and $\lambda = 2\log \binom{n-k}{k} + \Delta \le 3\binom{n-k}{k}$ to get}
        &\le \sum_{i=\eta k}^{(1-\delta) k}\exp\left(o(\Delta) + \left(\left(\frac{i}{k}\right)^d-1\right)\lambda+ (1+o(1))\sqrt{2\lambda\left(1 - \left(\frac{i}{k}\right)^{2d}\right)\left(1 - \frac{i}{k} + o(1)\right)\log \binom{n-k}{k}}\right)
        \intertext{Now we plug in $\lambda = 2\log \binom{n-k}{k} + \Delta$ and get}
        &\le \sum_{i=\eta k}^{(1-\delta) k}\exp\left(o(\Delta) + 2\log \binom{n-k}{k}\left(\left(\frac{i}{k}\right)^d - 1 + (1 + o(1))\sqrt{\left(1 - \left(\frac{i}{k}\right)^{2d}\right)\left(1 - \frac{i}{k} + o(1)\right)}\right) \right)\\
        &\quad \cdot \exp\left(\Delta\left(\left(\frac{i}{k}\right)^d-1 + \left(\frac{1}{2} + o(1)\right)\sqrt{\left(1 - \left(\frac{i}{k}\right)^{2d}\right)\left(1 - \frac{i}{k} + o(1)\right)}\right)\right)
        \intertext{It is easy to verify that there exists some constant $\gamma_{\eta, \delta} > 0$ such that $f(x) = x^d - 1 + \sqrt{(1-x)^{2d}(1-x)} < -\gamma_{\eta, \delta}$ for all $\eta \le x \le 1 - \delta$. As a result,}
        &\le \sum_{i=\eta k}^{(1-\delta) k}\exp\left(o(\Delta) - 2(\gamma_{\eta, \delta} + o(1))\log \binom{n-k}{k} \right)\\
        &\le \exp\left(\log k + o(\Delta) - 2(\gamma_{\eta, \delta} + o(1))\log \binom{n-k}{k} \right)\\
        &= o(1).
    \end{align*}

    Plugging this bound back to \eqref{ineq:upper-branch-posterior} and conditioning on $E_1$ and $E_2$, we see that with probability at least $1 - o(1)$, we have
    \begin{align*}
        &\quad \sum_{i=\eta k}^{(1-\delta) k} p_i\\
        &\le \sum_{i=\eta k}^{(1-\delta) k} \frac{S_i}{S_k}\\
        &\le \sum_{i=\eta k}^{(1-\delta) k}\exp\left(\left(\left(\frac{i}{k}\right)^d-1\right)\sqrt{\lambda} \langle W, x^{\otimes d} \rangle\right)\exp\left(\left(\left(\frac{i}{k}\right)^d-1\right)\lambda+ B_i + \log(1+B_i)\right)\\
        &= o(1).
    \end{align*}

    Now combining the analysis for small overlap and large overlap, we see that with probability $1 - o(1)$, we have
    \begin{align*}
        &\quad \sum_{i=0}^k p_i\cdot \frac{i}{k}\\
        &\ge \frac{(1-\delta)k}{k} \left(1 - \sum_{i=0}^{(1-\delta)k} p_i\right)\\
        &\ge 1 - \delta - o(1).
    \end{align*}
    Since $\delta > 0$ is arbitrary, we see that for $\lambda = 2\log \binom{n-k}{k} + \Delta$, we have
    \begin{align*}
        \text{MMSE} = 1 - \EE\left[\sum_{i=0}^k p_i\cdot \frac{i}{k}\right] \to 0.
    \end{align*}

\end{proof}

Next, we prove the lower edge of the critical window.
\begin{proof}[Proof of the lower edge of Proposition~\ref{prop:critical-window-PCA}]
    Suppose $d \ge 3$, $k \le \min\left\{n^{\frac{1}{2} - \eps}, n^{\frac{d-2}{d+2} - \eps}\right\}$ for some constant $\eps > 0$, and $\Delta > 0$ satisfies
    \begin{align*}
        \Delta = \omega\left(\sqrt{\log \binom{n-k}{k}}\right).
    \end{align*}
    Suppose $\lambda = 2\log \binom{n-k}{k} - \Delta$.

    Fix an arbitrary $\delta > 0$. We will show that $\sum_{i=\delta k}^k p_i \to 0$. Recall that we set $N_i \coloneqq \binom{k}{i}\binom{n-k}{k-i}$.

    We first lower bound $S_0$. Recall that we previously set $b = \sqrt{\frac{\Delta}{\sqrt{\log \binom{n}{k}}}}$. Let
    \begin{align*}
        M = \sum_{\substack{x' \in \Omega_{n,k}:\\ \langle x, x'\rangle = 0}} \allone\{\langle W, x'^{\otimes d}  \rangle \ge \sqrt{\lambda} + b\}.
    \end{align*}
    Then, we have the deterministic inequality $S_0 \ge M \exp(\lambda + \sqrt{\lambda} b)$.

    We will use the second moment method on $M$ to prove a high probability lower bound for $M$, and thus $S_0$. The expectation of $M$ is
    \begin{align*}
        \EE[M] &= N_0 \cdot \Pr\left(\langle W, x'^{\otimes d}\rangle \ge \sqrt{\lambda} + b\right)\\
        &= \exp\left(\log \binom{n-k}{k} + \log \Phi(-\sqrt{\lambda} - b)\right)
        \intertext{where $\Phi$ denotes the cdf of the standard normal distribution.}
        &= \exp\left(\log \binom{n-k}{k} -\frac{(\sqrt{\lambda} + b)^2}{2} - O(\log(\sqrt{\lambda} + b))\right)
        \intertext{using the asymptotics $\Phi(x) \sim \frac{1}{\sqrt{2\pi}|x|}\exp(-x^2/2)$ for $x \to -\infty$,}
        &= \exp\left(\log \binom{n-k}{k} -\frac{\lambda}{2} - b\sqrt{\lambda} - \frac{b^2}{2} - O(\log(\sqrt{\lambda} + b))\right)\\
        &=\exp\left(\log \binom{n-k}{k} -\frac{\lambda}{2} - o(\Delta) - O(\log(\sqrt{\lambda} + b))\right)
        \intertext{since $b\sqrt{\log \binom{n-k}{k}} = o(\Delta)$ and $b^2 = \frac{\Delta}{\sqrt{\log \binom{n}{k}}} = o(\Delta)$. Now we plug in $\lambda = 2\log \binom{n-k}{k} - \Delta$ and get}
        &= \exp\left(\frac{\Delta}{2} - o(\Delta)\right).
    \end{align*}

    Next, we upper bound the second moment of $M$. We have
    \begin{align*}
        &\quad \frac{\EE[M^2]}{\EE[M]^2}\\
        &= \frac{\sum_{\substack{x',x'' \in \Omega_{n,k}:\\ \langle x, x'\rangle = \langle x, x''\rangle = 0}} \Pr\left(\langle W, x'^{\otimes d} \rangle \ge \sqrt{\lambda} + b, \langle W, x''^{\otimes d} \rangle \ge \sqrt{\lambda} + b\right)}{\sum_{\substack{x', x'' \in \Omega_{n,k}:\\ \langle x, x'\rangle = \langle x, x''\rangle = 0}}\Pr\left(\langle W, x'^{\otimes d} \rangle \ge \sqrt{\lambda} + b\right)\Pr\left(\langle W, x''^{\otimes d} \rangle \ge \sqrt{\lambda} + b\right)}\\
        &= \EE_{r \sim \text{Hypergeo}(n-k,k,k)} \frac{\Pr\left( Z \ge \sqrt{\lambda} + b, Z_r \ge \sqrt{\lambda} + b\right)}{\Pr\left(Z \ge \sqrt{\lambda} + b\right)^2}
        \intertext{here $r$ denotes the size of the overlap of two uniform random draws from $\left\{x': \Omega_{n,k}: \langle x, x'\rangle = 0\right\}$, and $Z$ and $Z_r$ are jointly centered Gaussian with variance $1$ and covariance $\left(\frac{r}{k}\right)^d$.}
        &= \sum_{r=0}^k \Pr(r)\frac{\Pr\left( Z \ge \sqrt{\lambda} + b, Z_r \ge \sqrt{\lambda} + b\right)}{\Pr\left(Z \ge \sqrt{\lambda} + b\right)^2}.
    \end{align*}
    To show that $\EE[M^2] \le (1 + o(1))\EE[M^2]$, it is enough to show that
    \begin{align}
        \sum_{r=1}^k \Pr(r)\frac{\Pr\left( Z \ge \sqrt{\lambda} + b, Z_r \ge \sqrt{\lambda} + b\right)}{\Pr\left(Z \ge \sqrt{\lambda} + b\right)^2} = o(1). \label{ineq:second-moment-condition}
    \end{align}

    To show \eqref{ineq:second-moment-condition}, we will break the analysis into three pieces: small overlap, middle overlap, and large overlap. Fix $\eta > 0$ a small constant so that $3\eta^{d-1} \le \frac{c_{\varepsilon}}{2}$.

    \paragraph{\textbf{Small overlap analysis:}}
    Consider $1 \le r \le \eta k$.

    We have by Lemma~\ref{lem:class-size-bound}
    \begin{align*}
        \Pr(r) = \frac{\binom{k}{r}\binom{n-k}{k-r}}{\binom{n-k}{k}} = \frac{N_r}{N_0} \le \frac{1}{r!}\left(\frac{k^2}{n-2k}\right)^r,
    \end{align*}
    and by Lemma~\ref{lem:correlated-gaussian-bound} and $\Pr(Z \ge x) \ge \frac{1}{2x\sqrt{2\pi}} \exp\left(-x^2/2\right)$ for $x \ge 1$,
    \begin{align*}
        &\quad \frac{\Pr\left( Z \ge \sqrt{\lambda} + b, Z_r \ge \sqrt{\lambda} + b\right)}{\Pr\left(Z \ge \sqrt{\lambda} + b\right)^2}\\
        &\le C \frac{1}{\sqrt{1 - \left(\frac{r}{k}\right)^{2d}}} \exp\left(- \frac{(\sqrt{\lambda}+b)^2}{1 + \left(\frac{r}{k}\right)^{d}} + \left(\sqrt{\lambda} + b\right)^2\right)\\
        &\le 2C \exp\left((\sqrt{\lambda} + b)^2 \frac{\left(\frac{r}{k}\right)^{d}}{1 + \left(\frac{r}{k}\right)^{d}}\right).
    \end{align*}
    Thus,
    \begin{align*}
        &\quad \sum_{r=1}^{\eta k} \Pr(r)\frac{\Pr\left( Z \ge \sqrt{\lambda} + b, Z_r \ge \sqrt{\lambda} + b\right)}{\Pr\left(Z \ge \sqrt{\lambda} + b\right)^2}\\
        &\le \sum_{r=1}^{\eta k}\frac{2C}{r!}\left(\frac{k^2}{n-2k}\right)^r  \exp\left((\sqrt{\lambda} + b)^2 \frac{\left(\frac{r}{k}\right)^{d}}{1 + \left(\frac{r}{k}\right)^{d}}\right)\\
        &= 2C\sum_{r=1}^{\eta k} \exp\left((\sqrt{\lambda} + b)^2 \frac{\left(\frac{r}{k}\right)^{d}}{1 + \left(\frac{r}{k}\right)^{d}} + r\log\left(\frac{k^2}{n-2k}\right) - r\right)
        \intertext{since $\log\left(\frac{k^2}{n-2k}\right) \le -c_{\varepsilon}\log(n/k)$ and $(\sqrt{\lambda} + b)^2 \le 2\log \binom{n-k}{k} + o(\Delta) \le 3k\log (n/k)$, we have}
        &\le 2C\sum_{r=1}^{\eta k} \exp\left(3r\log(n/k) \frac{\left(\frac{r}{k}\right)^{d-1}}{1 + \left(\frac{r}{k}\right)^{d}} - c_{\varepsilon} r\log(n/k) - r\right)\\
        &\le 2C\sum_{r=1}^{\eta k} \exp\left( - \frac{c_{\varepsilon}}{2} r\log(n/k) - r\right)\\
        &= o(1).
    \end{align*}

    \paragraph{\textbf{Middle overlap analysis:}} Consider $\eta k \le r \le (1 -\eta)k$. We have 
    \begin{align*}
        \Pr(r) = \frac{\binom{k}{r}\binom{n-k}{k-r}}{\binom{n-k}{k}} \le 2^k \left(\frac{e(n-k)}{k-r}\right)^{k-r} \left(\frac{k}{n-k}\right)^k \le C_{\eta}^k \left(\frac{k}{n-k}\right)^{r},
    \end{align*}
    for some constant $C_{\eta} > 0$, and the same bound as in the small overlap analysis,
    \begin{align*}
        &\quad \frac{\Pr\left( Z \ge \sqrt{\lambda} + b, Z_r \ge \sqrt{\lambda} + b\right)}{\Pr\left(Z \ge \sqrt{\lambda} + b\right)^2} \le 2C \exp\left((\sqrt{\lambda} + b)^2 \frac{\left(\frac{r}{k}\right)^{d}}{1 + \left(\frac{r}{k}\right)^{d}}\right).
    \end{align*}
    Thus,
    \begin{align*}
        &\quad \sum_{r=\eta k}^{(1-\eta) k} \Pr(r)\frac{\Pr\left( Z \ge \sqrt{\lambda} + b, Z_r \ge \sqrt{\lambda} + b\right)}{\Pr\left(Z \ge \sqrt{\lambda} + b\right)^2}\\
        &\le \sum_{r=\eta k}^{(1-\eta) k}2 C \cdot C_{\eta}^k \left(\frac{k}{n-k}\right)^{r}  \exp\left((\sqrt{\lambda} + b)^2 \frac{\left(\frac{r}{k}\right)^{d}}{1 + \left(\frac{r}{k}\right)^{d}}\right)\\
        &\le 2C \sum_{r=\eta k}^{(1-\eta) k}  \exp\left((\sqrt{\lambda} + b)^2 \frac{\left(\frac{r}{k}\right)^{d}}{1 + \left(\frac{r}{k}\right)^{d}} + O(k) - r\log \left(\frac{n}{k}\right)\right)
        \intertext{Since $b \sqrt{\log \binom{n-k}{k}} = o(\Delta)$ and $b^2 = o(\Delta)$, we now use $(\sqrt{\lambda} + b)^2 \le (\sqrt{2\log \binom{n-k}{k}- \Delta} + b)^2 \le (\sqrt{2\log \binom{n-k}{k}}  + b)^2  \le 2\log \binom{n-k}{k} + o(\Delta)$ to get}
        &\le 2C \sum_{r=\eta k}^{(1-\eta) k}  \exp\left(\left(2\log \binom{n-k}{k} + o(\Delta)\right) \frac{\left(\frac{r}{k}\right)^{d}}{1 + \left(\frac{r}{k}\right)^{d}} + O(k) - r\log \left(\frac{n}{k}\right)\right)\\
        &\le 2C \sum_{r=\eta k}^{(1-\eta) k}  \exp\left(k \log \left(\frac{n}{k}\right)\left(\frac{2\left(\frac{r}{k}\right)^{d}}{1 + \left(\frac{r}{k}\right)^{d}}- \frac{r}{k}+ o(1)\right)\right)
        \intertext{Now observe that since $d \ge 3 \ge 2$, there exists a constant $\gamma_\eta > 0$ such that $f(x) = \frac{2x^d}{1+x^d} - x \le - \gamma_{\eta}$ for all $\eta \le x \le 1-\eta$. Thus,}
        &\le 2C \sum_{r=\eta k}^{(1-\eta) k}  \exp\left(-\frac{\gamma_\eta}{2}k \log \left(\frac{n}{k}\right)\right)\\
        &= o(1).
    \end{align*}

    \paragraph{\textbf{Large overlap analysis:}} Consider $(1-\eta)k \le r \le k$. We have
    \begin{align*}
        \Pr(r) &= \frac{\binom{k}{r}\binom{n-k}{k-r}}{\binom{n-k}{k}}\\
        &=\frac{\binom{k}{k-r}\binom{n-k}{k-r}}{\binom{n-k}{k}}\\
        &\le \frac{1}{((k-r)!)^2} \cdot \frac{k^{k-r}(n-k)^{k-r}}{\binom{n-k}{k}},
    \end{align*}
    and we have the same bound as before
    \begin{align*}
        &\quad \frac{\Pr\left( Z \ge \sqrt{\lambda} + b, Z_r \ge \sqrt{\lambda} + b\right)}{\Pr\left(Z \ge \sqrt{\lambda} + b\right)^2} \le 2C \exp\left((\sqrt{\lambda} + b)^2 \frac{\left(\frac{r}{k}\right)^{d}}{1 + \left(\frac{r}{k}\right)^{d}}\right).
    \end{align*}

    Thus,
    \begin{align*}
        &\quad \sum_{r=(1-\eta) k}^{k} \Pr(r)\frac{\Pr\left( Z \ge \sqrt{\lambda} + b, Z_r \ge \sqrt{\lambda} + b\right)}{\Pr\left(Z \ge \sqrt{\lambda} + b\right)^2}\\
        &\le 2C\sum_{r=(1-\eta) k}^{k} \frac{1}{((k-r)!)^2} \cdot \frac{k^{k-r}(n-k)^{k-r}}{\binom{n-k}{k}} \exp\left((\sqrt{\lambda} + b)^2 \frac{\left(\frac{r}{k}\right)^{d}}{1 + \left(\frac{r}{k}\right)^{d}}\right)\\
        &\le 2C \sum_{r=(1-\eta) k}^{k} \exp\left((\sqrt{\lambda} + b)^2 \frac{\left(\frac{r}{k}\right)^{d}}{1 + \left(\frac{r}{k}\right)^{d}} - \log \binom{n-k}{k} + (k-r)\log(k(n-k)) - O(k-r)\right)
        \intertext{Since $b \sqrt{\log \binom{n-k}{k}} = o(\Delta)$ and $b^2 = o(\Delta)$, we now use $(\sqrt{\lambda} + b)^2 \le (\sqrt{2\log \binom{n-k}{k}- \Delta} + b)^2   \le 2\log \binom{n-k}{k} - \Delta + o(\Delta)$ to get}
        &\le 2C \sum_{r=(1-\eta) k}^{k} \exp\Bigg(\left(2\log \binom{n-k}{k} - \Delta + o(\Delta)\right) \frac{\left(\frac{r}{k}\right)^{d}}{1 + \left(\frac{r}{k}\right)^{d}} - \log \binom{n-k}{k} \\
        &\quad + (k-r)\log(k(n-k)) - O(k-r)\Bigg)
        \intertext{Now let $s = k -r$. We have $0 \le s \le \eta k$. Then, $(r/k)^d = 1 - d\frac{s}{k} + O(s^2/k^2) $ and we have for some constant $c_\eta > 0$ that}
        &\le 2C \sum_{r=(1-\eta) k}^{k} \exp\Bigg(\left(2\log \binom{n-k}{k} - \Delta + o(\Delta)\right) \left(\frac{1}{2} - \left(\frac{1}{2} - c_\eta \right)d\cdot \frac{s}{k}\right) - \log \binom{n-k}{k}\\
        &\quad + s\log(k(n-k)) - O(s)\Bigg)\\
        &\le  2C \sum_{r=(1-\eta) k}^{k} \exp\left(-\frac{\Delta}{4} + o(\Delta) - sd (1 - 2 c_{\eta}) \log \left(\frac{n-k}{k}\right) + s\log(k(n-k)) - O(s)\right)\\
        &= 2C \sum_{r=(1-\eta) k}^{k} \exp\left(-\frac{\Delta}{4} + o(\Delta) - s\left(\log\left(\frac{(n-k)^{d(1-2c_\eta)-1}}{k^{d(1-2c_\eta)+1}}\right)+O(1)\right)\right)\\
        &= o(1),
    \end{align*}
    where we used that $k \le n^{\frac{1}{2} - \varepsilon} \le n^{\frac{d-1}{d+1}-\varepsilon}$.

    Combining the analysis for all overlaps, we conclude that $\EE[M^2] \le (1 + o(1))\EE[M^2]$. Thus, with probability $1 - o(1)$, we have $M \ge (1-o(1))\EE[M]$ and
    \begin{align*}
        S_0 &\ge (1-o(1))\EE[M] \exp(\lambda + \sqrt{\lambda} b)\\
        &\ge (1-o(1))\exp\left(\frac{\Delta}{2} - o(\Delta) + \lambda + \sqrt{\lambda} b\right)\\
        &\ge (1-o(1))\exp\left(\frac{\Delta}{2} - o(\Delta) + \lambda\right).
    \end{align*}

    On the other hand, by Lemma~\ref{lem:log-sum-exp-tail} as discussed in the upper edge analysis, with probability at least $1 - \frac{2}{k}$, we have that for all $\delta k \le i \le k$,
    \begin{align*}
        S_i &\le \exp\left(\lambda \left(\frac{i}{k}\right)^d + \left(\frac{i}{k}\right)^d\sqrt{\lambda} \langle W, x^{\otimes d} \rangle\right)\exp\left(B_i + \log(1+B_i)\right),
    \end{align*}
    where $B_i = \sqrt{2\lambda\left(1 - \left(i/k\right)^{2d}\right)\left(2\log k + \log N_i\right)}$. We need to make sure $B_i$ satisfies $B_i \le \lambda \left(1 - (i/k)^{2d}\right)$ as Lemma~\ref{lem:log-sum-exp-tail} demands. We may easily verify that for $\delta k \le i \le (1 - \delta)k$
    \begin{align*}
        B_i &= \sqrt{2\lambda\left(1 - \left(i/k\right)^{2d}\right)\left(2\log k + \log N_i\right)}\\
        &= \sqrt{2\lambda\left(1 - \left(i/k\right)^{2d}\right)\left(2\log k + \log \binom{k}{i} + \log \binom{n-k}{k-i}\right)}\\
        &\le \sqrt{2\lambda\left(1 - \left(i/k\right)^{2d}\right)\left(2\log k + k + (k-i)\log (e(n-k)/(k-i)) \right)}\\
        &\le \sqrt{2\lambda\left(1 - \left(i/k\right)^{2d}\right) \left((k-i)\log((n-k)/k) + O(k)\right)}\\
        &\le \sqrt{2\lambda\left(1 - \left(i/k\right)^{2d}\right) \left(\left(1 - \frac{i}{k}\right)\log \binom{n-k}{k} + O(k)\right)}\\
        &\le \lambda\left(1 - \left(i/k\right)^{2d}\right) \sqrt{\frac{\left(2\log \binom{n-k}{k}\right)\left(1- (i/k) + o(1)\right)}{\lambda\left(1 - \left(i/k\right)^{2d}\right)}}\\
        &\le \lambda\left(1 - \left(i/k\right)^{2d}\right),
    \end{align*}
    since $\lambda = 2\log \binom{n-k}{k} - \Delta = (2-o(1))\log\binom{n-k}{k}$ and there exists some constant $c_\delta > 0$ such that $\frac{1 - (i/k) + o(1)}{1 - (i/k)^{2d}} \le 1 - c_{\delta}$ for $\delta k \le i \le (1-\delta)k$.

    Now for $(1 - \delta)k \le i \le k-1$, we may check
    \begin{align*}
        B_i &= \sqrt{2\lambda\left(1 - \left(i/k\right)^{2d}\right)\left(2\log k + \log N_i\right)}\\
        &= \sqrt{2\lambda\left(1 - \left(i/k\right)^{2d}\right)\left(2\log k + \log \binom{k}{k-i} + \log \binom{n-k}{k-i}\right)}\\
        &\le \sqrt{2\lambda\left(1 - \left(i/k\right)^{2d}\right)\left(2\log k + (k-i)\log (k(n-k)) + O(k-i) \right)}
        \intertext{Now since $k \le \min\{n^{\frac{d-2}{d+2}-\varepsilon}, n^{\frac{1}{2} - \varepsilon}\}$, we have $\log(k) \le c_{\varepsilon}\log((n-k)/k)$ for some constant $c_{\varepsilon} < \min\{\frac{d-2}{4},1\}$. Thus,}
        &\le \sqrt{2\lambda\left(1 - \left(i/k\right)^{2d}\right) \left((k-i)\log((n-k)/k) + 2(k-i)\log k + O(k-i)\right)}\\
        &\le \sqrt{2\lambda\left(1 - \left(i/k\right)^{2d}\right) \left((k-i)\log((n-k)/k) + 2c_{\varepsilon}(k-i)\log ((n-k)/k) + O(k-i)\right)}\\
        &\le \sqrt{2\lambda\left(1 - \left(i/k\right)^{2d}\right) \left(\left(1 - \frac{i}{k}\right)(1 + 2c_{\varepsilon})\log \binom{n-k}{k} + O(k\log k)\right)}\\
        &\le \lambda\left(1 - \left(i/k\right)^{2d}\right) \sqrt{\frac{\left(2\log \binom{n-k}{k}\right)\left((1- (i/k))(1 + 2c_{\varepsilon}) + o(1)\right)}{\lambda\left(1 - \left(i/k\right)^{2d}\right)}}\\
        &\le \lambda\left(1 - \left(i/k\right)^{2d}\right),
    \end{align*}
    where we used that $\lambda = 2\log \binom{n-k}{k} - \Delta = (2-o(1))\log\binom{n-k}{k}$ and $1 - (i/k)^{2d} = 2d\frac{k-i}{k} + O\left(\frac{(k-i)^2}{k^2}\right) \ge 2\cdot \left( (1- (i/k))(1 + 2c_{\varepsilon}) + o(1)\right)$, since $d \ge 3$ and $c_{\varepsilon} < 1$.

    Finally, when $i = k$, there is nothing to check as $B_k = 0$. Thus, we have shown that with probability at least $1 - \frac{2}{k}$, 
    \begin{align*}
        \sum_{i=\delta k}^k S_i \le \sum_{i=\delta k}^k\exp\left(\lambda \left(\frac{i}{k}\right)^d + \left(\frac{i}{k}\right)^d\sqrt{\lambda} \langle W, x^{\otimes d} \rangle\right)\exp\left(B_i + \log(1+B_i)\right),
    \end{align*}
    where $B_i = \sqrt{2\lambda\left(1 - \left(i/k\right)^{2d}\right)\left(2\log k + \log N_i\right)}$. Now, condition on the event $E_1' = \{\{\langle W, x^{\otimes d} \rangle \le b\}\}$. Recall that $b = \sqrt{\frac{\Delta}{\sqrt{\log \binom{n-k}{k}}}}$ and $E_1'$ takes place with probability $1 - o(1)$. Thus, with probability at least $1 - \frac{k}{2} - o(1) = 1 - o(1)$,
    \begin{align*}
        &\quad \sum_{i=\delta k}^k S_i\\
        &\le \sum_{i=\delta k}^k\exp\left(\lambda \left(\frac{i}{k}\right)^d + \left(\frac{i}{k}\right)^d\sqrt{\lambda} \langle W, x^{\otimes d} \rangle\right)\exp\left(B_i + \log(1+B_i)\right)\\
        &\le \sum_{i=\delta k}^k\exp\left(\lambda \left(\frac{i}{k}\right)^d + \left(\frac{i}{k}\right)^d b\sqrt{\lambda} \right)\exp\left(B_i + \log(1+B_i)\right)
        \intertext{Now we use that $b\sqrt{\lambda} = o(\Delta)$ and that $B_i = \sqrt{2\lambda\left(1 - \left(i/k\right)^{2d}\right)\left(2\log k + \log N_i\right)}$ to get}
        &\le \sum_{i=\delta k}^k\exp\left((\lambda + o(\Delta)) \left(\frac{i}{k}\right)^d + \sqrt{2\lambda\left(1 - \left(i/k\right)^{2d}\right)\left(2\log k + \log N_i\right)} + o(\Delta)\right).
    \end{align*}
    Recall that with probability at least $1 - o(1)$, 
    \[S_0 \ge (1-o(1))\exp\left(\frac{\Delta}{2} - o(\Delta) + \lambda\right).\]
    Thus, with probability $1 - o(1)$, we have
    \begin{align*}
        &\quad \frac{1}{S_0}\sum_{i=\delta k}^k S_i\\
        &\le (1+o(1))\sum_{i=\delta k}^k\exp\left(- \frac{\Delta}{2} - \lambda+(\lambda + o(\Delta)) \left(\frac{i}{k}\right)^d + \sqrt{2\lambda\left(1 - \left(i/k\right)^{2d}\right)\left(2\log k + \log N_i\right)} + o(\Delta)\right)\\
        &\le (1+o(1))\sum_{i=\delta k}^k\exp\left(- \frac{\Delta}{2} - \lambda+\lambda \left(\frac{i}{k}\right)^d + \sqrt{2\lambda\left(1 - \left(i/k\right)^{2d}\right)\left(2\log k + \log N_i\right)} + o(\Delta)\right)\\
        &\le (1+o(1))\sum_{i=\delta k}^{(1-\delta)k}\exp\left(- \frac{\Delta}{2} - \lambda+\lambda \left(\frac{i}{k}\right)^d + \sqrt{2\lambda\left(1 - \left(i/k\right)^{2d}\right)\left(2\log k + \log N_i\right)} + o(\Delta)\right)\\
        &\quad + (1+o(1))\sum_{i=(1-\delta) k}^k\exp\left(- \frac{\Delta}{2} - \lambda+\lambda \left(\frac{i}{k}\right)^d + \sqrt{2\lambda\left(1 - \left(i/k\right)^{2d}\right)\left(2\log k + \log N_i\right)} + o(\Delta)\right)
        \intertext{Now we reuse the previous bounds that $2\log k + \log N_i \le \log \binom{n-k}{k}\left(1- (i/k) + o(1)\right)$ for all $\delta k \le i \le (1-\delta)k$ and that $2\log k + \log N_i \le \log \binom{n-k}{k}\left((1- (i/k))(1 + 2c_{\varepsilon}) + o(1)\right)$ for all $(1-\delta)k \le i \le k$ where $0 < c_{\varepsilon} < \min\{\frac{d-2}{4} ,1\}$ is a constant, and get}
        &\le \exp(o(\Delta))\sum_{i=\delta k}^{(1-\delta)k}\exp\left(- \frac{\Delta}{2} - \lambda+\lambda \left(\frac{i}{k}\right)^d + \sqrt{2\lambda\left(1 - \left(\frac{i}{k}\right)^{2d}\right)\log \binom{n-k}{k}\left(1- \frac{i}{k} + o(1)\right)}\right)\\
        &\quad + \exp(o(\Delta))\sum_{i=(1-\delta) k}^k\exp\Bigg(- \frac{\Delta}{2} - \lambda+\lambda \left(\frac{i}{k}\right)^d\\
        &\quad + \sqrt{2\lambda\left(1 - \left(\frac{i}{k}\right)^{2d}\right)\log \binom{n-k}{k}\left(\left(1- \frac{i}{k}\right)(1 + 2c_{\varepsilon}) + o(1)\right)}\Bigg)
        \intertext{Now plug in $\lambda = 2\log \binom{n-k}{k} - \Delta$. We get}
        &\le \exp(o(\Delta))\sum_{i=\delta k}^{(1-\delta)k}\exp\left(\Delta\left(-\frac{1}{2} + \frac{1}{2}\sqrt{\left(1 - \left(\frac{i}{k}\right)^{2d}\right)\left(1 - \frac{i}{k} + o(1)\right)}\right)\right)\\
        &\quad \cdot \exp\left(\left(2\log \binom{n-k}{k} - \Delta\right) \left(-1 +\left(\frac{i}{k}\right)^d +\sqrt{\left(1 - \left(\frac{i}{k}\right)^{2d}\right)\left(1 - \frac{i}{k} + o(1)\right)}\right) \right)\\
        &\quad + \exp(o(\Delta))\sum_{i=(1-\delta) k}^{k}\exp\left(\Delta\left(-\frac{1}{2} + \frac{1}{2}\sqrt{\left(1 - \left(\frac{i}{k}\right)^{2d}\right)\left(\left(1 - \frac{i}{k}\right)(1 + 2c_{\varepsilon}) + o(1)\right)}\right)\right)\\
        &\quad \cdot \exp\left(\left(2\log \binom{n-k}{k}-\Delta\right) \left(-1 +\left(\frac{i}{k}\right)^d +\sqrt{\left(1 - \left(\frac{i}{k}\right)^{2d}\right)\left(\left(1 - \frac{i}{k}\right)(1+2c_{\varepsilon}) + o(1)\right)}\right) \right)
        \intertext{Now observe that there exists a constant $\gamma_\delta > 0$ such that $f(x) = -1 + x^d + \sqrt{\left(1-x^{2d}\right)(1-x)} \le -\gamma_\delta$ for all $\delta \le x \le 1 - \delta$, and that for sufficiently small $\delta > 0$, there exists $\tau_\gamma > 0$ such that $g(1 - y) = -1 + (1-y)^d + \sqrt{\left(1- (1-y)^{2d}\right)y(1+2c_{\varepsilon})} \le -dy + \sqrt{2d(1+2c_{\varepsilon})y^2} + O(y^2) \le - \tau_\delta y$ for all $0\le y \le \delta$. Here we used that $c_{\varepsilon} \le \frac{d-2}{4}$ so that $\sqrt{2d(1+2c_{\varepsilon})} < d$. Thus,}
        &\le \exp(o(\Delta))\sum_{i=\delta k}^{(1-\delta)k} \exp\left(-\gamma_\delta\left(2\log \binom{n-k}{k} - \Delta\right)  \right)\\
        &\quad + \exp(o(\Delta))\sum_{i=(1-\delta) k}^{k} \exp\left(-\tau_\delta \frac{k-i}{k}\left(2\log \binom{n-k}{k}-\Delta\right)  \right)\\
        &= o(1).
    \end{align*}

    Thus, with probability $1-o(1)$, we have 
    \begin{align*}
        &\quad \sum_{i=\delta k}^k p_i \le \frac{1}{S_0} \sum_{i=\delta k}^k S_i = o(1).
    \end{align*}
    Since $\delta > 0$ is arbitrary, we see that for $\lambda = 2\log \binom{n-k}{k} - \Delta$, we have
    \begin{align*}
        \text{MMSE} = 1 - \EE\left[\sum_{i=0}^k p_i\cdot \frac{i}{k}\right] \ge (1 - \delta) - \EE\left[\sum_{i=\delta k}^k p_i\cdot \frac{i}{k}\right]  \ge 1 - \delta - o(1),
    \end{align*}
    and thus $\text{MMSE} \to 1$.

\end{proof}

}

\end{document}